\documentclass{article}
\usepackage{
amsmath,
amsfonts,
latexsym, 
amssymb
}

\usepackage[non-sorted-cites]{amsrefs}

\usepackage{subcaption}

\usepackage{tikz}
\usetikzlibrary{arrows}  

\usepackage{xspace}

\usepackage{titlesec}
\usepackage{titletoc}\usepackage[title,titletoc]{appendix}

\usepackage{bm}

\usepackage{tocloft}

\usepackage{cite}

\usepackage{graphicx}
\newcommand{\1}{\mathds{1}}
\usepackage{enumerate}
\usepackage{mathrsfs}
\usepackage{wrapfig}

\usepackage{color}

\numberwithin{equation}{section}

\newcommand{\pt}{\partial}
\newcommand{\rd}{{\rm d}}

\newcommand{\bR}{{\mathbb R}}

\newcommand{\ba}{{\bf{a}}}

\newcommand{\bu}{{\bf{u}}}
\newcommand{\bv}{{\bf{v}}}
\newcommand{\bw}{{\bf{w}}}

\newcommand{\al}{\alpha}

\newcommand{\be}{\begin{equation}}
\newcommand{\ee}{\end{equation}}

\newcommand{\e}{{\varepsilon}}

\newcommand{\la}{\lambda}

\newcommand{\om}{{\omega}}

\usepackage{amsmath} 
\usepackage{amssymb}
\usepackage{amsthm}

\def\RR{{\mathbb R}}

\renewcommand{\b}[1]{\bm{\mathrm{#1}}} 

\renewcommand{\cal}{\mathcal}

\newcommand{\wh}{\widehat}
\newcommand{\wt}{\widetilde}

\newcommand{\bb}{{\bf{b}}}
\newcommand{\bx}{{\bf{x}}}
\newcommand{\ii}{\mathrm{i}} 

\newcommand{\col}{\mathrel{\mathop:}}

\renewcommand{\epsilon}{\varepsilon}
\renewcommand{\leq}{\leqslant}
\renewcommand{\geq}{\geqslant}

\renewcommand{\le}{\leq}
\renewcommand{\ge}{\geq}

\renewcommand{\P}{\mathbb{P}}
\newcommand{\E}{\mathbb{E}}
\newcommand{\R}{\mathbb{R}}
\newcommand{\C}{\mathbb{C}}
\newcommand{\N}{\mathbb{N}}
\newcommand{\Z}{\mathbb{Z}}

\newcommand{\hb}[1]{\bigl\{{#1}\bigr\}}

\newcommand{\absb}[1]{\bigl\lvert #1 \bigr\rvert}

\DeclareMathOperator{\tr}{Tr}

\DeclareMathOperator{\re}{Re}
\DeclareMathOperator{\im}{Im}

\DeclareMathOperator{\OO}{O}

\theoremstyle{plain} 
\newtheorem{theorem}{Theorem}[section]
\newtheorem*{theorem*}{Theorem}
\newtheorem{lemma}[theorem]{Lemma}
\newtheorem*{lemma*}{Lemma}
\newtheorem{corollary}[theorem]{Corollary}
\newtheorem*{corollary*}{Corollary}
\newtheorem{proposition}[theorem]{Proposition}
\newtheorem*{proposition*}{Proposition}

\newtheorem{definition}[theorem]{Definition}
\newtheorem*{definition*}{Definition}

\newtheorem*{example*}{Example}

\newtheorem*{remark*}{Remark}
\newtheorem*{remarks*}{Remarks}

\newtheorem{assumption}{Assumption}
\newtheorem*{assumption*}{Assumption}

\newcommand{\g}{{\sigma}}

\makeatletter
\renewcommand{\subsection}{\@startsection
{subsection}
{2}
{0mm}
{-\baselineskip}
{0 \baselineskip}
{\normalfont\itshape}} 
\makeatother

\usepackage{dsfont}
\usepackage{stmaryrd}

\newcommand{\nc}{\normalcolor}

\def\bR{{\mathbb R}}

\newcommand{\bq}{{\bf q}}

\renewcommand{\b}[1]{\boldsymbol{\mathrm{#1}}} 

\def\@empty{}

\def\author#1{\par
    {\centering{\authorfont#1}\par\vspace*{0.05in}}
}

\def\titlefont{\fontsize{13}{15}\bfseries\boldmath\selectfont\centering{}}
\def\authorfont{\fontsize{13}{15}}
\def\abstractfont{\fontsize{8}{10}}

\let\affiliationfont\rhfont

\def\address#1{\par
    {\centering{\affiliationfont#1\par}}\par\vspace*{11pt}
}

\def\keywords#1{\par
    \vspace*{8pt}
    {\authorfont{\leftskip18pt\rightskip\leftskip
    \noindent{\it\small{Keywords}}\/:\ #1\par}}\vskip-12pt}

\def\body{
\setcounter{footnote}{0}
\def\thefootnote{\alph{footnote}}
\def\@makefnmark{{$^{\rm \@thefnmark}$}}
}

\def\title#1{
    \thispagestyle{plain}
    \vspace*{-14pt}
    \vskip 79pt
    {\centering{\titlefont #1\par}}%
    \vskip 1em
}

\renewenvironment{abstract}{\par%
    \vspace*{6pt}\noindent 
    \abstractfont
    \noindent\leftskip18pt\rightskip18pt
}{%
  \par}

\renewcommand{\b}[1]{\boldsymbol{\mathrm{#1}}} 

\usepackage{tocloft}

\makeatletter
\renewcommand{\section}{\@startsection
{section}
{1}
{0mm}
{-2\baselineskip}
{1\baselineskip}
{\normalfont\large\scshape\centering}} 
\makeatother

\begin{document}

~\vspace{-0.6cm}

\title{Universality for a class of random band matrices}

\vspace{1cm}
\noindent\begin{minipage}[b]{0.5\textwidth}

 \author{P. Bourgade}

\address{New York University, Courant Institute\\
bourgade@cims.nyu.edu}
 \end{minipage}
\begin{minipage}[b]{0.5\textwidth}

 \author{L. Erd{\H o}s}

\address{Institute of Science and Technology Austria\\
lerdos@ist.ac.at}

 \end{minipage}

\noindent\begin{minipage}[b]{0.5\textwidth}

 \author{H.-T. Yau}

\address{Harvard University\\
htyau@math.harvard.edu}
 \end{minipage}
\begin{minipage}[b]{0.5\textwidth}

 \author{J. Yin}

\address{University of Wisconsin, Madison\\
jyin@math.wisc.edu}

 \end{minipage}

~\vspace{0.3cm}

\begin{abstract}
We prove  the  universality for the eigenvalue gap statistics  in the bulk of the spectrum
for band matrices, in the  regime where the band width is comparable with the
dimension of the matrix, $W\sim N$.  All previous results concerning universality of non-Gaussian 
random matrices  are for mean-field models.
By relying on a new mean-field reduction technique, we deduce universality 
from quantum unique ergodicity for band matrices.
\end{abstract}



{\let\thefootnote\relax\footnote{\noindent 
The work of P. B. is partially supported by NSF grants DMS-1208859 and DMS-1513587. 
The work of L. E. is partially supported  by ERC Advanced Grant, RANMAT 338804.
The work of H.-T. Y. is partially supported by the NSF grant DMS-1307444 and the Simons investigator fellowship.
The work of J. Y. is partially supported by NSF Grant DMS-1207961. The major part of
this research was conducted when all authors were visiting IAS and were also supported by
the NSF Grant DMS-1128255.}}
\setcounter{footnote}{0}


\section{Introduction}

\subsection{Previous studies of Wigner and band matrices.}\ There has been tremendous progress on the  universality of non-invariant random matrices over the past decade.  
The basic model for such matrices,  the Wigner 
ensemble, consists of $N\times N$ real  symmetric or  complex  Hermitian 
matrices $H = (H_{ij}\nc )_{1\leq i,j\leq N}$ whose matrix entries are identically
distributed centered \nc random variables that are independent up to the symmetry constraint
 $H=H^*$. The fundamental conjecture regarding the universality of the Wigner ensemble, the Wigner-Dyson-Mehta conjecture, 
states  that the eigenvalue  gap distribution is universal in the sense that it depends only
on the symmetry class of the matrix, but is otherwise
independent of the details of the distribution of the matrix entries. 
This conjecture has recently been established for all symmetry classes in a series of works \cites{ErdSchYau2011,
ErdYau2012singlegap,BouErdYauYin2015} (see \cite{ErdPecRamSchYau2010, Joh2001,TaoVu2011}
 for the Hermitian class of Wigner matrices). 
The  approach initiated in \cite{ErdPecRamSchYau2010, ErdSchYau2011} to prove universality
consists of three steps: (i) establish a local semicircle law for the density of eigenvalues (or more generally estimates on the Green functions); 
(ii) prove universality of Gaussian divisible ensembles, i.e., Wigner matrices with a small Gaussian
component,  by analyzing the convergence of Dyson Brownian motion to local equilibrium; (iii) remove the small Gaussian
component by comparing Green functions of Wigner ensembles 
with those of Gaussian divisible ones.
For an overview of universality results for Wigner matrices and this three-step
strategy, see \cite{ErdYau2012}.

Wigner in fact  predicted that universality should hold
for any  large  quantum system, described by a 
 Hamiltonian $H$, of sufficient complexity. 
One prominent example
 where random matrix statistics are expected
to hold is the random Schr\"odinger
operator in the delocalized regime. 
The random Schr\"odinger  operator describes a  system  
 with spatial structure, whereas Wigner matrices are mean-field models. 
Unfortunately, there has been virtually no progress in establishing the universality for the random Schr\"odinger  operator
 in the delocalized regime.  One prominent  model interpolating between the  Wigner matrices and the random Schr\"odinger  operator  
 is the  \emph{ random band matrix}.  In this model the physical state space,
which labels the matrix elements, is equipped with a distance. Band matrices are \nc
 characterized
by the property that $H_{ij}$   becomes negligible if $\mbox{dist}(i,j)$ 
exceeds
a certain parameter, $W$, called the \emph{band width}. 
A fundamental conjecture \cite{fy} states  that the local spectral statistics of a band matrix  $H$ are governed by random matrix statistics for large $W$ and by Poisson statistics for small $W$. The transition 
is conjectured to be sharp \cite{fy, Spe} for the  band matrices 
 in one spatial dimension  around the critical value $W = \sqrt{N}$.
In other words, if $W \gg \sqrt{N}$, we expect the universality results of \cites{ErdPecRamSchYau2010, ErdSchYau2011,
ErdYau2012singlegap,BouErdYauYin2015}  to hold. Furthermore, 
the eigenvectors of $H$ are expected to be completely delocalized  in this range. For $W \ll \sqrt{N}$, one expects that 
the eigenvectors  are  exponentially localized.   This is the analogue of
the celebrated Anderson metal-insulator transition for random band matrices. 
The only rigorous work indicating the $\sqrt{N}$ threshold
concerns  the second mixed moments of the characteristic
polynomial for a special class of Gaussian band matrices \cite{Sch1, Sch2}
\nc

The localization length  for band matrices  in one spatial dimension
was recently investigated in numerous works.
For general distribution of the matrix entries,  
eigenstates were proved to be localized \cite{Sch2009} for $W\ll N^{1/8}$,  
and delocalization  of most eigenvectors \nc in a certain averaged sense holds for
 $W\gg N^{6/7}$ \cite{ErdKno2013}, improved to \nc $W\gg N^{4/5}$  \cite{ErdKnoYauYin2013}.
The Green's function $(H-z)^{-1}$  was controlled
 down to the scale $\im z\gg W^{-1}$ in \cite{ErdYauYin2012Univ}, implying
a lower bound of order $W$ for the localization length of all eigenvectors. \nc
When the entries are Gaussian with some specific covariance profiles, supersymmetry techniques
 are applicable to obtain stronger results. 
This approach  
has first been developed by physicists (see \cite{Efe1997} for an overview);
 the  rigorous analysis was initiated by Spencer (see \cite{Spe} for an overview),
with  an accurate estimate on the expected density of states on arbitrarily short scales \nc for a three-dimensional band matrix ensemble 
 in \cite{DisPinSpe2002}.  More recent works  include 
  universality for $W=\Omega(N)$ \cite{Sch2014}, and  the control of the Green's function
  down to the optimal scale $\im z\gg N^{-1}$, hence  delocalization in a strong sense for all
  eigenvectors,  when 
  $W\gg N^{6/7}$ \cite{BaoErd2015} with first four moments matching the Gaussian ones
(both results require a block structure and hold in part of the bulk spectrum).  These rigorous results
based on the supersymmetric method so far assumed the complex hermitian condition.
Our work is about  statistics in the bulk of the spectrum for both real symmetric and complex hermitian band matrices, \nc
but we note that for universality at the  spectral  edge, much more is known \cite{Sod2010}: extreme eigenvalues follow the Tracy-Widom law for $W\gg N^{5/6}$, an essentially optimal condition.

\subsection{Difficulties and new ideas for general non mean-field models.\ } In trying to use the above three-steps strategy for band matrices, let us first mention difficulties related to step (i), the local law.
The Wigner-Dyson-Gaudin-Mehta conjecture was originally stated for 
Wigner matrices, but the methods of \cite{ErdPecRamSchYau2010, ErdSchYau2011} also apply to certain ensembles 
with independent but not identically distributed entries, which however retain the mean-field character of Wigner matrices.
 For generalized Wigner matrices with entries having varying variances,
but still following the semicircle law,  \nc
 see \cite{ErdYauYin2012Rig}, and more 
generally \cite{AjaErdKru2015}, where even the global density differs from the semicircle law. 
 In particular, the local law up to the smallest scale $N^{-1}$
can be obtained under the assumption that
the entries of $H$ satisfy 
\be\label{genwig}
s_{ij}:=\E(|H_{ij}|^2) \;\leq\; \frac{C}{N}
\ee
for some positive constant  $C$.   In this paper, we  assume that $\sum_i s_{ij}=1$;  this
 normalization guarantees 
that  the spectrum is supported on $[-2, 2]$. \nc
However, if the matrix entries vanish outside the band $|i-j|\lesssim W\ll N$,  (\ref{genwig})  cannot \nc hold and 
the best  known 
 local semicircle law in this context \cite{ErdKnoYauYin2013} gives  estimates  only up to scale $W^{-1}$,
while the optimal scale would be $N^{-1}$,
 comparable with the eigenvalue spacing. 
Hence for $W = N^{1-\delta}$, $\delta>0$, 
the optimal local law is not known  up to the smallest scale,
which is a key source of difficulty for proving the delocalization
of the band matrices. 
In this article, as  $W = c N$ for some fixed small constant $c$, the local law holds up to the optimal scale.

While  step (i) for the three-step  strategy holds in this paper, 
  steps (ii) and (iii) present a key hurdle to prove 
the universality.   To explain this difficulty, consider Gaussian divisible  matrices of the form 
$H_0+   {\rm GOE}(t)$,  where $H_0$ is an arbitrary Wigner matrix and \nc ${\rm GOE}(t)$  is a $N\times N$
Gaussian orthogonal ensemble with matrix entries given by independent Brownian motions (up to the symmetry requirement)
 starting from 0.  
For any fixed time $t$,   ${\rm GOE}(t)$ is a GOE matrix ensemble with variances 
of the matrix entries proportional to $t$.  The basic idea for  step (ii) is to prove the universality for
 matrices  of the form \nc $H_0+   {\rm GOE}(t)$ for $t$ small, say, $t = N^{-1 + \e}$ for some 
  $\e>0$.  \nc Finally, in step (iii), one shows that the eigenvalue statistics of the 
original matrix $H$ can be approximated by $H_0+   {\rm GOE}(t)$  for a good choice of $H_0$. For $0\leq\e<1/2$ and $H$ satisfying (\ref{genwig})  with a matching lower bound $s_{ij}\ge c/N$, $c>0$, up to a trivial rescaling 
 we can choose $H_0 = H$  \cite{BouYau2013}.  If $1/2\le\e<1$, more complicated arguments requiring matching higher moments of the matrix entries are needed  to choose an appropriate $H_0$ \nc \cite{ErdYauYin2012Univ}. 
Unfortunately, both methods for this third step depend on the fact that the second moments of the entries of the original matrix match those 
of   $H_0+{\rm GOE}(t)$,  up to rescaling. 
 For  band matrices, the variances outside the band vanish; therefore, the second moments of $H_0+   {\rm GOE}(t)$ and 
 the band matrix $H$ will never match outside the band. For the past years, this obstacle in step (iii) has  been  a major roadblock 
to extend the three-step strategy to the band matrices and to other non mean-field models. In this paper, 
we introduce a new method that overcomes this difficulty. In order to outline the main idea, we first need to describe the quantum unique ergodicity as proved in \cite{BouYau2013}. 
 
From the local law for band matrices \cite{ErdKnoYauYin2013} with $W=c N$, we have the  
complete delocalization of eigenvectors: 
with very high  probability 
$$
\max |\psi_k(i)|\leq \frac{(\log N)^{C \log \log N}}{\sqrt{N}},
$$
 where $C$ is a fixed constant and the maximum ranges over all coordinates $i$ of  all 
 the $\ell^2$-normalized eigenvectors, 
 $\b\psi_1,\dots,\b\psi_N$.
Although this bound prevents concentration of eigenvectors  onto a set of size less than
$N(\log N)^{- C \log\log N}$, it does not imply 
the ``complete flatness" of eigenvectors in the  sense that $|\psi_k(i)|\approx N^{-1/2}$. 
 Recall the quantum ergodicity theorem  
(Shnirel'man \cite{Shn1974}, Colin de Verdi\`ere \cite{Col1985} and  Zelditch \cite{Zel1987}) asserts  that 
``most" eigenfunctions for the Laplacian on a  compact Riemannian 
manifold with ergodic geodesic flow are completely flat.   
For  $d$-regular graphs under certain assumptions  on the injectivity radius and  spectral gap of the adjacency matrices, 
similar results were proved for eigenvectors of the adjacency matrices \cite{AnaLeM2013}.
A stronger notion of quantum ergodicity, the  quantum unique ergodicity (QUE) proposed by Rudnick-Sarnak \cite{RudSar1994} demands that    \emph{all} high energy eigenfunctions become completely flat, and it supposedly holds for negatively curved compact Riemannian manifolds. 
One case for which QUE was rigorously proved concerns arithmetic surfaces, 
thanks to tools from number theory and ergodic theory on homogeneous spaces 
\cites{Lin2006,Hol2010,HolSou2010}.

For Wigner matrices,  a probabilistic version of QUE  was settled 
in \cite{BouYau2013}. In particular, it is known that there exists $\e>0$ such that
for any  deterministic $1\leq j\leq N$ and $I\subset\llbracket 1,N\rrbracket$, for any $\delta>0$ we have
\begin{equation}\label{eqn:QUEintro}
 \P\left( \Big | \sum_{i\in I} |\psi_j(i)| ^2  - \frac{|I|}{N} \Big | \ge \delta\right)\le 
    N^{- \e}/\delta^2.  
\end{equation}
Our key idea for proving universality of band matrices is a \emph{mean-field reduction}.  In this method, the above probabilistic  QUE
will be a central tool.  To explain the mean-field reduction and its link with QUE,  we 
block-decompose \nc the band matrix   $H$ and its eigenvectors as
\be\label{H0}
   H=  \begin{pmatrix} A  & B^* \cr B & D  \end{pmatrix}, \quad 
   \b\psi_j:=   \begin{pmatrix}\b w_j \cr \b p_j \end{pmatrix},
\ee
where $A$ is a $W\times   W$ matrix. From the eigenvector equation $H \b\psi_j = \lambda_j \b\psi_j$ we have
\be\label{1100}
   \Big( A- B^* \frac{1}{D - \lambda _j }B\Big) \bw_j  = \lambda_j  \bw_j,
\ee
i.e. $\bw_j $ is an eigenvector of $ A- B^* (D  - \lambda_j)^{-1}B$, with corresponding eigenvalue $\lambda_j$. 
In agreement with the band structure, we may assume that the  matrix
elements of  $A$ do not vanish and thus the eigenvalue  problem  in \eqref{1100} features a mean field random matrix 
(of smaller size).  \nc

For a  real parameter $e$,
consider the following matrix
\be\label{Qdefnew}
   Q_e =  A
 - B^* \frac{1}{D-e}B,
\ee
and  let   $\xi_k (e)$, $\bu_k (e)$ be its sequence of eigenvalues and eigenvectors:
$
   Q_e \b u_k(e) = \xi_k(e) \b u_k(e).
$
Consider the curves $e \to \xi_k(e)$ (see Figure \ref{Fig1}).  By definition, 
the intersection points of these curves with the diagonal $e= \xi$ 
are  eigenvalues for $H$, i.e.,  given $j$,  we have
 $\xi_k (\lambda_j) = \lambda_j$ for some $k$. 
From this relation, we can find the eigenvalue $\lambda_j$ near an energy $e$ 
from the values of $\xi_k(e)$ provided that we know the slope of the curves $e \to \xi_k(e)$. It is a simple computation that 
this slope is given by $1-(\sum_{i=1}^W \left| \psi'_j (i) \right|^2)^{-1}$, where $\b\psi'_j$ is the eigenvector of  
$ H_e$ where $H_e$ is the same as $H$ except $D$ is replaced by $D-e$
(see Subsection \ref{subsec:sketch} for details). If the QUE in the sense of \eqref{eqn:QUEintro} holds for $\b\psi_j'$, then 
$\sum_{i=1}^W \left| \psi'_j (i) \right|^2 \sim W/N$ and the leading order of the slope is a constant, independent of $k$. Therefore, the statistics of 
$\lambda_j$ will be given by those of $\xi_k$ up to a trivial scaling factor.  
Since $\xi_k$'s are eigenvalues of a mean field random matrix, thanks to $A$, 
the universal statistics of $\xi_k$  will follow from previous methods. \nc

\begin{figure}
\centering
\begin{subfigure}{.4\textwidth}
  \centering
\begin{tikzpicture}
\node[anchor=south west,inner sep=0] (x) at (0,0) {\includegraphics[width=7cm]{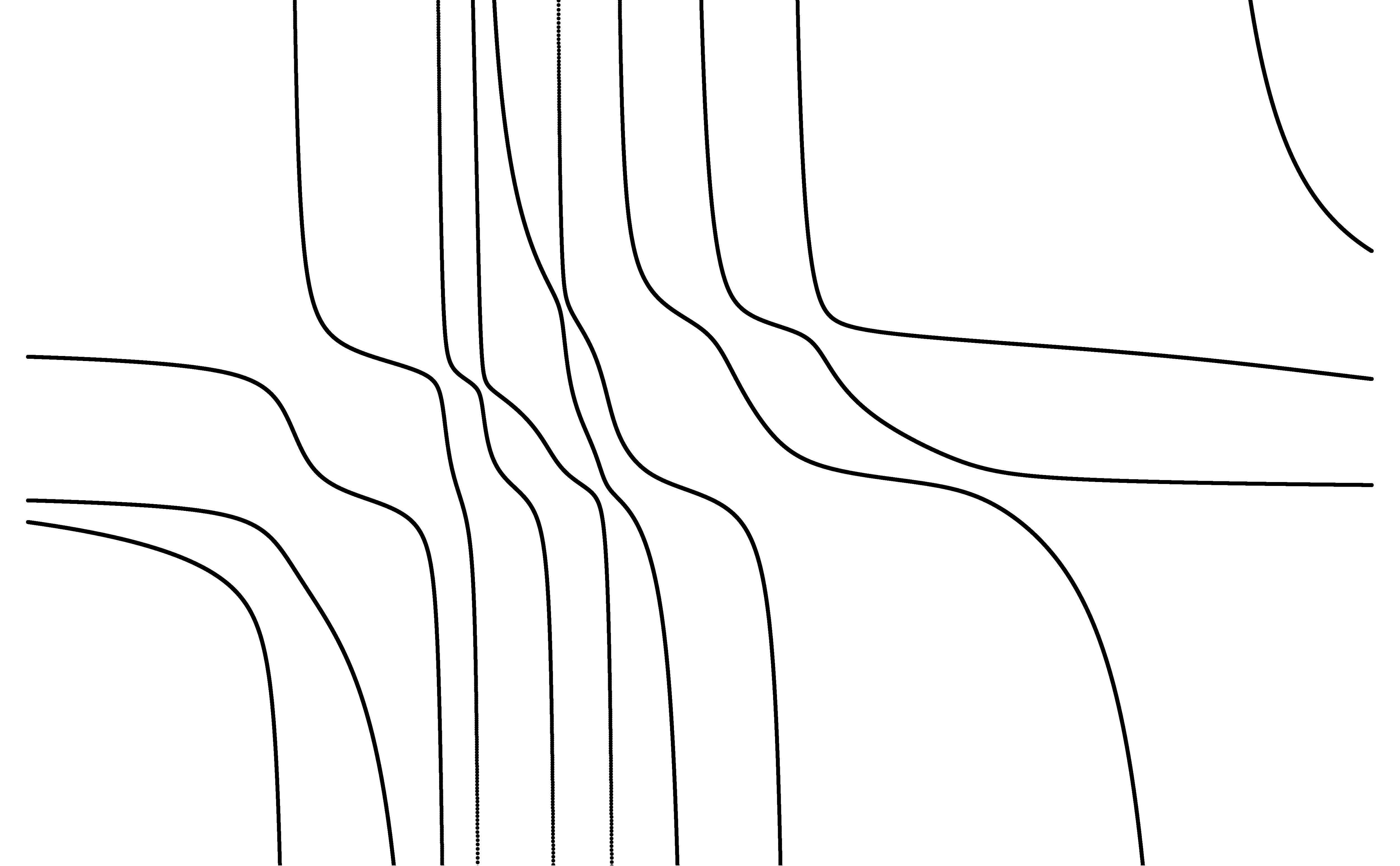}};
\begin{scope}[x={(x.south east)},y={(x.north west)}]
\draw[black,thick,rounded corners] (0.5,0.55) rectangle (0.64,0.75);
\draw [black,->] (0,0.5) -- (1.02,0.5);
\draw [black,->] (0.47,0) -- (0.47,1);
\draw [black,dashed] (0.47,0.5) -- (0.81,1);
\draw [black,dashed] (0.47,0.5) -- (0.13,0);
\fill[black] (0.198,0.1)  circle[black,radius=1.5pt];
\fill[black] (0.26,0.19)  circle[black,radius=1.5pt];
\fill[black] (0.313,0.265)  circle[black,radius=1.5pt];
\fill[black] (0.339,0.305)  circle[black,radius=1.5pt];
\fill[black] (0.388,0.378)  circle[black,radius=1.5pt];
\fill[black] (0.423,0.426)  circle[black,radius=1.5pt];
\fill[black] (0.435,0.445)  circle[black,radius=1.5pt];
\fill[black] (0.452,0.468)  circle[black,radius=1.5pt];
\fill[black] (0.52,0.58)  circle[black,radius=1.5pt];
\fill[black] (0.555,0.622)  circle[black,radius=1.5pt];
\fill[black] (0.585,0.665)  circle[black,radius=1.5pt];
\fill[black] (0.198,0.1)  circle[black,radius=1pt];
\node[black] at (1.05,0.5) {$e$};
\end{scope}
\end{tikzpicture}
  \caption{A simulation of  eigenvalues of $Q_e=A-B^*(D-e)^{-1}B$, i.e. functions $e\mapsto \xi_j(e)$. Here $N=12$ and $W=3$. The $\lambda_i$'s are the abscissa of the intersections with the diagonal.}
  \label{fig:sub1}
\end{subfigure}%
\begin{subfigure}{.1\textwidth}
  \centering
\begin{tikzpicture}
\draw [black,dashed] (0,0) -- (0,0);
\end{tikzpicture}
\end{subfigure}%
\begin{subfigure}{.4\textwidth}
\centering
\vspace{-1.1cm} \hspace{-2cm}
\begin{tikzpicture}
\node[anchor=south west,inner sep=0] (y) at (0,0) {\includegraphics[width=8.5cm]{C.jpg}};
\begin{scope}[x={(y.south east)},y={(y.north west)}]
\draw[white,ultra thick,fill=white]  (0,0) rectangle (1,1);
\draw [black,-,thick,dashed] (0.21,0.16) -- (0.65,0.8);
\draw [black,-,thick,dashed] (0.37,0.16) -- (0.37,0.8);
\draw [black,-,thick] (0.1,0.7) -- (0.8,0.3);
\draw [black,-,thick] (0.1,0.8) -- (0.8,0.4);
\draw [black,-,thick] (0.1,0.93) -- (0.8,0.53);
\draw [black,-,thick] (0.1,1) -- (0.8,0.6);
\draw [black,-,thick] (0.1,0.64) -- (0.8,0.24);
\draw [black,-,thick] (0.1,0.5) -- (0.8,0.1);
\draw [black,-,thick] (0.1,0.4) -- (0.8,0.02);
\draw [black,-,thick] (0.1,0.32) -- (0.63,0.04);
\fill[black] (0.26,0.235)  circle[black,radius=1.8pt];
\fill[black] (0.3,0.29)  circle[black,radius=1.8pt];
\fill[black] (0.35,0.36)  circle[black,radius=1.8pt];
\fill[black] (0.415,0.46)  circle[black,radius=1.8pt];
\fill[black] (0.445,0.5)  circle[black,radius=1.8pt];
\fill[black] (0.495,0.573)  circle[black,radius=1.8pt];
\fill[black] (0.56,0.665)  circle[black,radius=1.8pt];
\fill[black] (0.595,0.715)  circle[black,radius=1.8pt];
\draw[black,fill=black] (0.37,0.775)  +(-1.3pt,-1.3pt) rectangle +(1.3pt,1.3pt) ;
\draw[black,fill=black] (0.37,0.645)  +(-1.3pt,-1.3pt) rectangle +(1.3pt,1.3pt) ;
\draw[black,fill=black] (0.37,0.545)  +(-1.3pt,-1.3pt) rectangle +(1.3pt,1.3pt) ;
\draw[black,fill=black] (0.37,0.485)  +(-1.3pt,-1.3pt) rectangle +(1.3pt,1.3pt) ;
\draw[black,fill=black] (0.37,0.345)  +(-1.3pt,-1.3pt) rectangle +(1.3pt,1.3pt) ;
\draw[black,fill=black] (0.37,0.255)  +(-1.3pt,-1.3pt) rectangle +(1.3pt,1.3pt) ;
\draw[black,fill=black] (0.37,0.18)  +(-1.3pt,-1.3pt) rectangle +(1.3pt,1.3pt) ;
\draw [black,->](0.37,0.845)-- (0.495,0.774);
\filldraw[white,fill=white]
(0,0) -- (0,1) -- (1,1) -- (1,0) -- cycle
(0.2,0.15) -- (0.65,0.15) -- (0.65,0.8) -- (0.2,0.8) -- cycle;
\draw[black,thick,rounded corners] (0.2,0.15) rectangle (0.65,0.8);
\draw [black,->,thick] (0.37,0.05) -- (0.37,0.13);
\node[black] at (0.37,0.02) {$e$};

 \draw [black,->,thick] (0.31,0.05) -- (0.31,0.13);
\node[black] at (0.31,0.02) {$\lambda'$};

\draw [black,->,thick] (0.26,0.05) -- (0.26,0.13);
\node[black] at (0.26,0.02) {$\lambda$};

\draw [black,->](0.37,0.775)-- (0.475,0.715);
\draw [black,->] (0.37,0.645) -- (0.44,0.605);
\draw [black,->] (0.37,0.545)-- (0.415,0.52);
\draw [black,->] (0.37,0.49)-- (0.398,0.468);
\draw [black,->] (0.37,0.255) -- (0.33,0.274);
\draw [black,->] (0.37,0.18) -- (0.31,0.209);
\end{scope}
\end{tikzpicture}
  \caption{Zoom into the framed region of Figure (a), for large $N,W$: the curves $\xi_j$ are almost parallel, with  slope about $1-N/W$.  The eigenvalues of  $A-B^*(D-e)^{-1}B$ and those of $H$ are related by a
  projection to the diagonal followed by a projection to
  the horizontal axis.
 }
   \label{fig:sub1}
\end{subfigure}

\caption{The idea of mean-field reduction: universality of gaps between eigenvalues for fixed $e$ implies universality on the diagonal through parallel projection.}
\label{Fig1}
\end{figure}

To summarize, our idea is to use the mean-field reduction  to convert the  problem of 
 universality of the band matrices ($H$) to a 
matrix ensemble ($Q_e$) of the form $A+R$ with $A$ a Wigner ensemble of the size of the band, independent of $R$. The key input for this mean-field reduction is the 
QUE for the big band matrix. This echoes the folklore belief that delocalization (or QUE) and random matrix statistics occur simultaneously. 
In fact, this is the first time that universality of random matrices is proved via QUE. We wish to emphasize that, as a tool for proving universality, we will need 
QUE while quantum ergodicity is not strong enough. 

In order to carry out this idea, we need  $(i)$ to prove the QUE \eqref{localque0} for the band matrices;
 $(ii)$ to  show 
that the eigenvalue statistics of $Q_e$ are universal. The last problem was recently studied in \cite{ES,LY} which can be applied to the current setting once some basic estimate for $Q_e$ is obtained. The QUE for the band matrices, however, is a difficult problem.  The method 
in \cite{BouYau2013} for proving QUE  depends  on analysis of the flow of the eigenvectors $H_0+   {\rm GOE}(t)$ and on the comparison between  the eigenvectors
of this matrix ensemble and those of the original  matrices. Once again, 
due to vanishing matrix elements in $H$, \nc we will not be able to use the comparison idea and the method  
in \cite{BouYau2013} cannot be applied directly.  Our  idea to resolve this difficulty is
 to use again the mean field reduction, this time for eigenvectors,  and 
consider the eigenvector of the matrix $Q_e$. 
Recall the decomposition (\ref{H0}) of the band matrix.  From (\ref{1100}), 
$\bw_j$ is an eigenvector to $Q_{\lambda_j}$.  Temporarily neglecting the fact  that $\lambda_j$ is random, we will prove that QUE holds 
for $Q_{e}$ for any $e$ fixed  and thus $\bw_j$ is completely flat.  This implies that   the first $W$ indices of $\b\psi_j$ are completely flat. 
We now apply this procedure  inductively  to the  decompositions of the band matrix 
where the role of $A=A_m$ will be played by the $W\times W$ minor on the diagonal of $H$ between
indices $m W/2+1$ and $(m+1)W/2$,
 where  $m=0, \ldots, (2N-W)/W$ is an integer. Notice that the successively considered blocks $A_1, A_2, \ldots, A_m$
  overlap to guarantee consistency. \nc
Assuming QUE holds in each decomposition,  we have concluded that $\b\psi_j$ is completely flat by this {\rm patching} procedure. This supplies the QUE we need for the band matrices,
provided that we can resolve the technical problem that we need these results for 
 $e= \lambda_j$, which is random.   The resolution of this question relies on a new tool in analyzing non mean-field random matrices:   an {\it uncertainty principle} 
asserting  that  whenever a vector is nearly  an eigenvector, it  is  delocalized on macroscopic scales. This extends 
the delocalization estimate for eigenvectors  to approximate eigenvectors 
and is of independent interest. This will be presented in Section 3.\\

\noindent {\it Convention.} We denote $c$ (resp. $C$) a small (resp. large) constant which may vary from line to line but does not depend on other parameters. By $W=\Omega(N)$ we mean $W\geq c N$
 and $\llbracket a, b\rrbracket:=[a,b]\cap \Z$ refers to all integers between $a$ and $b$. \nc

\section{Main results and sketch of the proof}\label{sec:main}

\subsection{The model and the results.\ }
Our method mentioned in the introduction applies to all symmetry classes, but for definiteness we will discuss 
the real symmetric case (in particular all eigenvectors are assumed to be real). 
Consider an $N\times N$ \nc band matrix  $H $ with real centered  entries that are \nc
independent up to the symmetry condition, and band width $4W-1$ (for notational convenience later in the paper) such that $N=   2 Wp$ with some fixed $p \in \N$,  i.e. in this paper we consider the case $W=\Omega(N)$. \nc
More precisely,   we assume that 
\begin{equation}\label{eqn:band1}
H_{ij} = 0, \text { if } |i-j| > 2W,
\end{equation}
where the distance  $| \cdot |$ on $\{1,2,\ldots , N\}$ 
 is defined by periodic boundary condition mod $N$. 
We   assume that  the variances  $s_{ij}:= \E ( H_{ij}^2)$ satisfy
\begin{equation}\label{sumsone}
    \sum_{j} s_{ij}=1 \;\; \text{for all $i$};
\end{equation}
\begin{equation}\label{bandcw}
  \frac{c}{W}\le s_{ij} \le \frac{C}{W},  \;\;  \text { if } |i-j| \le 2W
\end{equation}
for some positive constants.
For simplicity of the presentation, we assume identical variances within the band, i.e. we set
\begin{equation}\label{eqn:band2}
s_{ij}= \E  \left(H_{ij}^2\right)  = \frac{1}{4W-1}, \;\;  \text { if } |i-j| \le 2W.
\end{equation}
 but  our result with the same proof
holds under the general conditions \eqref{sumsone} and \eqref{bandcw}. 
We also assume that for some $\delta>0$ we have
\begin{equation}\label{eqn:subgaus}
\sup_{N,i,j}\E\left(e^{\delta W H_{ij}^2}\right)<\infty.
\end{equation}
This condition can be easily weakened to some finite moment assumption, we assume (\ref{eqn:subgaus}) mainly for the convenience of presentation. The eigenvalues of $H$ are ordered,
$
\lambda_1\leq\dots\leq \lambda_N
$,
and we know that the empirical spectral measure $\frac{1}{N}\sum_{k=1}^N\delta_{\lambda_k}$ converges
almost surely to the Wigner semicircle distribution with density
\begin{equation}\label{eqn:semicir}
\rho_{\rm sc}(x)=\frac{1}{2\pi}\sqrt{(4-x^2)_+}.
\end{equation}
Our main result is the universality of the gaps between eigenvalues:  finitely many
consecutive \nc  spacings between eigenvalues of $H$ have the same limiting distribution as 
 for the Gaussian Orthogonal Ensemble, ${\rm GOE}_N$, which is known as the
 multi-dimensional Gaudin distribution. \nc

\begin{theorem} \label{Univ}
Consider a band matrix $H$
satisfying \eqref{eqn:band1}--\eqref{eqn:subgaus} with parameters $N=2pW$. \nc
  For any fixed $\kappa>0$ and  $n\in\mathbb{N}$  there exists an $\e =\e (p,\kappa,n)>0$ such that
  for any smooth  and compactly supported  function $O$ \nc in $\R^n$, and $k\in \llbracket \kappa N, N-\kappa  N\rrbracket$ we have
\be
\bigg|  \left(\E^{H}- \E^{{\rm GOE}_N} \right)
 O \left(  N \rho_{\rm sc} (\lambda_k ) ( \lambda_{k+1}  - \lambda_{k}) ,\dots,  N  \rho_{\rm sc} ( \lambda_i ) 
( \lambda_{k+n}  - \lambda_{k+n-1} ) \right)  \bigg| \leq C_O   N^{-\e},
\ee
where the constant $C_O$ depends only on $\kappa$  and the test function $O$.
\end{theorem} 
Prior to our work, the only universality result for band matrices was given by T. Shcherbina in \cite{Sch2014}
via rigorous  supersymmetric analysis. The supersymmetric technique required complex hermitian symmetry,
Gaussian distribution and  a very specific variance structure   involving finitely many  blocks with
 i.i.d. matrix elements. Our theorem holds for a general distribution and for both  the
real symmetric and complex hermitian case. Moreover, no block structure or i.i.d. entries are required,
see remark after \eqref{eqn:band2}.

As mentioned in the introduction, a key ingredient for Theorem \ref{Univ}  is  the quantum unique ergodicity
of the eigenvectors of our band matrix model. In fact, we will need QUE for small perturbations of $H$ on the diagonal:
for any  vector $\b g= (g_1, \ldots, g_N )\in\RR^N $ we define 
\be\label{Hd}
H^{\b g}= H - \sum_{j=1}^N g_j \b e_j \b e_j^*,  
\ee
 where $\b e_j$ is the $j$-th coordinate vector. \nc
Let $\lambda_1^{\b g}\leq\dots\leq\lambda_N^{\b g}$ be the eigenvalues of $H^{\b g}$ and $  \b\psi_k^{\b g} $ be the
corresponding eigenvectors, 
i.e.
$
 H^{\b g} \b\psi_k^{\b g}= \lambda_k^{\b g}\b\psi_k^{\b g}$. 

\begin{theorem} \label{QUE-Band}  Consider a band matrix $H$
satisfying \eqref{eqn:band1}--\eqref{eqn:subgaus} with parameters $N=2pW$. Then for any  \nc small $\b g$, $H^{\b g } $ satisfies the QUE  in the bulk. More precisely,  there exists
 $\e,\zeta\nc>0$ such that for any fixed $\kappa>0$, there exists $C_{\kappa,  p \nc}>0$ such that for any $k\in \llbracket\kappa N, (1-\kappa)N\rrbracket$,  $\delta>0$, and $\b a\in [-1,1]^N$, we have 
\be\label{localque0}
\sup_{ \|{\b g}\|_\infty \le N^{-1+  \zeta } }  
 \P\left( \left| \sum_{i=1}^{N} \b a(i) \left( |\psi_k ^{\b g} (i)| ^2  - \frac1N\right) \right| \ge \delta\right)\le 
 C_{\kappa, p\nc}
    N^{-\e}/\delta^2.  
   \ee 
\end{theorem} 

\noindent For the simplicity of exposition, 
we have stated the above result  for QUE  only at macroscopic scales (i.e.,  by choosing a bounded test vector $\b a$),  while it holds at any scale (like in \cite{BouYau2013}). The macroscopic scale will be enough for our proof of Theorem \ref{Univ}.

\subsection{Sketch of the proof.\ } \label{subsec:sketch}
In this outline of  the proof, amongst other things we explain why QUE for small diagonal perturbation $H^{\b g}$ of $H$ 
is necessary to our mean-field reduction strategy. The role of other tools such as the uncertainty principle and the local law is also  
enlightened  below.

We will first need some notation: we decompose $H^{\b g}$ and its eigenvectors as
\be\label{H}
   H^{\b g} :=  \begin{pmatrix} A^{\b g}  & B^* \cr B & D^{\b g} \end{pmatrix}, \quad 
\b\psi_k^{\b g}=   \begin{pmatrix} \b w_k^{\b g} \cr \b p_k^{\b g}\end{pmatrix}, \quad k=1,2,\ldots, W,\nc
\ee
where $A^{\b g}$ is a $W\times   W$ matrix.  
The equation $H^{\b g}\b\psi_k^{\b g}=\la_k^{\b g}\b\psi_k^{\b g}$ then gives
\be\label{110}
   \Big( A^{\b g}- B^* \frac{1}{D^{\b g} - \lambda_k^{\b g} }B\Big) \bw_k^{\b g}  \nc = \lambda_k^{\b g}  \bw_k^{\b g},
\ee
i.e. $\bw_k^{\b g}, \lambda_k^{\b g}$ are the eigenvectors and eigenvalues of $Q^{\b g}_{\la_k^{\b g}}$ where we define
\be\label{Qdef}
   Q^{\b g}_e :=  A^{\b g} 
 - B^* \frac{1}{D^{\b g}-e}B
\ee
for any real parameter $e$.
Notice that $A^{\b g}$ depends only on $g_1,\cdots, g_{W}$ and $D^{\b g}$ depends only on $g_{W+1,}\dots, g_{N}$.
Let $\xi^{\b g}_1(e)\leq\dots\leq\xi^{\b g}_W$ be the ordered sequence of eigenvalues of $Q_e^{\b g}$ and $\bu^{\b g}_k (e)$ the corresponding eigenvectors:
\be\label{Qu}
   Q^{\b g}_e \b u^{\b g}_k(e) = \xi^{\b g}_k(e) \b u^{\b g}_k(e).
\ee
We will be  interested in a special class $g_i = g \1_{i>W}$ for some  $g\in\R$, and we denote the matrix
\be
  H^g  : =  \begin{pmatrix} A  & B^* \cr B & D-g \end{pmatrix},
\label{meq}
\ee
and let $\psi_j^g$, $\lambda_j^g$ be its eigenvectors and eigenvalues. \nc

\bigskip 
\noindent{\it First step:  From QUE   of $H^g$  to universality  of $H$ by  mean-field reduction.} 
Following Figure \ref{Fig1}, we obtain  eigenvalue statistics of $H$ by parallel projection.
Denote $\mathcal{C}_1,\dots,\mathcal{C}_N$ the continuous curves depicted in Figure \ref{fig:sub1}, labelled
in increasing order of their intersection with the diagonal (see also Figure \ref{Fig13} and Section \ref{sec:universality} for a formal definition of these curves).

Assume we are interested in the universality of the gap $\lambda_{k+1}-\lambda_k$ for some fixed $k\in\llbracket \kappa N,(1-\kappa)N\rrbracket$, and let $\xi>0$ be a small constant. By some a priori local law, we
know $|\lambda_k-e_0|\leq N^{-1+\xi}$ for some deterministic $e_0$, with overwhelming probability.
Universality of the eigenvalue gaps around $\lambda_k$ then follows from two facts: (i) universality of gaps between eigenvalues of $Q_{e_0}$ in the local window $I=[e_0-N^{-1+\xi},e_0+N^{-1+\xi}]$, (ii) the lines $(e\mapsto \mathcal{C}_j(e))_{j=k, k+1}$ have  almost constant  identical  negative slope  in the window $e\in I$. \nc 

For (i), note that the $Q_{e_0}=A+R$ where $A$ is a mean-field, Wigner, random matrix and $R$ is independent of $A$. For such matrices, bulk universality is known \cites{LeeSchSteYau2015,ES,LY}. The key tools are some a priori 
 rigidity estimates  for the eigenvalues (see the fourth step), a coupling between Dyson Brownian motions \cite{BouErdYauYin2015} and H\"{o}lder estimates for a resulting parabolic equation \cite{ErdYau2012}.

For the key step (ii), the slopes are expressed through QUE properties of matrices of type $H^{g}$.
More precisely, first note that   any $e\in I$ can be written uniquely as
$$
e=\lambda_k^g+g
$$
 for some \nc $|g|\leq C N^{-1+\xi}$. Indeed,  this is true for  $e=\lambda_k$ with $g=0$,
 and the function $g\to \lambda_k^g+g$ has a regular inverse, since  \nc
 by perturbative calculus $\partial (\lambda_k^g+g)/\partial g=\sum_{i=1}^{W} \left|\psi_k^g( i) \right|^2$, which is larger than  \nc
 some deterministic $c>0$, with overwhelming probability, by the uncertainty principle detailed in the third step. Once such a writing of $e$ is allowed, differentiating in $g$ the identity $\mathcal{C}_k(\lambda_k^g+g)=\lambda_k^g$ (a direct consequence of (\ref{110})) gives
\begin{equation}\label{eqn:slope}
\frac{\partial}{\partial e}\mathcal{C}_k(e)=1-\left(\sum_{i=1}^{W} \left|\psi_k^g( i) \right|^2\right)^{-1}.
\end{equation}
As a consequence, using QUE in the sense of Theorem \ref{QUE-Band}, we know that $(\partial/\partial e)\mathcal{C}_k$ and $(\partial/\partial e)\mathcal{C}_{k+1}$ are almost constant, approximately
$1-(N/W)$. By parallel projection we obtain universality for $H$ 
from universality of $Q_{e_0}$.  In terms of scales, the average gap between eigenvalues of $Q_{e_0}$ around $e_0$
is $(W\rho_{\rm sc}(e_0))^{-1}$, hence the average gap $\lambda_{k+1}-\lambda_k$ is $(N\rho_{\rm sc}(e_0))^{-1}$ as expected.
This mean-field reduction strategy is detailed in Section \ref{sec:universality}.
\\

\noindent{\it Second step. Quantum unique ergodicity.}
The proof of Theorem \ref{QUE-Band} proceeds in four steps, with successive proofs of QUE for the following eigenvectors ($k'$ is the unique index such that $\xi_{k'}$ lies on the curve $\mathcal{C}_{k}$):
\begin{enumerate}[(i)]
\item $\bu_{k'}^{g}(e)$ ($\ba\in[-1,1]^W$);
\item $\bu_{k'}^{g}(\lambda_k^g)$ ($\ba\in[-1,1]^W$);
\item $\bw_k^{g}$ ($\ba\in[-1,1]^W$);
\item $\b\psi_k^{g}$ ($\ba\in[-1,1]^N$).
\end{enumerate}
 In the parentheses we indicated the type of test vectors used in the QUE statement. \nc

First, (i) is QUE for a  matrix of type 
$Q_{e}=A+R$ where $A$ is a mean-field, Wigner, random matrix and $R$ is independent of $A$. For such matrices, QUE is known from the work \cite{BouHuaYau2016}, which made  use of the local \nc eigenvector moment flow method from \cite{BouYau2013}. For this step, some a priori information on location of eigenvalues of $Q_e$ is necessary and given by the local law (see the the fourth step).

From (i) to (ii), some stability the eigenvectors of $Q_e$ is required as $e$  varies. \nc 
Accurate estimates on $(\partial/\partial e)\bu^{g}_{k'}(e)$ are given by 
the uncertainty principle (see the third step) and
rigidity estimates of the eigenvalues (see the fourth step).

From (ii) to (iii), note that $\bw_k^{g}$ and  $\bu_{k'}^{g}(\lambda_k^g)$ are collinear, so QUE for $\bw_k^{g}$ will be proved provided it is properly normalized: 
\begin{equation}\label{eqn:normalized}
\|\bw_k^{g}\|_{\ell^2}^2\approx W/N.
\end{equation}
This is proved by patching:  in (ii), choosing ${\b a}(i)=1$ for $i\in\llbracket 1,W/2\rrbracket$, $-1$ for $i\in\llbracket W/2 +1\nc ,W\rrbracket$, and using translation invariance in our problem, we have $\sum_{i\in\llbracket 1,W/2\rrbracket+\ell W/2}|\psi_k^g(i)|^2\approx\sum_{i\in\llbracket 1,W/2\rrbracket+(\ell+1) W/2}|\psi_k^g(i)|^2$ for any $\ell$, so that (\ref{eqn:normalized}) holds.

The final step from (iii) to (iv) is a simple consequence of translation invariance, as (iii) holds for any $W$ successive coordinates of $\b\psi^{g}_k$. These steps are detailed in Section \ref{sec:QUE}.
\\

\noindent{\it Third step. Uncertainty principle.} This important ingredient of the proof can be summarized as follows: any vector  approximately satisfying the eigenvector equation 
 of $H^{\b g}$ or $D^{\b g}$ is delocalized in the sense that macroscopic subsets of its coordinates carry a non-negligible portion of its $\ell^2$ norm (see Proposition \ref{apL} for a precise statement).
This information allows us to bound the slopes of the curves $e\mapsto \mathcal{C}_k(e)$ through (\ref{eqn:slope}). It is also important in the proof of the local law for matrices of type $Q_e$ (see Lemma \ref{lem: obG}).

The proof of the uncertainty principle relies on an induction on $q$, where $N=qW$, classical large deviation estimates and   discretization of the space \nc arguments. Details are given in Section \ref{sec:uncertainty}.\\

\noindent{\it Fourth step. Local law.}
The local law for matrices of type $Q_e$ is necessary for multiple purposes in the first two steps, most notably to establish  universality of eigenvalues in a neighborhood of $e$  and QUE for corresponding eigenvectors. 

Note that the limiting empirical spectral distribution of $Q_e$ is hard to be made explicit, and in this work we do not aim at describing it.
Instead, we only prove bounds on the Green's function of $Q_e$ {\it locally}, i.e. 
$$
(Q_e-z)^{-1}_{ij}\approx m(z)\delta_{ij},\ \ N^{-1+\omega}\leq {\rm Im}(z)\leq N^{-\omega},
$$
in the range when $|{\rm Re}(z)-e|$ is small enough.  Here $m(z)$ is the Stieltjes transform
of the limiting spectral density whose precise form is irrelevant for our work. 
This estimate is obtained from the local law for the band matrix $H$ \cite{ErdKnoYauYin2013} through Schur's complement formula.
This {\it local} a priori information on eigenvalues (resp. eigenvectors) is enough to prove universality by Dyson Brownian motion coupling (resp. QUE through the eigenvector moment flow) strategy. 
The proof of the local law is given in Section \ref{sec:local}.\\

In the above steps, we assumed that the entries of $H$ have distribution which is a convolution with a small normal component (a {\it Gaussian-divisible ensemble}), so that the mean-field matrices $Q_e$ are the result of a matrix Dyson Brownian motion evolution. This assumption is classically removed by density arguments such as the Green functions comparison theorem \cite{ErdYauYin2012Univ} or microscopic continuity of the Dyson Brownian motion \cite{BouYau2013}, as will be appearent later along the proof.

\section{Uncertainty principle}

\label{sec:uncertainty}

This section proves an uncertainty principle for our band matrices
satisfying \eqref{eqn:band1}--\eqref{eqn:subgaus}: if a vector approximately 
satisfies the eigenvalue equation, then 
  it is \nc delocalized on macroscopic scales.

\begin{proposition}\label{apL} Recall the notations (\ref{H}). There exists $\mu>0$ such that for any (small) $c>0$ and (large) $D>0$,  we have, for large enough $N$, 
\begin{align} \label{brownBand}
& \P\left(\exists e\in \R,  \exists \bu \in \R^{ N-W}, \exists \b g \in \R^N \;: \; 
\;\|\b g\|_\infty\le N^{-c}, \; \|\bu\|=1,\ 
  \|(D^{\b g}-e)\bu\|\le \mu,\ \sum_{1\le i\le W} |u_i|^2\le  \mu ^2 \right)\le N^{-D},\\
&\P\left(\exists e\in \R,  \exists \b g \in \R^N \;: \;  \;\|\b g\|_\infty\le N^{-c}, \;
  B^* \frac {\mu^2}{(D^{\b g}-e)^2} B \ge   \Big( B^*\frac{1}{D^{\b g}-e}B\Big)^2 
    + 1 \right)\le N^{-D}\label{brownBand2}
\end{align}
\end{proposition}

This  proposition gives  useful information for two purposes.
\begin{enumerate}[(i)]
\item An a priori bound on the slopes of lines $e\mapsto \mathcal{C}_k^{\b g}(e)$ (see  Figure \ref{Fig13} in Section \ref{sec:universality}) will be provided  by inequality (\ref{brownBand2}).
\item The proof of the local law for the matrix $Q_e^{\b g}$ will require  the uncertainty principle (\ref{brownBand}).
\end{enumerate}

\noindent For the proof, we first  consider general random matrices in Subsection \ref{subsec:prelest}
before making an induction on the size of some blocks in Subsection \ref{subsec:induction}.

\subsection{Preliminary estimates. }\label{subsec:prelest}
 In this subsection, we consider a random matrix $B$ of dimension $L\times M$ and a Hermitian matrix
  $D$ of dimension $L\times L$ matrix where $L$ and $M$ are comparable. 
   We have the decomposition~\eqref{H0} in mind  and in the next subsection
   we will apply the results of this subsection, 
Lemma  \ref{basic lem} and Proposition \ref{prop:uncert1}, for 
 $M=W$ and $L=k W$ 
with  some $k\in\llbracket 1,2p-1\rrbracket$.
We assume that $B$ has real independent, mean zero entries and, similarly to (\ref{eqn:subgaus}),
\begin{equation}\label{eqn:subgausB}
\sup_{M,i,j}\E\left(e^{\delta M B_{ij}^2}\right)<C_\delta<\infty
\end{equation}
for some $\delta,C_\delta>0$. In particular, we have the following bound:
\begin{equation}\label{eqn:boundC}
\sup_{M,i,j} s_{ij}<\frac{C_\delta}{\delta M}, \qquad {\rm where}\ s_{ij}:=\E\left(|B_{ij}|^2\right).
\end{equation}
The main technical tool, on which the whole section relies, is the following  lemma. 

\begin{lemma}\label{basic lem}
Let $B$ be an $L\times  M$ random matrix  satisfying the above assumptions and set $\beta := M/L$.
 Let $S$ be a subspace of  
$\R^L$ with $\mbox{dim}\, S =: \alpha L$.  Then for any given $\gamma$ and $\beta$, for small enough positive $\alpha$, we have   
\be\label{Bu2}
    \P \Big( \exists \bu \in S\; : \; \|\bu\|=1, \;
\| B^*\bu\| \le \sqrt\gamma/4 , \; \mbox{and} \;   \min_{1\leq  j\leq M  }    \sum_{i=1}^L s_{ij} |u_i|^2 \ge \gamma  M^{-1}
\Big) \le e^{-cL}
\ee
for large enough $L$.  Here $0<\alpha<\alpha_0(\beta,\gamma,\delta,C_\delta)$  and $c=c(\alpha,\beta, \gamma,\delta, C_\delta)>0$.
\end{lemma}

\begin{proof} With the replacement $B\to   \sqrt\gamma B$, we only need to  prove the case $\gamma =1$
by adjusting $\delta$ to $\delta/\gamma$. Hence in the following proof we set $\gamma=1$.
 
First, we have an upper bound on the norm of $BB^*$. For any  $T\ge T_0(\beta, \delta,C_\delta)$ (with $\delta,C_\delta$ in \eqref{eqn:subgausB}),  
\be\label{uppernorm}
  \P (\| BB^*\|\ge T ) \le e^{-c_1 TL }
\ee
for some small $c_1=c_1(\beta )>0$. This is a standard large deviation result, e.g.
it was proved in   \cite[Lemma 7.3, part (i)]{ESY1}
(this was stated when the $B_{ij}$'s are i.i.d, but  only independence was used in the proof, the identical 
law was not).

Let $\bb_1, \bb_2, \ldots,\bb_M\in \R^L$  be the columns of $B$, then
$
   \| B^*\bu\|^2 = \sum_{j=1}^M |\bb_j\cdot \bu|^2.
$ 
Since the $\bb_j\cdot \bu$ scalar products are independent, we have for any $g>0$
$$
  \P \left( \| B^*\bu\|^2\le 1/2\right) \le e^{g M /2} \E \left( e^{-g M\| B^* \bu\|^2}\right)
  =\prod_{j=1}^M\left( e^{  g /2} \E \left(e^{- g M  |\bb_j\cdot \bu |^2}\right)\right).
$$ 
Since $e^{{ -}g r} \le 1-g r +\frac{1}{2}g^2r^2$ for all $r>0$, for $\|\bu\|=1$ we have
\begin{equation}\label{tayl}
 \E \left(e^{- g M  |\bb_j\cdot \bu |^2} \right) \leq 1-g M   \E   \left(| \bb_j\cdot \bu |\right)^2 + \frac{g^2M^2}{2} 
\E \left(| \bb_j\cdot \bu |^4\right)
  = 1- Mg\sum_{i} \E \left(|B_{ij}|^2 \right)| u_i|^2 + \OO(g^2).
\end{equation}
If $\bu$ satisfies the last condition in the left hand side of \eqref{Bu2}, i.e. (with $\gamma=1$)
$
\sum_i s_{ij} |u_i|^2 \ge  M^{-1}$ for all 
$1\le j\le M$
then \eqref{tayl} is bounded by 
$ 1- g  + \OO(g^2)
 \le \exp{\big(-g  +\OO(g^2)\big)}
$.
 Choosing  $g$ sufficiently small, we have
\be\label{ug}
   \P ( \| B^*\bu\|^2\le 1/2) \le \Big( e^{-g /2 + \OO(g^2)}\Big)^M \le e^{-c_2M }
\ee
 where $c_2 $ depends only on the constants $\delta,C_\delta$ from \eqref{eqn:subgausB}.
 
Now we take an $\epsilon$ grid in the unit ball of $S$, i.e. vectors $\{ \bu_j\; : \; j\in I\} \subset S$ such that
for any $ \bu \in S$, with $\|\bu  \|\le 1$ we have $\| \bu- \bu_j\|\le \e$ for some $j\in I$. 
It is well-known that   $ |I| \le (c_3\e)^{-\dim S}$ for some constant $c_3$ of order one. 
We now choose  $\e = (4\sqrt{T})^{-1}$ (where $T$ is chosen large enough to satisfy \eqref{uppernorm}).
If there exists a $\bu$ in the unit ball of $S$ with $\| B^* \bu\| \le 1/4$  then 
by choosing $j$ such that $\| \bu-\bu_j\|\le \e$ 
we can bound $\| B^*\bu_j \|\le \| B^* \bu\| + \sqrt{T}  \| \bu-\bu_j\| \le 1/2$, provided that $\| B B^*\|<T$.
Hence together with \eqref{ug}, we have
\begin{eqnarray*}
\P \Big( \exists u\in S\, : \; \| B^*u\| \le \frac{1}{4}, \; \| u\|=1 \Big)
\le & \P (\| BB^*\|\ge T ) +  \sum_{j\in I} 
\P \Big( \| B^*u_j\| 
\le \frac{1}{2} \Big)\\
\le & e^{-c_1 T L} + (c_3\e)^{-\dim S} e^{-c_2M}  \le e^{-cL}, 
\end{eqnarray*}
where the last estimate holds if  
\begin{equation}\label{eqn:ineq}
c\le  \alpha \log(c_3 \e)+c_2\beta.
\end{equation}
After the fixed choice of a sufficiently large constant $T$ we have $\log(c_3\e) <0$, and for small enough $\alpha$ there exists $c>0$ such that (\ref{eqn:ineq}) holds, and consequently \eqref{Bu2} as well.
\end{proof}

\begin{proposition}\label{prop:uncert1} 
Let $D$ be an $L\times L$  deterministic matrix and $B$ 
be a random matrix as in Lemma \ref{basic lem}.
Assume that $D$ satisfies
the following two conditions:
\be\label{apriori}
  \|D\|\le C_D
\ee
for some  large constant $C_D$ (independent of L) and 
 \be\label{apriori2}
 \max_{a,b: |a-b|\le (C_D \log L)^{-1}}  \#\left\{{\rm Spec}(D) \cap  [a,b]\right\} \le \frac{L}{\log L}. 
\ee 
For any fixed $\gamma>0$, there exists $\mu_0(\beta, \gamma,  \delta,C_\delta, C_D)>0$ 
 such that if 
$\mu \le \mu_0$, then for large enough $L$ we have
\be
  \P \left(\exists e \in \bR,\  \exists \bu\in  \R^L\ :\   \|\bu\|=1,\ \| B^*\bu\|\le \sqrt \gamma \mu,\
   \min_{1\leq j\leq M}   \sum_{i=1}^L s_{ij} |u_i|^2 \ge \gamma M^{-1},
 \| (D-e) \bu\|\le  \mu \right)\le e^{-cL}.
\label{triple}
\ee
\end{proposition}

\begin{proof}
We will first prove the following weaker statement: for any fixed $e\in \R$ and $\gamma>0$,
if   $\mu \le \mu_0(\beta, \gamma,  C_\delta, C_D)$  is sufficiently small, then for large enough $L$ we have
\be
  \P \left( \exists \bu\; : \; \| \bu \|=1, \;   \| B^*\bu\|\le \sqrt\gamma\mu ,  \;   \min_{j}    \sum_i s_{ij} |u_i|^2 \ge \gamma M^{-1},  \;  \| (D-e) \bu\|\le  \mu \right) \le e^{-cL}.
\label{double}
\ee
As in the proof of Lemma \ref{basic lem}, with the replacement $B\to\sqrt{ \gamma} B$, we only need to prove the case $\gamma =1$.
Fix a small number $\nu$ and consider $P$ to be the spectral projection
$$
   P :=P_\nu:= {\bf 1}( |D-e|\le \nu).
$$
Assume there exists some $\bu$ satisfying the conditions in the left hand side of \eqref{double}. Then we have
$$
\mu^2 \ge   \|(D-e)\bu \|^2  \ge \|(D-e)(1-P)\bu \|^2  \ge \nu^2 \| (1-P)\bu\|^2.
$$
Consequently, denoting  $\bv=P\bu$  and $\b w=(1-P)\bu$, we have
$$
\| \bw\| \le \frac{  \mu }{\nu },\ \quad \| \b v\|^2\ge 1- \frac{ \mu^2}{\nu^2} \ge \frac{1}{2},
$$
provided that  $   \mu^2 \le \nu^2/2$.  Using  the bound $\| B^* \bu\|\le\mu$  in \eqref{double} and $\| \b v\|^2\ge  1/2$,  assuming $\|B^*\|\leq C_1$ (this holds with probability $e^{-c L}$ for large enough $C_1$, by \eqref{uppernorm}), we have
\be\label{appl} 
  \| B^*\b v\|\le \|B^*\bu\| + \| B^*\b w\| \le  \mu + C_1 \| \b w\| \le  2\mu \|\b v\|+ C_2 \frac{\mu }{\nu}\| \b v\|
\ee
with probability $1-\OO(e^{-c L})$. 
Moreover,  by \eqref{eqn:boundC} and the assumption $ \sum_i s_{ij} | u_i|^2 
\ge M^{-1} $ in \eqref{double}, we have 
\be\label{11}
2  \sum_i s_{ij} |v_i|^2 \ge  \sum_i s_{ij} | u_i|^2 -  2 \sum_i s_{ij} | w_i|^2
\ge M^{-1}-  2 \wt C  \| \bw\|^2L^{-1} \ge  (2M)^{-1}
\ee
with $\wt C= C_\delta/(\delta \beta)$ (see \eqref{eqn:boundC})  and 
 provided that $\nu^2\geq 4  \beta \wt C \mu^2$.
Define $\wt \bv=\bv/\|\bv\|$, which is a unit vector in $ {\rm Im}(P)$, the range of $P$.
 So far we proved that 
\begin{multline*}
 \P \left( \exists \bu\; : \; \| \bu \|=1, \;   \| B^*\bu\|\le  \mu ,  \;   \min_{j}    \sum_i s_{ij} |u_i|^2 \ge M^{-1},  \;  \| (D-e) \bu\|\le  \mu \right)\\
\le  \P \left(\exists \wt  \bv\in {\rm Im}(P)\ :\  \|\wt \bv\|=1,\ \| B^* \wt\bv\| \le 2\mu+  C_2\frac{\mu}{\nu}, \quad  \sum_i s_{ij} |\wt v_i|^2\ge (4M)^{-1}
\right)+e^{-cL}.
\end{multline*}
We now set $\mu$ and $\nu$  such that  
$ 2\mu+  C_2\mu/\nu \le1/8$, $ \mu^2 \le \nu^2/2$ and $4\beta \wt C \mu^2\leq \nu^2$. By  Lemma \ref{basic lem}, 
with $S:= {\rm Im}(P)$ and $\gamma=1/4$,  the probability of the above event is
 exponentially small  as
long as $$
{\rm rank }(P)/L\quad  i.e.\quad   \#\left\{ \mbox{Spec}(D) \cap  [e-\nu,e+\nu]\right\}/L $$ is sufficiently small 
(determined by $\beta, \delta, C_\delta$, see the threshold $\alpha_0$ in Lemma \ref{basic lem}).
Together with \eqref{apriori2},
by writing the interval $[e-\nu,e+\nu]$ as a union of intervals of length $(C_D\log L)^{-1}$, 
 by choosing small enough $\nu$, then even smaller $\mu$ 
and finally a large $L$,  we  proved \eqref{double}.

The proof of \eqref{triple} follows by a simple grid argument.
For fixed $\mu>0$, consider a discrete set of energies
$(e_i)_{i=1}^r$ such that (i) $r\leq 2 (C_D+1)/\mu$, (ii)
$|e_j| \le  C_D+1$ for any $1\leq j\leq r$ and (iii) for any $|e|\le    C_D+1$, there is a $1\leq j\leq r$ with
$|e_j-e|\le \mu$. 
If $|e|\leq C_D+1$, we therefore have, for some $1\leq j\leq r$,
$$
\| (D-e_j) \bu\|\le \mu  + |e-e_j| \le  2 \mu.
$$
If $|e|>C_D+1$, then
$
\|(D-e)\bu\|\ge |e|-C_D>1.
$
We therefore proved that, for any $\mu<1$, 
\begin{multline*}
  \P \left( \exists e \in \bR, \;\; \exists \bu\in  \R^L \; : \;  \|\bu\|=1, \;  \| B^*\bu\|\le   \mu, \;
   \min_{j}   \sum_i s_{ij} |u_i|^2 \ge   M^{-1}, \; 
 \| (D-e) \bu\|\le  \mu  \right)\\
\le\sum_{j=1}^r  \P \left(\exists \bu\in \R^L \; : \;  \|\bu\|=1, \;    \| B^*\bu\|\le \mu  ,  \;
   \min_{j}  \sum_i s_{ij} |u_i|^2 \ge M^{-1}, \; \nc
 \| (D-e_j) \bu\|\le 2 \mu  \right).
\end{multline*}
For large enough $L$, the   right hand side is exponentially small by \eqref{double}. 
\end{proof}

\subsection{Strong uncertainty principle. }\label{subsec:induction}

In this subsection, we study the matrix with the following block structure. Let $H =H_0$ be  a $N\times N$ random matrix such that $\{H_{ij}\}_{i\le j}$'s, are independent of each others. Consider the inductive decomposition 
\be\label{fry}
H_{m-1}=\begin{pmatrix}
A_m& B_m^* \cr
B_m & H_{m}
\end{pmatrix},
\ee
where  $A_m$ is a $W\times W$ matrix  and $H_m$ has dimensions $(N-mW)\times (N-mW)$. Remember that in our setting $N=2p W$, so that the decomposition (\ref{fry}) is defined for $1\leq m\leq 2p$ with $H_{2p-1} = A_{2p}$.

\begin{lemma}\label{general-uncertainty}  
In addition to the previous assumptions, assume that the entries of  $B_m$'s, $1\le m\le 2p$, satisfy  \eqref{eqn:subgausB} and  
 \be\label{apriori3.3}
\E |(B_{m})_{ij}|^2\ge \frac{ \wh c}{ W} \quad {\rm for\, all}\ 1\le i, j\le W,
\ee
for some constant $\wh c>0$. For any $K>0$, let $\Omega:=\Omega_K(H)$ be the set of events such that 
\begin{align}
\label{apriori3.1}
&\|A_m\|+\|B_m\|+\|H_m\|\le K,
\end{align}
and
\begin{align}
\label{apriori3.2}
&\max_{a,b: |a-b|\le K^{-1}(\log N)^{-1}}  \#\left\{ \mbox{{\rm Spec}}(H_m) \cap  [a,b]\right\} \le N/(\log N),  
\end{align}
for all $0\leq m\leq 2p$.
Then there exist (small) $\mu_0$ and $c_0$ depending on $(\wh c,K,\delta,C_\delta,p)$, such that for any $0<\mu<\mu_0$ and
$0\le m\le 2p-1$ we have
 \be \label{24}
 \P \Big( \exists e\in \R,\;  \exists \bu\in \mathbb{R}^{ N-mW}:  \| \bu \| = 1,\  \|(H_m-e)\bu\| \le  \mu, 
 \quad \sum_{1\le i\le W} |u_i|^2\le  \mu^2
 \Big )\le e^{- c_0 N}+\P(\Omega^{ \rm c}). 
\ee
\end{lemma}

\begin{proof}  
 We will use an induction from $m=2p-1$ to $m=0$ to prove that  
for each $1\le m\le 2p-1$ there exist two  sequences  of parameters
$\mu_m\in (0,1)$ and   $c_m>0$, \nc depending on $(\wh c,K,\delta,C_\delta)$, such that
\be \label{brown}
  \P\left(\exists e\in \R, \bu \in \R^{N-mW}: \; \|\bu\|=1,\ 
  \|(H_{m}-e)\bu\|\le \mu_m, \  \sum_{1\le i\le W } |u_i|^2\le  \mu_m^2 \right)\le e^{-c_m N}+\P(\Omega^{ \rm c}). 
\ee
This would clearly imply (\ref{24}).
First the case  $m= 2p-1$ is trivial, since we can choose $\mu_{2p-1}=1/2$ and use
$\sum_{1\le i\le W } |u_i|^2=\|\bu\|^2=1$ in this case. 

Now we assume that \eqref{brown} has been proved for some $m+1$, and we need to prove it for $m$.   Assume we are in $\Omega$ and there exists $e$ and $\bu\in \R^{N-mW}$ such that the event in the left hand side of \eqref{brown} holds.  We write $\bu=\begin{pmatrix}
\bv'
\cr
\bv 
\end{pmatrix}$ with  $\bv'\in \R^{W}$, $\|\bv'\|^2=\sum_{1\le i\le W} |u_i|^2$.  
From  $ \|(H_{m}-e)\bu\|\le \mu_m$, we have
$$
\|(A_{m+1}-e)\bv'+B_{m+1}^*\bv\|+\| B_{m+1} \bv'+(H_{m+1}-e)\bv\|\le \sqrt2 \mu_{m}.
$$
Combining \eqref{apriori3.1} with  $ \|(H_{m}-e)\bu\|\le \mu_{m}$, we have   $|e|\le K+\mu_{m}$.  
Inserting it in the above inequality together with $\|\bv'\|\le\mu_m$, and using \eqref{apriori3.1} again, we obtain
$$
\| B_{m+1}^*\bv\|+\|(H_{m+1}-e)\bv\|\le \sqrt2 \mu_{m}+( 4 K +2 \mu_{m})\mu_{m}.
$$
Since $\|\bv\|\ge \sqrt{1-\mu_{m}^2}$, denoting $\wt \bv: =\bv/\|\bv\|$ we have 
$$
\| B_{m+1}^*\wt \bv\|+\|(H_{m+1}-e)\wt \bv\|\le \left(\sqrt2 \mu_{m}+(4K +2\mu_{m})\mu_{m}\right)(1-\mu_{m}^2)^{-1/2}   =:\wt \mu_{m}.
$$
We therefore proved
\begin{align*}
&  \P\left(\exists e\in \R, \exists\bu \in \R^{N-mW}: \; \|\bu\|=1,\ 
  \|(H_{m}-e)\bu\|\le \mu_m, \  \sum_{1\le i\le W } |u_i|^2\le  \mu_m^2 \right)\\
\leq&
  \P\left(\exists e\in \R, \exists\wt\bv \in \R^{N-(m+1)W}: \; \|\wt\bv\|=1,\;
 \| B_{m+1}^*\wt \bv\|+ \|(H_{m+1 }-e)\wt \bv\|\le \wt \mu_{m}  \right)
+
\P(\Omega^{\rm c})\\
\leq&
  \P\left(\exists e\in \R, \exists\wt\bv \in \R^{N-(m+1)W}: \; \|\wt\bv\|=1,\;
 \| B_{m+1}^*\wt \bv\|+ \|(H_{m+1 }-e)\wt \bv\|\le \wt\mu_{m},
\sum_{1\leq i\leq W}|\wt v_i|^2\geq \mu_{m+1}^2\right)\\
&+ e^{-c_{m+1}N}+
\P(\Omega^{\rm c}),
\end{align*}
where in the last inequality we used the induction hypothesis (at rank $m+1$) and we assumed that $\wt\mu_m\leq \mu_{m+1}$, which holds by choosing $\mu_m$ small enough.

With \eqref{apriori3.3}
the last probability is bounded by 
\begin{multline*}
 \P\left(\exists e\in \R, \; \exists\wt\bv \in \R^{N-(m+1)W}: \; \|\wt\bv\|=1,\;
 \| B_{m+1}^*\wt \bv\|+ \|(H_{m+1 }-e)\wt \bv\|\le  \wt \mu_{m}  ,\right. \\
\left.  \min_{1\leq j\leq W} \sum_{i} \E \left|(B_{m+1})_{ij}\right|^2 |\wt  v_i|^2\ge \mu_{m+1}^2 \frac{\wh c}{W} \right) 
\end{multline*}
Applying \eqref{triple} with $\mu=\wt\mu_m$ and $\gamma = \wh c \mu_{m+1}^2$, together 
 with assumption \eqref{apriori3.2}, we know that for small enough $\mu_{m}$ 
(and therefore small enough $\wt \mu_{m}$), the above probability is bounded by $e^{-\wt c N}$ for some $\wt c>0$. Therefore  \eqref{brown} holds at rank $m$  if we define $c_m$ recursively backwards such that $c_m < \min\{ c_{m+1}, \wt c \}$.
 The sequence $\mu_m$  may also be defined recursively backwards
with an initial $\mu_{2p-1}=1/2$ so that each  $\wt \mu_{m}$ remains smaller than $\mu_{m+1}$ and
the small threshold $\mu_0(\beta,\gamma = \wh c \mu_{m+1}^2, \delta, C_\delta, C_D)$ from
 Proposition~\ref{prop:uncert1}.
\end{proof}
   
\begin{corollary}\label{prop:uncert3} 
Under  the assumptions of   Lemma \ref{general-uncertainty}, 
there exist (small) $\wt \mu_0$  and $\wt c$ depending on $(\wh c,K,\delta,C_\delta,p)$, such that for any $0<\mu<\mu_0$
we have\nc
\be\label{jiazy}
  \P \left(\exists e\in \bR\ :\
   B_1^* \frac {\mu^2}{(H_1-e)^2} B_1 \ge   \Big( B_1^*\frac{1}{H_1-e}B_1\Big)^2 
    + 1\right)\le \P(\Omega^{ \rm c})+e^{-\wt c N}.
\ee
  \end{corollary}
\begin{proof}  By definition, the left hand side of \eqref{jiazy} is
$$
\P \left(  \exists e\in \bR, \; \exists  \bu\in \R^{W}\; : \; \|\bu\|=1, \;  \mu\Big\| \frac{1}{H_1-e} B_1 \bu\Big\| \ge 
  \Big\{\Big\|  B_1^*\frac{1}{H_1-e}B_1 \bu\Big\|^2 +1  \Big\}^{1/2} \right)
$$
Define $ \bv:= (H_1-e)^{-1} B_1 \bu$,  and $\wt \bv:= \bv/\|\bv\| $. As $\|B_1\|\le K$ in $\Omega$,
the above probability is bounded by
\begin{multline*}
\P \left(\exists e\in \bR, \;  \exists \bv\in \R^{N-W} \; : \;  \mu  \|\b v \| \ge 
  \big( \|  B_1^*\bv\|^2 +1\big)^{1/2}, \; \; \| (H_1-e)\bv\| \le K\right)  +\P(\Omega^{ \rm c})\\
  \le   \P \left( \exists e\in \bR, \; \exists \wt\bv\in \R^{N-W}  \; : \;  \|\wt\bv\|=1, \;  
  \|  B_1^*\wt \bv\| \le \mu  , \; \; \| (H_1-e)\wt \bv\| \le K\mu    \right)     +\P(\Omega^{\rm c}).
\end{multline*}
With \eqref{24} (choosing $m=1$), for any  $\mu\le \mu_0$, 
where $\mu_0$ was obtained in Lemma~\ref{general-uncertainty},
  the above expression is bounded  by 
\begin{multline*}
  \P \left(\exists e\in \bR, \; \exists \wt\bv\in \R^{N-W}  \; : \;  \|\wt\bv\|=1, \;  
  \|  B_1^*\wt \bv\| \le \mu  , \; \; \| (H_1-e)\wt \bv\| \le K\mu, \;\; \sum_{1\le i\le W} |\wt \bv_i|^2\ge  \mu_{0} ^2\right)     +e^{-c_0N}+\P(\Omega^{\rm c})\\
\le   \P \Bigg(\exists e\in \bR, \; \exists \wt\bv\in \R^{N-W}  \; : \;  \|\wt\bv\|=1, \;  
  \|  B_1^*\wt \bv\| \le \mu  , \; \; \| (H_1-e)\wt \bv\| \le K\mu, \;\;   \\ 
 \min_{1\leq j\leq W}  \sum_{1\le i\le W} \E |(B_1)_{ij}|^2|\wt \bv_i|^2\ge  \mu_{0} ^2  \frac{\wh c}{W}  \Bigg)
   +e^{-c_0N}  +\P(\Omega^{\rm c}).
\end{multline*}
For the last inequality we used \eqref{apriori3.3}.  From \eqref{triple} with $\gamma = \wh c\mu_0^2$ and
 for small enough $\mu\le \wt\mu_0:=\mu_0(\beta,\gamma = \wh c \mu_0^2, \delta, C_\delta, C_D)$,  
the above term is bounded by $\P(\Omega)+e^{-\wt c N}$ for some $\wt c>0$, which completes the proof of Corollary \ref{prop:uncert3}.
\end{proof}

\begin{proof}[Proof of Proposition \ref{apL}.]  We write $H^{\b g}$  in the from of \eqref{fry}.  Then $H^{\b g}$ satisfies the assumptions \eqref{eqn:subgausB} and \eqref{apriori3.3}. Define $\Omega:=\Omega_K(H)$ as in 
\eqref{apriori3.1} and \eqref{apriori3.2}.  Lemma \ref{general-uncertainty} and Corollary \ref{prop:uncert3}
would thus immediately prove   \eqref{brownBand}--\eqref{brownBand2} if $\b g$ were fixed. To guarantee
the bound simultaneously for any $\b g$,
we only need to prove that there exists a  fixed  (large) $K>0$ such that for any $D>0$ we have 
$$
\P\Big(\bigcup_{\|\b g\|\le N^{-c}}\Omega_K(H^{\b g})\Big)\le N^{-D}
$$ 
if $N$ is large enough.
This is just a crude bound on the norm of band matrices which can be proved by many different methods. For example, 
by perturbation theory, we can remove $\b g$ and thus we only need to prove  
$\P\big(\Omega_K(H^{\b 0} )\big)\le N^{-D}$.
This follows easily from the rigidity of the eigenvalues of the matrix $H$ (see  \cite[Corollary 1.10]{AjaErdKru2015}).
\end{proof}

\section{Universality}

\label{sec:universality}

In this section, we  prove the universality of band matrix $H$ (Theorem \ref{Univ}) assuming  the QUE for the band matrices of type  $H^{\b g}$ (Theorem \ref{QUE-Band}).
In the first subsection, we remind some a priori information on the location of the eigenvalues
of the band matrix. The following subsections give details for the mean-field reduction technique previously presented.

\subsection{Local semicircle law for band matrices. } 
We first recall several known results concerning eigenvalues and Green function estimates for  band matrices. 
For $e\in \R$ and $\omega>0$, we define 
 \begin{align}\label{eqn:omega}
  \b S( e,  N; \omega)&=\hb{z=E+\ii\eta \in\C \col |E-e|\le N^{-\omega}\,,\, N^{-1 + \omega} \leq \eta \leq N^{-\omega}},
 \\   \wh{\b S}( e,  N; \omega)&=\hb{z=E+\ii\eta \in\C \col |E-e|\le N^{-\omega}\,,\, N^{-1 + \omega} \leq \eta \leq 1}.
\end{align}
In this section, we are interested only in $ \wh{\b S}$; the other set $\b S$ will be needed later on. 
We will drop the dependence in $N$ whenever it is obvious.  
We  view $\om$ 
as an arbitrarily small number playing few  active roles and we will put  all  these type of parameters 
after semicolon. In the statement below, we will also need $m(z)$, the Stieltjes transform of the semicircular distribution, i.e. 
\be\label{credit}
m(z)=\int\frac{\varrho_{\rm sc}(s)}{s-z}\rd s=\frac{-z+\sqrt{z^2-4}}{2}, 
\ee
where $\varrho_{\rm sc}$ is the semicircle distribution defined in (\ref{eqn:semicir}) and the square root is chosen so that $m$ is holomorphic in the upper half plane and $m(z)\to 0$ as $z\to\infty$. The following results
on the Green function $G(z)=(H-z)^{-1}$ of the random band matrix
$H$ and its normalized trace $m(z)=m_N(z)= \frac{1}{N}\tr G(z)$
 have been proved in \cite{ErdKnoYauYin2013}.

In the following theorem and more generally in this paper, the notation $A_N\prec B_N$ means that for any (small) $\e>0$ and (large) $D>0$ we have
$\mathbb{P}(|A_N|>N^\e|B_N|)\leq N^{-D}$ for large enough $N\ge N_0(\e, D)$. When $A_N$ and $B_N$ depend on a 
parameter (typically on $z$ in some set $\b S$ or some label) then by $A_N(z)\prec B_N(z)$  uniformly in $z\in \b S$ 
we mean that the threshold $N_0(\e, D)$ may be chosen independently of $z$.

\begin{theorem}[Local semicircle law, Theorem 2.3 in \cite{ErdKnoYauYin2013}] \label{thm: with gap}
For the band matrix ensemble defined by (\ref{eqn:band1}) and (\ref{eqn:band2}), satisfying the tail condition \eqref{eqn:subgaus}, 
uniformly in $z \in \wh{\b S}( e,  W; \omega)$ we have 
\begin{align}
\label{Gijest 2}
\max_{i,j} \absb{G_{ij}(z) - \delta_{ij} m(z)} &\prec\,  \sqrt{\frac{\im m(z)}{W \eta}} + \frac{1}{W \eta},\\
\label{m-mest 2}
\absb{m_N(z) - m(z)} &\prec\; \frac{1}{W \eta}.
\end{align}
\end{theorem}

We now recall  the following rigidity estimate of the  eigenvalues for band matrices   \cite{ErdKnoYauYin2013}. 
This estimate was first proved for generalized Wigner matrices in \cite{ErdYauYin2012Rig}  (for our finite band
case this latter result would be sufficient). 
We define the classical location of the $j$-th eigenvalue by the equation 
\be\label{gammadef}
\frac j N = \int_{-\infty}^{\gamma_j} \varrho_{\rm sc} (x) \rd x . 
\ee

\begin{corollary} [Rigidity of eigenvalues, Theorem 2.2 of \cite{ErdYauYin2012Rig} or
Theorem 7.6 in \cite{ErdKnoYauYin2013}] \label{thm:7.1}
Consider the band matrix ensemble defined by (\ref{eqn:band1}) and (\ref{eqn:band2}), satisfying the tail condition \eqref{eqn:subgaus}, 
and $N = 2 p W$ with $p$ finite. Then,  uniformly in $j\in\llbracket 1,N\rrbracket$, we have
\begin{equation}\label{rigidity}
 |\lambda_j-\gamma_j| \prec  \left( \min \big ( \, j ,  N-j+1 \,  \big) \right)^{-1/3}   N^{-2/3}.
\end{equation}
\end{corollary}

\subsection{Mean field reduction for Gaussian divisible band matrices.\ }\label{sec:MF}
Recall the definition of  $H^g$ from \eqref{meq},
\be
   H^g  
  =  \begin{pmatrix} A  & B^* \cr B & D-g \end{pmatrix},
\label{meq1}
\ee
i.e. $A$ has dimensions $  W \times  W $, $B$ has dimensions $(N- W)\times  W $ and
$D$ has dimensions $(N- W)\times (N-  W)$.  Its eigenvalues and eigenvectors are
denoted by $\lambda^g_j$ and $\b\psi^g_j$, $1\leq j\leq N$.

Almost surely, there is no multiple eigenvalue for any $g$, i.e. the curves
$g\to \lambda^g_j$ do not cross, as shown both by absolute continuity argument (we consider Gaussian-divisible ensembles) and by classical codimension counting argument (see \cite[Theorem 5.3]{Colin}). In particular, the indexing is consistent, i.e. 
if we label them in increasing order for $g$ near $-\infty$,
$
  \lambda_1^{-\infty} < \lambda_2^{-\infty} <\ldots < \lambda_N^{-\infty}$,
then the same order will be kept for any $g$: 
\begin{equation}\label{lambdaorder}
\lambda_1^g <\lambda_2^g < \ldots < \lambda_N^g.
\end{equation}

  Moreover, the eigenfunctions are well defined (modulo a phase and normalization)
and by standard perturbation theory, the functions $g\to \lambda_j^g$ and $g\to \b\psi_j^g$
are analytic functions (very strictly speaking 
in the second case these are analytic functions into homogeneous space of the unit ball of $\C^N$
modulo $U(1)$).
Moreover, by variational principle $g\to \lambda_j^g$ are decreasing. In fact, they are
strictly decreasing (almost surely) and they satisfy
\be\label{duduo}
   -1<\frac{\pt \lambda_k^g}{\pt g} < 0
\ee
since by perturbation theory we have
\be\label{pt1}
   \frac{\pt \lambda_k^g}{\pt g}  = - 1 +\sum_{i=1}^{W} \left|\psi_k^g( i) \right|^2  
\ee
and $0<\sum_{i=1}^{W} \left|\psi_k^g( i) \right|^2 <\|\b\psi_k^g\|_2=1$ almost surely. 
We may also assume (generically), that $A$ and $D$ have simple spectrum, and denote their spectra 
$$
  \sigma(A)=\{ \alpha_1 <\alpha_2 <\ldots <\alpha_{W}\}, \qquad  \sigma(D) = \{ \delta_1<\delta_2< \ldots < \delta_{N-W}\}.
$$
We claim the following behavior of $\lambda_j^g$ for $g\to \pm\infty$ (see Figure \ref{fig:sub21}):
\begin{align}
\label{lambdainfty}
\lambda_j^g &= 
      \begin{cases}
     \alpha_j + \OO\Big( |g|^{-1}\Big),  &\ {\rm for}\ j\le W,\\
     -g + \delta_{j-W}   + \OO\Big( |g|^{-1}\Big),  &\ {\rm for}\  W<j\le N, 
\end{cases}
\ {\rm as}\  g\to -\infty,\\
\label{lambdainfty3}
     \lambda_j^g &=   \begin{cases}-g + \delta_{j}   + \OO\Big( |g|^{-1}\Big), &\ {\rm for}\ 
      j\le N-W,\\
     \alpha_{j-(N-W)} + \OO\Big( |g|^{-1}\Big),  &\ {\rm for}\    N-W<j\le N,
  \end{cases}
  \ {\rm as}\  g\to \infty.
  \end{align}
Notice that the order of labels is consistent with \eqref{lambdaorder}.  
The above formulas are easy to derive by simple analytic
perturbation theory. For example, for $g\to -\infty$ and $j\le W$ we
use
$$
   \Big( A - B^* \frac{1}{D-g-\lambda^g}B\Big)\bw_j^g =\lambda_j^g \bw_j^g.
$$
Let $\bq_j$ be the eigenvector  of $A$ corresponding to $\alpha_j$, $A\bq_j=\alpha_j \bq_j$, then we can
express $\bw_j^g = \bq_j + \Delta\bq_j$ and $\lambda_j^g = \alpha_j +\Delta\alpha_j$, plug it into the formula
above and get that $\Delta\bq_j, \Delta\alpha_j = \OO(|g|^{-1})$.

The formulas \eqref{lambdainfty} and\eqref{lambdainfty3} together with the information that the
eigenvalue lines do not cross and that the functions $g\to \lambda_j^g$ are strictly monotone decreasing, give
the following picture.
The lowest $W$ lines, $g\to \lambda_j^g$, $j\leq W$ start at $g\to-\infty$ almost horizontally
at the levels $\alpha_1, \alpha_2, \ldots, \alpha_{W}$ and go down linearly, shifted with $\delta_1, \ldots, \delta_{W}$
 at $g\to \infty$.
The lines $g\to \lambda_j^g$, $W<j\leq N-W$ start decreasing linearly at  $g\to-\infty$, 
shifted with $\delta_1, \delta_2, \ldots, \delta_{N-W}$ (in this order)
and continue to decrease linearly at $g\to\infty$ but shifted with $\delta_{W+1}, \delta_{W+2}, \ldots \delta_{N}$.
Finally, the top lines, $g\to \lambda_j^g$, $N-W< j\leq  N$, start decreasing linearly at  $g\to-\infty$, 
shifted with $\delta_{N-2W+1}, \ldots, \delta_{N-W}$ and become almost horizontal at levels 
$\alpha_1, \alpha_2, \ldots, \alpha_{W}$ for $g\to\infty$.

\begin{figure}[h]
\centering
\begin{subfigure}{.5\textwidth}
  \centering
\begin{tikzpicture}
\node[anchor=south west,inner sep=0] (y) at (0,0) {\includegraphics[width=6cm]{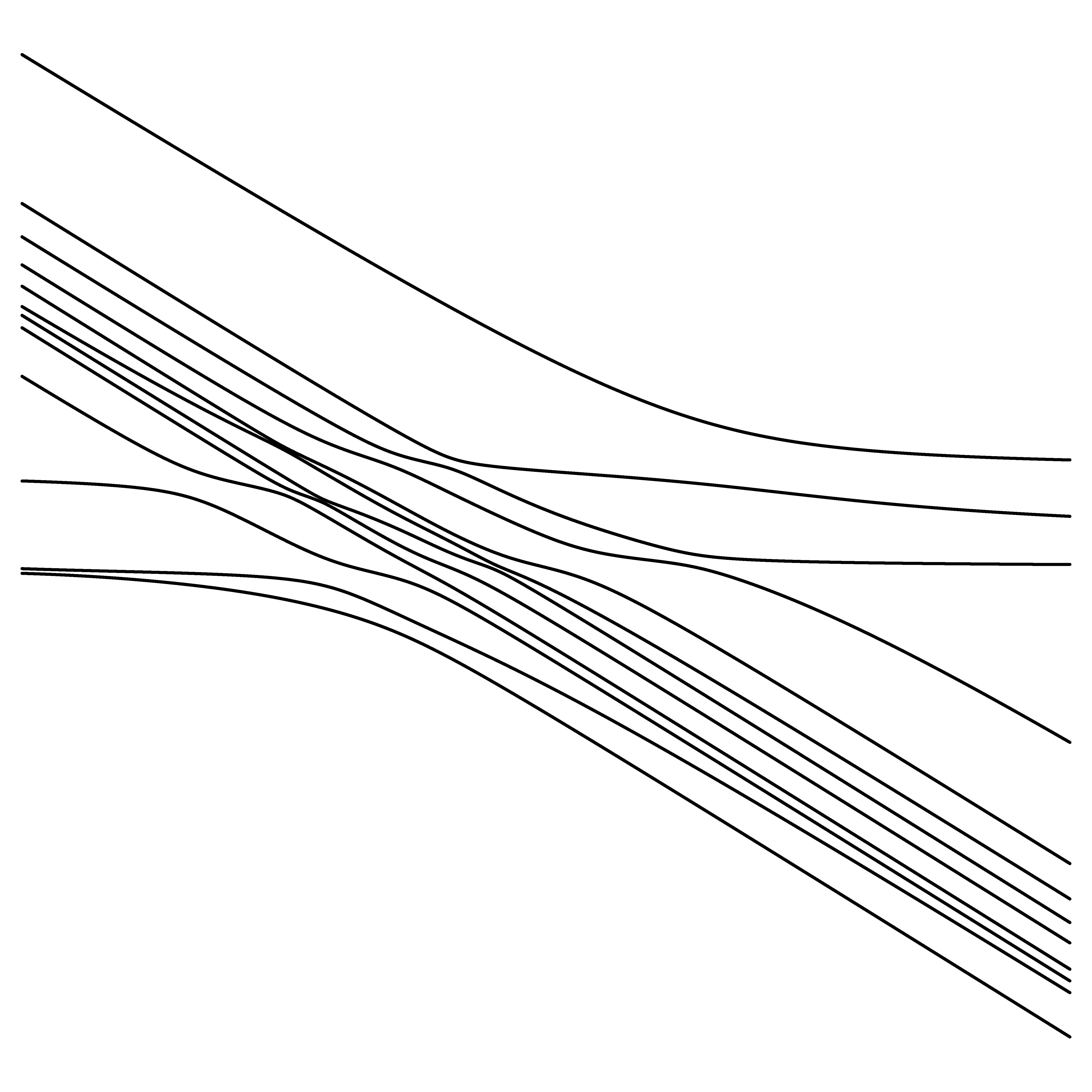}};
\begin{scope}[x={(y.south east)},y={(y.north west)}]
\draw [black,->] (0,0.52) -- (1.02,0.52);
\draw [black,->] (0.47,0.2) -- (0.47,0.8);
\node[black] at (1.07,0.52) {$g$};
\end{scope}
\end{tikzpicture}
  \caption{The maps $g\mapsto \lambda_j^g$, $1\leq j\leq N$.}
  \label{fig:sub21}
\end{subfigure}%
\begin{subfigure}{.5\textwidth}
  \centering
\begin{tikzpicture}
\node[anchor=south west,inner sep=0] (x) at (0,0) {\includegraphics[width=6cm]{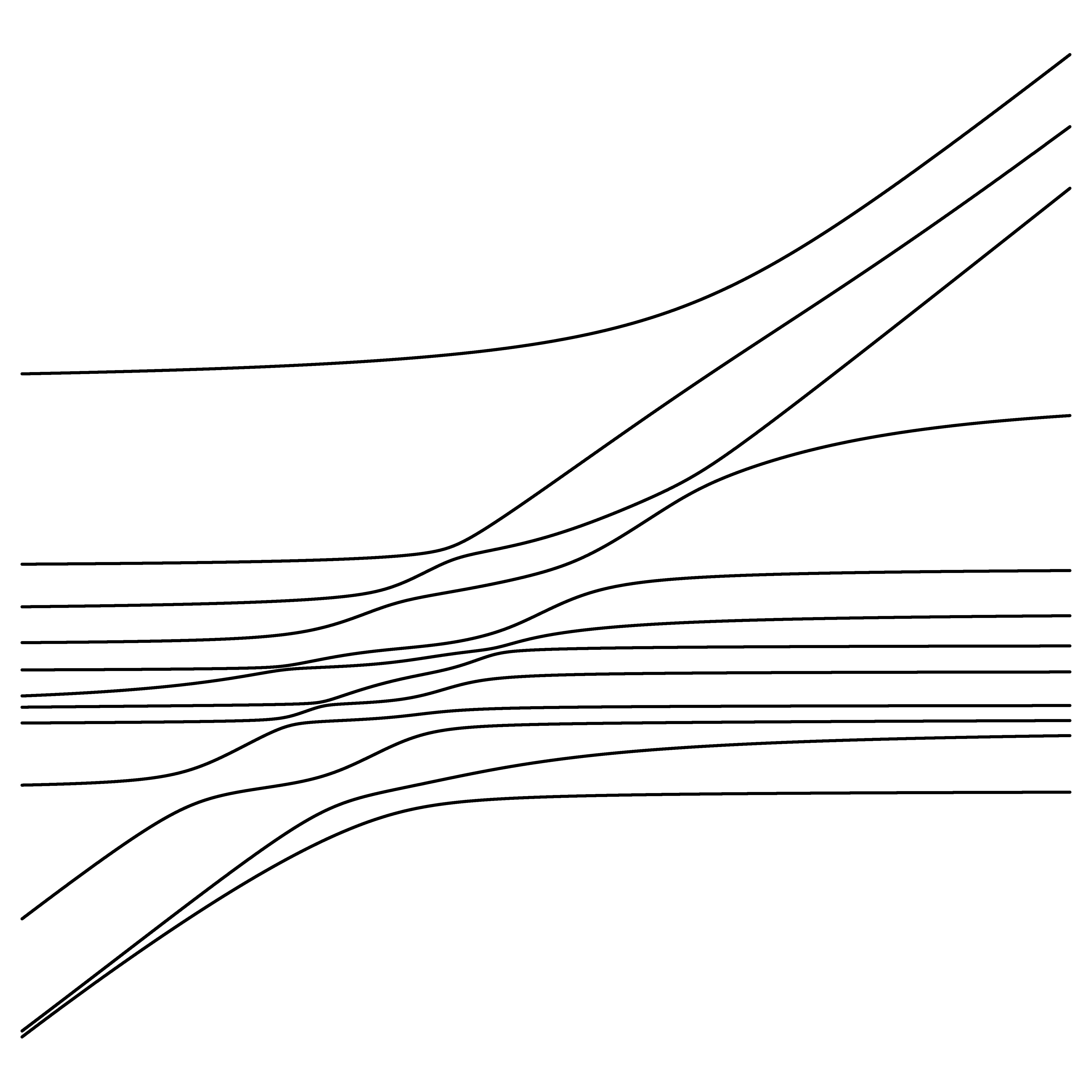}};
\begin{scope}[x={(x.south east)},y={(x.north west)}]
\draw [black,->] (0,0.52) -- (1.02,0.52);
\draw [black,->] (0.47,0.2) -- (0.47,0.8);
\node[black] at (1.07,0.52) {$g$};
\end{scope}
\end{tikzpicture}
  \caption{The maps $g\mapsto \lambda_j^g+g= x_j(g)$, $1\leq j\leq N$.}
  \label{fig:sub11}
\end{subfigure}%
\caption{The eigenvalues of $H^{g}$ (left) and $H^{g}+g\,{\rm Id}$ (right) for $N=12$ and $W=3$.}
\label{fig:test}
\end{figure}
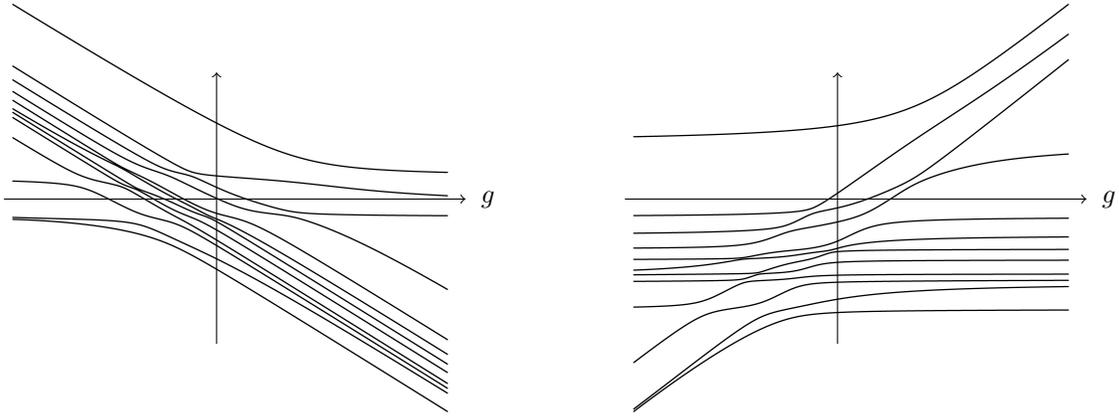

Similarly one can draw the curves $g\to x_j(g):=\lambda_j^g+g$ (see Figure \ref{fig:sub11})
Since $x_j(g)$ is  an  increasing function w.r.t. $g\in \R$ by \eqref{duduo}, with \eqref{lambdainfty}-\eqref{lambdainfty3}, it is easy to check that 
\be\label{ranges}
   \mbox{Ran} \;x_j =\begin{cases}
    (-\infty, \delta_j), &  j\le W,\\
     (\delta_{j-W}, \delta_j), &   W<j\le N-W,\\
     (\delta_{j-W}, \infty), &   N-W<j\le N.
    \end{cases}
\ee
From this  description it is clear that for any $e\not\in\sigma(D)$, 
 the equation $x_j(g)=e$ has exactly $W$ solutions, namely
 \be\label{JW}
   \big\{j: \exists g, \; \mbox{s.t.}\;  x_j(g)=e\big\}=\begin{cases}
   \llbracket  1, W\rrbracket , &\; e<\delta_1,\\
   \llbracket  m+1,m+W\rrbracket , &\;  \delta_m< e< \delta_{m+1},\\
   \llbracket  N-W+1, N\rrbracket , &\;e>\delta_{N-W}.
   \end{cases}
  \ee 
 For any such $j$, the corresponding $g$ is unique by strict monotonicity of $x_j(g)$, thus this function
 can be locally inverted.
Finally, for any $j$, we define the following curves: 
$$
\cal C_j (e)=\lambda_j^g,
\quad s.t. \quad e=x_j(g)=\lambda_j^g+g.
$$
 Their domains are defined as follows:
\be\label{domains}
   \mbox{Dom} \;\cal C_j =\begin{cases}
    (-\infty, \delta_j), &  j\le W,\\
     (\delta_{j-W}, \delta_j), &   W<j\le N-W,\\
     (\delta_{j-W}, \infty), &   N-W<j\le N.
    \end{cases}
\ee
From the definition of $\cal C$ it is clear that these are smooth functions, since they are
compositions of two smooth functions: $g\to \lambda_j^g$ and the inverse of $x_j(g)$.

Finally, by just comparing the definition of $\xi_j(e)$ in  \eqref{Qu}
for $\b g=0$, we know that if $\cal C_k(e)$ exists then it is one of the eigenvalues of $Q_ e$: 
$\cal C_k (e) = \xi_{k'}(e)$ for some $k'$. 
Moreover, we know that almost surely there is

\begin{wrapfigure}[14]{r}{0.5\textwidth} 
\vspace{-0.3cm}
  \centering
\begin{tikzpicture}
\node[anchor=south west,inner sep=0] (x) at (0,0) {\includegraphics[width=7cm]{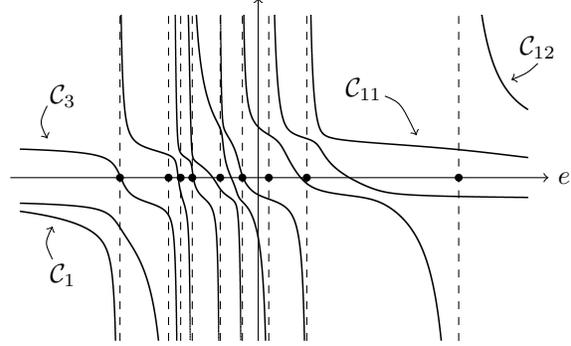}};
\begin{scope}[x={(x.south east)},y={(x.north west)}]
\draw [black,->] (0,0.5) -- (1.02,0.5);
\draw [black,->] (0.47,0) -- (0.47,1.05);
\draw [black,dashed] (0.208,0) -- (0.208,1);
\draw [black,dashed] (0.3,0) -- (0.3,1);
\draw [black,dashed] (0.323,0) -- (0.323,1);
\draw [black,dashed] (0.345,0) -- (0.345,1);
\draw [black,dashed] (0.398,0) -- (0.398,1);
\draw [black,dashed] (0.44,0) -- (0.44,1);
\draw [black,dashed] (0.49,0) -- (0.49,1);
\draw [black,dashed] (0.562,0) -- (0.562,1);
\draw [black,dashed] (0.85,0) -- (0.85,1);
\fill[black] (0.208,0.5)  circle[black,radius=1.5pt];
\fill[black] (0.3,0.5)  circle[black,radius=1.5pt];
\fill[black] (0.323,0.5)  circle[black,radius=1.5pt];
\fill[black] (0.345,0.5)  circle[black,radius=1.5pt];
\fill[black] (0.398,0.5)  circle[black,radius=1.5pt];
\fill[black] (0.44,0.5)  circle[black,radius=1.5pt];
\fill[black] (0.49,0.5)  circle[black,radius=1.5pt];
\fill[black] (0.562,0.5)  circle[black,radius=1.5pt];
\fill[black] (0.85,0.5)  circle[black,radius=1.5pt];
\node[black] at (0.1,0.2) {$ \cal C_1$};
\draw [black,->] (0.08,0.25) .. controls (0.065,0.3) .. (0.08,0.35);
\node[black] at (0.1,0.75) {$ \cal C_3 $};
\draw [black,<-] (0.07,0.62) .. controls (0.055,0.67) .. (0.07,0.72);
\node[black] at (0.67,0.77) {$ \cal C_{11} $};
\draw [black,->] (0.71,0.75) .. controls (0.74,0.73) .. (0.77,0.63);
\node[black] at (1,0.9) {$ \cal C_{12} $};
\draw [black,->] (1,0.85) .. controls (0.98,0.82) .. (0.95,0.8);
\node[black] at (1.05,0.5) {$e$};
\end{scope}
\end{tikzpicture}
\caption{A sample of curves $(\cal C^{\b g}_k)_{1\leq k\leq N}$ for $N=12, W=3$. Bullet points are eigenvalues of $D^{\b g}$.}
\label{Fig13}
\end{wrapfigure}

\noindent no $e$ such that $Q_e$ has multiple eigenvalues (see \cite[Theorem 5.3]{Colin}), so we can assume the curves
$(\cal C_k)_{1\leq k\leq N}$ do not intersect. 
This proves 
\be\label{Yanz}
  \cal C_k (e) = \xi_{k'}(e),    {\rm with}\ k'  = k' (e)= k - \mathcal{N}_D(e)
\ee

\noindent where we defined $\mathcal{N}_D(e)=|\sigma(D)\cap (-\infty,e)|$ the number of eigenvalues of $D$ smaller than $e$. 

The above discussion is summarized as follows, and extended to more general matrices, $A^{\b g}$ and $D^{\b g}$ instead of $A$ and $D$. We stress that the parameter $g$ in the above discussion was an auxiliary variable
and its role independent of the fixed $\b g$ in the definition below.

\begin{definition}[Curves $\cal C^{\b g}_k(e)$]\label{defcurves}  Fix any $\b g\in \R^N$ parameter vector. 
The curves  $\cal C^{\b g}_k (e)$ are the continuous extensions of 
$ \xi_{k-\cal N_{D^{\b g}}(e)}(e)$.  More precisely, we have
\be
  \cal C^{\b g}_k (e) = \xi^{\b g}_{k'}(e), \quad k'=k-\mathcal{N}_{D^{\b g }}(e) 
   \ee
for any $e\not\in\sigma(D^{\b g})$. 
\end{definition}

\noindent The result below shows that the slopes of these curves are uniformly bounded, for ordinates on compact sets.

\begin{lemma}\label{holo}
Consider  any fixed (large)  $K>0$. There exists
a constant  $C_K$ such that for any (small)  $\zeta>0$ and any (large) $D>0$ we have, for large enough $N$,
\be\label{evder}
\P\left(\exists \b g\,: \|\b g\|_\infty\leq N^{-\zeta},\   \sup_{e\not \in \sigma(D), 1\leq k\leq N} \mathds{1}_{| \cal C^{\b g}_k(e)| \le K} 
   \Big| \frac{\rd  \cal C^{\b g}_k}{\rd e} (e) \Big|    \le   C_K \right)\le N^{-D}.
   \ee 
 \end{lemma}

 \begin{proof} We first note that  for $ e \not \in \sigma(D)$, 
\be\label{evd1}
\Big| \frac{ \rd \cal C^{\b g}_k}{\rd e} \Big|= \Big\|  \frac{1}{D^{\b g}-e} B \bu^{\b g}_{k'}(e)  \Big\|^2, \quad k'=k-\mathcal{N}_{D^{\b g }}(e)
\ee
by differentiating \eqref{Qu} w.r.t. $e$ and multiplying it by $\bu^{\b g}_{k'}(e)$. Here we used
  that $k'=k'(e)$ is constant as $e$ varies between two consecutive eigenvalues of $D$.
By  \eqref{Qdef} and \eqref{Qu}, with    $\|\b u^{\b g}_{k'}(e)\|=1$, we have that 
\be\label{bbu}
  \Big\|  B^*\frac{1}{D^{\b g}-e} B \b u^{\b g}_{k'}(e)\Big\| \le \|A^{\b g}\| + |\cal C^{\b g}_{k}(e)|.
\ee
Using Proposition~\ref{apL}  and  that $\|A^{\b g}\|\le C$ holds with high probability,  we have  for any $ e \not \in \sigma(D) $  that
\be\label{evbound}
   \Big\|  \frac{1}{D^{\b g}-e} B \bu^{\b g}_{k'}(e)  \Big\|^2 \le \frac{2}{\mu ^2}  \Big\|  B^*\frac{1}{D^{\b g}-e} B \bu^{\b g}_{k'}(e)\Big\|^2 +  \frac{1}{\mu ^2}
  \le C_\mu(1+|\cal C^{\b g}_{k}(e)|^2)
\ee
for all $\|\b g\|_\infty\leq N^{-\zeta}$, with  high probability,
where in the last step we used \eqref{bbu}.  Together with \eqref{evd1}, we have proved \eqref{evder}.
\end{proof}

The following theorem summarizes the key idea  of the mean-field reduction.

\begin{theorem} \label{trivial-prop}
Let $\theta\in(0,1)$ be fixed. Let $H$ be a Gaussian divisible band matrix  of type 
\be \label{Hgauss0}
H=\sqrt{q}H_{1}+\sqrt{1-q}H_2,  \quad   q=W^{-1+\theta},
\ee
where $H_1, H_2$  are independent  band matrices of width $4W-1$, satisfying
 \eqref{eqn:band1}--\eqref{eqn:band2}, and let $H_1$ have Gaussian entries. 
Recall that $\b \psi_j^g $ is the eigenvector of $H^g$ defined in \eqref{meq}.    Fix an energy $e_0\in(-2,2)$ and let $k$ satisfy  $|\gamma_k-e_0|\le N^{-1 }(\log N)$. 
Suppose that
 all $ \b\psi_j^g $ are flat for $|j-k| \le \log N$ and  $|g|\le  N^{-1+ \zeta }$ for some $\zeta> 0$, in the 
 sense that 
\be\label{jPP0} 
  \sup_{
 \substack{
\text{$j:|j-k|\leq \log N$}\\ 
\hspace{-0.6cm}\text{$|g|\leq N^{-1+ \zeta}$}} 
}  
\ \E \left| \sum_{i=1}^W \left| \psi^g_j (i) \right|^2  - W/N \right| \le   N^{- \zeta  }.
\ee
Then  for any fixed constant $C$ we have (here $\lambda_j=\lambda_j^{g=0}$), for large enough $N$, 
\be\label{holography}
\sup_{j: |j-k| \le C}  \P\left(  \left|\cal C_{j}(e_0)-e_0-\frac{N}{ W}\left(\lambda _j-e_0\right)\right |\ge  N^{-1-\zeta/5}   \right)\le N^{-\zeta/5}.
\ee
\end{theorem}

\begin{proof} 
We will prove  \eqref{holography}  only for $j=k$,
  the general case clearly follows a similar argument. 
Denote by $\cal R$ the set of matrices $H$ such that
$$
  |\lambda_k-e_0|\le N^{-1 + \zeta/2}, \quad {\rm and}
 \quad  
 \sup_{   e \not \in \sigma(D)  }  \mathds{1}_{|\cal C _k(e)| \le 3} 
   \Big| \frac{\rd \cal C _k(e)}{\rd e} \Big|    \le   C.
$$
  By the  assumption $|\gamma_k-e_0|\le N^{-1 }\log N$ and the rigidity of $\lambda_k$ (see Corollary \ref{thm:7.1}),  for any $\zeta>0$ 
   the first condition above holds  with high probability. As guaranteed by  Lemma \ref{holo}, 
   the second condition  in the definition of $\cal R$ holds with high probability
   for a large enough $C$. \nc Hence, for such $\zeta$
    and $C$, for any $D>0$ and large enough $N$ we have
 $\P (\cal R) \ge 1 - N^{-D} $ .

In this proof, we will assume that $\lambda_k >  e_0$ for simplicity of notations. 
In  $\cal R$, we have 
\be\label{iccm}
\Big\{\left(e ,  \cal C_k(e) \right)\; : \;   e \subset [e_0, \lambda_k]  \nc \Big\}\in
\Big[e_0-N^{-1+\zeta/2}, e_0     + \nc N^{-1+\zeta/2}\Big]^2, \qquad \sup_{   e \in [e_0, \lambda_k] \setminus  \sigma(D)  }  
   \Big| \frac{\rd \cal C _k(e)}{\rd e} \Big|    \le   C. 
\ee
Recall that the function  $\cal C_k$  satisfies the relation 
\be\label{xGS}
\cal C_k\left(  g+\lambda^g_k   \right)=\lambda_k^g. 
\ee
Differentiating \eqref{xGS}  at the point $e= g+\lambda^g_k$,  and using  \eqref{pt1},  we have 
\be\label{h1}
\frac {\rd \cal C_k (e)}  { \rd e}  
=\frac{\partial_g\lambda^g_k}{1+\partial_g\lambda^g_k} 
=\frac{ \sum_{i=1}^W \left|\psi^g_k (i) \right|^2-1}{\sum_{i=1}^W \left|\psi^g_k (i) \right|^2 }. 
\ee
  Hence there is a constant $c$ such that in $\cal R$ we have  
\be\label{lb}
\inf_{   e \in [e_0, \lambda_k] \setminus  \sigma(D)  }   \sum_{i=1}^W \left|\psi^g_k (i) \right|^2 \ge c .
\ee
\nc
Since  $\cal C_k(\lambda_k)=\lambda_k$ at $g=0$, we have
 $$
  \cal C_k (e_0)-e_0
  =\int_{\lambda_k}^{e_0} 
\left(\frac {\rd \cal C_k (e)}  { \rd e}-1\right) \rd e
=\int_{e_0}^{\lambda_k} 
\left(\sum_{i=1}^W \left| \psi^g_k (i) \right|^2\right)^{-1} \rd e, 
$$ 
where $e$ and $g$ are related by $g+\lambda^g_k=e$.  Using the above equation,  a simple  calculation  
gives (remember $N=2pW$ and $c$ is defined in (\ref{lb}))
\begin{equation}
\label{36}
 \E \mathds{1}_{\cal R} \left|\cal C_{k}(e_0)-e_0-\frac{N}{ W}\left(\lambda _k-e_0\right)\right |
\le  \frac{2p}{c}  \E  \mathds{1}_{\cal R} \nc  \int_{e_0}^{\lambda_k}    \left|\sum_{i=1}^W \left| \psi^g_k (i)\right|^2-W/N\right| \rd e.
\end{equation}
 The integration domain is  over  $ e=g+\lambda^g_k\in [e_0, \lambda_k]$  with $g=0$ when $e=\lambda_k$, and $g=g_0$ when $e=e_0$ with the $g_0$ that satisfies $e_0=g_0+\lambda^{g_0}_k$. 
  Notice that in the set $\cal R$ we have
$$
\frac {\rd g }  { \rd e} =  \left(\frac {\rd \lambda^g_k }  { \rd g}+ 1 \right)^{-1} =  \left(\sum_{i=1}^W \left| \psi^g_k (i) \right|^2 \right)^{-1} \in[1,c^{-1}],
$$
which implies 
$
|g_0|\le c^{-1} |\lambda_k-e_0 | \le c^{-1}N^{-1 + \zeta/2}$,
i.e., $g$ is in the domain required for using \eqref{jPP0}.  Therefore, we can insert  the estimate \eqref{jPP0} 
into \eqref{36} and conclude that 
$$
 \E \mathds{1}_{\cal R}\left|\cal C_{k}(e_0)-e_0-\frac{N}{ W}\left(\lambda _k-e_0\right)\right |\le  \frac{2p}{c}N^{-1-\zeta/2}.
$$
\nc
This  implies \eqref{holography} and completes the proof of the theorem. 
\end{proof}

\subsection{Proof of Theorem \ref{Univ}. }
  We will first prove  Theorem \ref{Univ} for the class of Gaussian divisible band matrix ensemble, which was defined in \eqref{Hgauss0}.  We will prove general case at the end of this section.
Recall $ Q_e=A- B^*(D-e)^{-1}B$ where, for the Gaussian divisible band matrix ensemble, we 
can decompose $A$ as 
\be\label{ycei}
A=\sqrt{q}A_{1}+\sqrt{1-q}A_2, 
\ee
where $A_1$ and $A_2$ are independent and $A_1$ is a standard $W\times W$ GOE matrix.  For a smooth 
 test function $O$
of $n$ variables with compact support,  define the following observable  of  the rescaled  eigenvalue gaps of $H$:
\be\label{O}
   O_{k,n}(\b \lambda, N):= 
 O \left(  N \rho_{\rm sc} (\lambda_k ) ( \lambda_{k+1}  - \lambda_{k}) ,\dots, N  \rho_{\rm sc} ( \lambda_k ) 
( \lambda_{k+n}  - \lambda_{k+n-1} ) \right)  .
\ee
Our goal is to prove that for some $c>0$,  for any $k\in\llbracket \kappa N,(1-\kappa)N\rrbracket$ we have
$$
(\E^H- \E^{{\rm GOE}_N}) O_{k,n}(\b \lambda, N) \le N^{-c}.
$$
Given $k$,  let the (nonrandom)  energy  $e_0\in(-2,2)$ be such that  $|e_0 -\gamma_k|\le (\log N) N^{-1}$.
 We claim that we can choose $e_0$ with $ |e_0-\gamma_k|\le (\log N)N^{-1} $  such that 
\be\label{DL}
\P(\|(D-e_0)^{-1}\|\ge N^{4})\le N^{-1}.
\ee
To prove this, we note that $\|(D-e_0)^{-1}\|\ge N^4$ is equivalent 
to  $  \min_\ell |\delta_\ell-   e_0|\le N^{-4}$, where, remember, that  the spectrum 
of $D$ is denoted  $\delta_1<\delta_2< \ldots < \delta_{N-W}$.
  For any $\sigma(D)$ fixed, we have the trivial bound 
$$
 \int_{\gamma_k-(\log N)N^{-1}}^{\gamma_k+(\log N)N^{-1}} \mathds{1}_{\min_\ell |\delta_\ell \nc -e_0 \nc |\le N^{-4}}\rd e_0 \le N^{-5/2}.
$$ 
Taking expectation of the last inequality w.r.t the probability law of $D$ and using the Markov inequality, 
we have proved \eqref{DL}. 
 We remark that, by  smoothness of $O$ and by rigidity (Corollary \ref{thm:7.1}), $\rho_{\rm sc} (\lambda_k )$ can be replaced by $\rho_{\rm sc} (e_0)$  in (\ref{O}).

Denote  $\E^{Q_e}$ the expectation w.r.t the law of ${Q_e}$ induced from the distribution of
 the original band matrix  $H$ and let $\b {\xi}(e)=\left({\xi}_1(e), {\xi}_2(e), \ldots, {\xi}_W(e)\right)$ be
  the ordered spectrum of $Q_{e}$.
From the approximate affine transformation between the $\lambda$ and $\xi$ eigenvalues, guaranteed by
Theorem \ref{trivial-prop},  we have  
$$
\E^H  O_{k,n}(\b \lambda, N) =  \E^{Q_{e_0}} O_{k-\al ,n}(\b {\xi}(e_0), W)+ \OO(N^{-c}) 
, \quad \al:= \cal N_D(e_0),
$$
where we used Definition \ref{defcurves}, and the definition 
$$
O_{k ,n}(\b {\xi}(e_0), W):= O \left(  W \rho_\xi (\xi_k ) ( \xi_{k+1}  - \xi_{k}) ,\dots,  W \rho_\xi  ( \xi_k ) 
( \xi_{k+n}  - \xi_{k+n-1} ) \right), \qquad \xi_i = \xi_i(e_0).
$$
Here $\rho_\xi$ denotes the limiting density of the eigenvalues $Q_e$.
We also used  $\rho_{\xi}(e_0)=\rho_{\rm sc}(e_0)$, and that $\rho_\xi$ is smooth
so $\rho_\xi(\xi_k)$ is very close to $\rho_{\xi}(e_0)$ by rigidity, both are
 easy consequences \nc of the local law for $Q_{e_0}$, Theorem \ref{local law}. 
We therefore now need to prove
\begin{equation}\label{37}
\E^{Q_{e_0}} O_{k-\al ,n}(\b {\xi}(e_0), W)-\E^{{\rm GOE}_N} O_{k,n}(\b \lambda, N) =\OO( N^{-c}).
\end{equation}

We now compute the left side of  \eqref{37} by first 
conditioning  on the law of $A_2, B, D$. Theorem \ref{Univ} for Gaussian divisible matrices thus follows from \eqref{DL} and   the following   lemma (proved in the next subsection), 
which asserts the local spectral statistics of the matrix $Q_{e_0}$ are universal.

\begin{lemma}\label{popcorn} Under the assumptions of Theorem \ref{Univ} and  \eqref{ycei}, there exists   $c>0$ such that
$$
\P\left( {\mathds{1}}_{\|(D-e_0)^{-1}\|\le N^4} \left|\E^{A_1} \left(O_{k-\al ,n}(\b {\xi}(e_0), W)\Big| A_2, B, D\right)-
\E^{{\rm GOE}_N} O_{k,n}(\b \lambda, N)
\right|\ge N^{-c}\right)\le N^{-c}.
$$
\end{lemma}

Theorem \ref{Univ} for our band matrices with general entries follows from Lemma \ref{popcorn} and the following comparison result.
 Let $H_t=(H_{ij}(t))$ be a  time dependent flow of  symmetric $N\times N$ matrices
with $H_0=H$ our original band matrix. The 
dynamics of the matrix entries are given by the 
stochastic differential equations
\be\label{eqn:generalDBM}
  \rd H_{ij} (t)  = \frac {\rd \cal B_{ij}  (t) } { \sqrt N}  - \frac{1}{2 N s_{ij}} h_{ij} (t) \rd t,   \quad |i-j| \le 2W,
\ee
where  $\cal B$ is  a symmetric matrix  with 
 $(\cal B_{ij})_{i\leq j}$   a family of independent Brownian motions. By definition,  $H_{ij} (t)  = 0$
for $|i-j| >  2W$. 
The parameter   $s_{ij} > 0$ can take any positive values, but  we choose $s_{ij}$ to be the variance of  $H_{ij}(0)$, i.e., 
$s_{ij} = 1/(4W-1)$. 
Clearly,  for any $t\geq 0$ we have $\E(H_{ij}(t)^2)=s_{ij}$ for all $i, j$ and  thus the variance of the matrix element is preserved
in this flow. This flow is similar to the Dyson Brownian motion but adapted to the band structure. 
For this flow, the following continuity estimate  holds.
 
\begin{lemma}\label{lem:continuity}
Let $\kappa>0$ be arbitrarily small, $\delta\in(0,1/2)$ and $t=N^{-1+\delta}$. Suppose that $W= c N$ for some constant $c$ independent of $N$. 
Denote by  $H_t$ the solution of (\ref{eqn:generalDBM}) 
with initial condition a symmetric band  matrix $H_0$  as defined in (\ref{eqn:band1}), (\ref{eqn:band2}). 
Let $m$ be 
any positive integer and $\Theta:\RR^{m+m^2}\to\RR$ be a smooth function with derivatives satisfying
\begin{equation}\label{eqn:Theta}
\sup_{k\in\llbracket 0,5\rrbracket, x\in\RR^{m+m^2}}|\Theta^{(k)}(x)|(1+|x|)^{-C}<\infty
\end{equation}
for some $C>0$. 
Denote by $(\bu_1(t),\ldots,\bu_N(t))$ the eigenvectors of $H_t$ associated with the eigenvalues $\la_1(t)\leq \dots\leq \la_N(t)$, and $(u_k(t,\alpha))_{1\leq \al\leq N}$ the coordinates of $\bu_k(t)$.
Then there exists $\e>0$ (depending only on $\Theta,\delta$ and $\kappa$) such that, for large enough $N$,
$$
\sup_{I\subset\llbracket \kappa N,(1-\kappa) N\rrbracket, |I|=m = |J|}    \left|
(\E^{H_t}-\E^{H_0})\Theta\left(
(N(\la_k-\gamma_k),N   u_k(\cdot,\alpha)^2 )_{k\in I, \alpha \in J }
\right)
\right|\leq N^{-\e}.
$$ 
 \end{lemma}

The proof of this lemma is identical to that of the Corollary A.2  in \cite{BouYau2013} and we thus omit it. 
Instead of  Lemma \ref{lem:continuity}, the Green function comparison theorem from 
\cite{ErdYauYin2012Univ,  KnoYin2013} could be used as well to finish the proof.

We now complete the proof of Theorem \ref{Univ}. Recall that we have proved this theorem for Gaussian divisible ensembles of the form 
\eqref{Hgauss0}. At any time $t$,  
 the entry $h_{ij}(t)$ of $H_t$ for the flow \eqref{eqn:generalDBM}  is distributed as
\begin{equation}\label{eqn:hij}
e^{-\frac{t}{2Ns_{ij}}}H_{ij}(0)+\left(s_{ij}\left(1-e^{-\frac{t}{N s_{ij}}}\right)\right)^{1/2}\mathscr{N}^{(ij)}, \quad |i-j| \le 2 W, 
\end{equation}
where $(\mathscr{N}^{(ij)})_{i\leq j}$ are independent standard Gaussian random variables.
Hence $H_t$ is Gaussian divisible and bulk  universality holds for $t=N^{-1+\delta}$ with $ \delta$ a small positive number. 
By Lemma \ref{lem:continuity},  the bulk statistics of $H_t$ and $H_0$ are the same up to negligible errors. 
We have thus proved Theorem \ref{Univ}.

\subsection{Universality for mean-field perturbations.\ }\label{subsec:univ-proof} We now prove Lemma \ref{popcorn}.
We first recall a general theorem \cite{LY} concerning gap universality (see \cite{ES} for a related result).  We start from the following definition. 
In the rest of the paper, we fix a small number $\frak a>0$, and define the control parameter 
\begin{align}\label{control}
\varphi= W^{\frak a}.
\end{align}
We will be interested in the deformed GOE defined by 
\be\label{deformG}
\wt H_t = V +   \sqrt t Z,
\ee
where $V$ is a deterministic matrix and $Z$ is a $W\times W$ GOE matrix. 
We now list  the assumptions on the initial matrix $V$ at  some energy level $E_0$; 
 in order to formulate them 
we will need  two $W$-dependent  mesoscopic scales  $\eta_*\ge \varphi/W$ and 
 $r  \ge \varphi \nc\eta_* $.

\begin{assumption}\label{def:v}
Let $\eta_*$ and $r$ be two $W$-dependent parameters, such that $\varphi/W\leq \eta_*\leq r/\varphi\leq 1$.
 We assume that there exist    large positive constants $C_1,C_2$ such that 
 \begin{enumerate}[(i)]
  \item \label{boundnorm} The norm of $V$ is bounded, $\|V\|\leq W^{C_1}$.
  \item \label{locallaw} The imaginary part of the Stieltjes transform of $V$ is  bounded from above and below, i.e., 
  \begin{align}\label{e:imasup}
  C_2^{-1}\leq \Im(m_V(z))\leq C_2, \qquad  m_V(z): = \frac{1}{W} \tr (V-z)^{-1}, 
  \end{align}
  uniformly for any $z\in \{E+ \ii \eta: E\in[E_0-r, E_0+r], \eta_*\leq \eta\leq 2\}$.
 \end{enumerate}
A  deterministic matrix $V  $ satisfying these conditions will be called $(\eta_*, r)$-regular at $E_0$.
\end{assumption}

The following theorem was the main result of \cite{LY} (note that the 
size of the matrix $W$ was replaced by $N$ there).

\begin{theorem}[Universality for mean-field perturbations \cite{LY}]\label{thm:gap}  
 Suppose that $V$ is $  (\eta_*, r)$-regular at $E_0$ and set $T$ such that
  $\eta_* \varphi \le T \le r^2 /\varphi$  with $\varphi= W^{\frak a}$.     
Let $j $ be an index so that the $j$-th eigenvalue of $V$,   
$V_j\in   [E_0-r/3, E_0+r/3]$. Denote the eigenvalues of $\wt H_T$ (defined in \eqref{deformG}) by 
$\b \lambda_T =\{\lambda_{T, i}\}_{i=1}^W$
and let
\be
m_{\wt H_T}(z) = \frac1W\tr (\wt H_T-z)^{-1}.
\ee
Recall the definition of the gap observable $O_{j,n}$ from \eqref{O} for some fixed $n$.
 For  $\frak a$  small enough,  there is a constant  $c >0$ (depending on $C_1, C_2, \frak a$) such that 
\begin{align}
   \E^{  \wt H_T} O_{j,n} \left(\b \lambda_T, W\frac{\rho_T(\lambda_{T,j})}{ \rho_{\rm sc}(\lambda_{T,j}) }\right)- \E^{{\rm GOE}_W} O_{j,n}(\b \lambda, W)=\OO( W^{- c} ), 
\end{align}
where 
$$
\rho_T(\lambda_{T,j}) =  \im  m_{\wt H_T}(\lambda_{T, j}+ \ii \eta), \quad   \eta= 
 T/\varphi.
$$
 Furthermore, for any $\delta> 0$ the following  level repulsion estimate holds: 
\be\label{lr}
\P \left( | \lambda_{T,i} - \lambda_{T,i+1}  | \leq  x /W \right) \leq C_\delta  W^\delta  x^{2-\delta} 
\ee
for any $x>0$ (which can depend on $W$) and for all $i$ such that $\lambda_{T,i} \in   [E_0-r/3, E_0+r/3]$.  \nc
\end{theorem}
 
\noindent The compensating factor $\frac{\rho_T(\lambda_{T,j})}{ \rho_{\rm sc}(\lambda_{T,j}) }$ is due to our definition of
the observable \eqref{O} with a scaling $\rho_{\rm sc}$.

\begin{proof}[Proof of Lemma \ref{popcorn}.] We apply Theorem \ref{thm:gap}  to the matrix
\be\label{yujia}
\wt H_T=Q_{e_0} = \sqrt{q}A_{1}+V  \ {\rm where}\   V=   \sqrt {1-q} \nc A_2 - B^*(D-e_0)^{-1} B,  
\ee
with the following choices: 
\begin{equation}\label{eqn:parameters}
  T=q=N^{-1+\theta},\   \eta_*= N^{ - 1+\theta/2},\ r = N^{-1/2+\theta},\  E_0=e_0, 
   \;  j= k-\al\ \quad (\al=\cal N_D(e_0)),  \;    \lambda_{T, k}=\xi_{k }(e_0), \; C_1 = 5,    
\end{equation}
and   $C_2$ some large constant (in the regularity assumptions on $V$).
 Remember that $\xi_{k }(e_0)$ is the eigenvalue of $Q_{e_0}$ and $\cal N_D(e_0)$ was defined below \eqref{Yanz}.

In order to verify the  regularity assumption of Theorem \ref{thm:gap}, we need a  local law for $Q_{e_0}$, which is 
stated and proved in Theorem \ref{local law}: from 
\eqref{Evden2},  there exists some $c>0$ such that for any $D>0$  we have, for large enough $N$,
$$
\P\left(\forall z=  E+\ii\eta: \;  |E-e_0|\le r  ; \; \eta_*\le\eta\le c, \quad    \frac1W\Im  \tr (V-z)^{-1} \in [c, c^{-1}] \right)\ge 1-N^{-D}.
$$ 
This verifies that part $(ii)$ of the assumption of Theorem \ref{thm:gap}.

Moreover, since the statement of Lemma \ref{popcorn} concerns only  the set  $\|(D-e_0)^{-1}\|\le N^{4}$, together with the fact that $A_2$ and $B$ are bounded with high probability, 
we have in this set 
$$
 \|\sqrt{1-q}A_2 - B^*(D-e_0)^{-1} B\|\le N^{5}
$$
with  high probability.  This verifies that part $(i)$ of the assumption  of Theorem \ref{thm:gap} with $C_1 = 5$.
  
Recall the  mean field reduction  from Section~\ref{sec:MF}. 
 By \eqref{Yanz} and $\cal C_k(\lambda_k)=\lambda_k$, we have
\be\label{comp}
\left|\xi_{j }(e_0)-\gamma_k \right|= \left|\xi_{k-\al }(e_0)-\gamma_k \right|=  \left|\cal C_{k }(e_0)-\gamma_k \right| \le  \left|\cal C_{k}(e_0)-\cal C_k(\lambda_k)\right|+ 
 \left|\lambda_k-\gamma_k\right|
\le N^{-1+\omega} 
\ee
 with probability larger than $1-N^{-D}$ for any small $\omega>0$ and large $D>0$. Here we have used 
 the rigidity of $\lambda_k$, the assumption $|e_0-  \gamma_k|   \le (\log N)N^{-1} $ and the estimate  \eqref{evder} on $(\rd/\rd e) \cal C_{k}(e)$.

Since 
$\xi_{j }= \xi_j(e_0)$ is the $j$-th eigenvalue of   $\sqrt{q}A_{1}+V$ and let $ V_{j }$ be $j$-th eigenvalue of $V$, we have  
$\xi_{j }(e_0)-V_{j }=\OO(\sqrt q N^\omega)$
with probability larger than $1-N^{-D}$.  Therefore with high probability  
$
V_{j }\in[e_0-r/3, e_0+r/3]
$. Hence we can apply Theorem \ref{thm:gap} to get 
\be\label{22}
\P\left(  \left|\E^{A_1} \left(O_{k-\al ,n}\left(\b {\xi}(e_0), W\frac{ \rho_T(   \xi_{k-\al})   }{ \rho_{\rm sc}(\xi_{k-\al})}\right)\Big| A_2, B, D\right)-
\E^{{\rm GOE}_W} O_{k-\alpha,n}(\b \lambda, W)
\right|\ge N^{-c}\right)\le N^{-c}
\ee
 for some $c>0$.
By  \eqref{comp} and smoothness of $\rho_{\rm sc}$,  we can replace $ \rho_{\rm sc}(\xi_{k-\alpha}) $ with  $\rho_{\rm sc}(\gamma_k)$ up to negligible error.
Furthermore, by  the local law \eqref{Evden} we have  for some $c>0$ 
that 
$$
\P\left(   \forall z=E+\ii\eta: \;  |E-e_0|\le N^{-1/2}, \quad \eta= T/\varphi, \quad  \left| \frac1W\Im  \tr (Q_{e_0}-z)^{-1} -\rho_{\rm sc}(e_0) \right|\le N^{-c}\right)\ge 1-N^{-D}.
$$
Therefore, we can replace $\rho_T(\xi_{k-\al})$ by $\rho_{\rm sc}(e_0)$, again up to negligible error. 
With this replacement, \eqref{22} is exactly the statement of  Lemma \ref{popcorn}, after noticing that 
$
\E^{{\rm GOE}_W} O_{k-\alpha,n}(\b \lambda, W)$  converges, as $W\to\infty$, to a limit independent of the bulk index $k-\alpha$.
\end{proof}

\section{Quantum unique ergodicity}
\label{sec:QUE}

In this section, we prove  Theorem \ref{QUE-Band}, in particular we check that the assumption of Theorem \ref{trivial-prop} concerning the flatness of eigenvector holds.
The following lemma implies the assumption  \eqref{jPP0} by choosing $\b a(i)=1$ for all $1\le i\le W$,
$\b g = (g_1, \ldots, g_N)$ with $g_i = g \1_{i>W}$ 
 and noticing  $0\le \sum_{i=1}^W  \left| \psi^g_j (i) \right|^2\le 1$. 
 We will prove this lemma after completing the proof of Theorem \ref{QUE-Band}.

 \begin{lemma}[Quantum unique ergodicity for Gaussian divisible band matrices]\label{QbM}
Recall  that $\b\psi_k^{\b g}=\begin{pmatrix} \b w_k ^{\b g} \cr \b p_k^{\b g}\end{pmatrix}$  is the $k$-th eigenvector  of  $ H^{\b g} $
with eigenvalue  $ \lambda_k^{\b g} $. Suppose that \eqref{Hgauss0} holds.  Let
$\kappa>0$ be fixed. There exists $\e,\zeta>0$ such that for any   $k\in \llbracket  \kappa N, (1-\kappa ) N\rrbracket$, $\b a\in [-1,1]^W$ and $\delta>0$ we have
\be\label{jPP} 
 \sup_{ \|{\b g}\|_\infty \le N^{-1+  \zeta } }  
 \P\left( \left| \sum_{i=1}^{W} \b a(i) \left(|w_k ^{\b g} (i)| ^2  - \frac1N\right) \right| \ge \delta\right)\le 
    C_\kappa N^{-\e}/\delta^2.
\ee
\end{lemma}

\begin{proof}[Proof of Theorem \ref{QUE-Band}.]  We will first prove  Theorem \ref{QUE-Band} for the class of Gaussian divisible band matrix ensemble, which was defined in \eqref{Hgauss0}.    With \eqref{jPP}, we know that there exists $\zeta,\e>0$ such that for any $k\in\llbracket\kappa N,(1-\kappa)N\rrbracket$, $\b a \in[-1,1]^N$, $m\in \llbracket0, N/W-1\rrbracket$, $\|\b g\|_\infty<N^{-1+\zeta}$,  and $\delta>0$, we have 
$$\P\left( \left| \sum_{i= 1}^{ W} \b a(i+mW) \left(|\psi_k ^{\b g} (i+mW)| ^2  - \frac1N\right) \right| \ge \delta\right)\le C_\kappa 
N^{-\e}/\delta^2.
$$
 Then summing up $m\in \llbracket0, N/W-1\rrbracket=0,1, \ldots 2p-1$, we have proved Theorem \ref{QUE-Band} in the case of Gaussian divisible band matrix.  For the general case, we consider  $\b g=0$ for simplicity, without loss of generality.  Recall the definition of $H_t$ in \eqref{eqn:generalDBM}. With \eqref{localque0} for any Gaussian divisible band matrix, we know that for some $\e>0$, 
\be\label{zuoxiu}
 \E^{H_t}\left| \sum_{i=1}^{N} \b a(i) \left(|\psi_k(i)|^2  -\frac1N\right) \right| ^2\le C_{\kappa,p}  N^{-\e}.
\ee
Then comparing $H=H_0$ with $H_t$ using Lemma \ref{lem:continuity}, we have
$$
\left|\left( \E^{H_t}-\E^{H_0}\right)  |\psi_k(i)| ^2\right| \leq C_\kappa  N^{-1-\widetilde\e}, \ \left|\left( \E^{H_t}-\E^{H_0}\right)  |\psi_k(i)| ^2  |\psi_k(j)| ^2\right| \leq C_\kappa  N^{-2-\widetilde\e},  
$$
for some $\widetilde\e>0$ and for any $i, j$.  Together with \eqref{zuoxiu}, we therefore proved
$$\E^{H_0}\left|\sum_{i=1}^{N} \b a(i) \left(|\psi_k ^{\b g} (i)| ^2  -\frac1N\right) \right| ^2\le C_{\kappa, p} ( N^{-\e}+N^{-\widetilde\e}),
$$
which implies the desired result \eqref{localque0} by Markov's inequality.
\end{proof}

We now  prove Lemma \ref{QbM}.    Recall  the notations  in \eqref{H}-\eqref{Qu}, i.e. that 
$\b u^{\b g}_j(e)$, ($e\in \R$ and $j\in \llbracket1, W\rrbracket$)  is  a (real) eigenvector of the matrix 
\be\label{QV}
 Q^{\b g}_e =  A^{\b g}  - B^* (D^{\b g}-e)^{-1}B =  \sqrt{q}A_{1}+V^{\b g},  \quad   V^{\b g}=
 \sqrt {1-q} \nc A_2 +\sum_{1\le i\le W}g_i\b e_i\b e_i^*- B^*(D^{\b g}-e)^{-1} B.
\ee
Note that not only $A$ has a Gaussian divisible decomposition \eqref{ycei} but also $B$ and $D$, however
this latter fact is irrelevant and we will not follow it in the notation.
With the labeling of eigenvalue convention in \eqref{Yanz},  we have  the following relation between $\bu ^{\b g}$ and $w^{\b g}$. 
\be\label{bowl}
\b u^{\b g}_{\hat k}\left( \lambda_k^{\b g}\right)=\frac{w_k^{\b g}}{\|w_k^{\b g}\|}, \quad {\hat k}:=k'(\lambda_k^{\b g})=k-\cal N_D(\lambda_k^{\b g}).
\ee  
To prove Lemma \ref{QbM},
 we first claim that the following QUE  for $\b u^{\b g}_{\hat k}\left( \lambda_k^{\b g}\right)$ holds. 
 The challenge is that we  consider the matrix $Q^{\b g}_e$ with a random shift $e$, namely $e=\lambda_k^{\b g}$,
 and the index $\hat k$ is also random.

\begin{lemma}[Quantum unique ergodicity for mean-field matrices with random shift $e$]\label{QsM} Let $\kappa>0$ be fixed. 
Under the assumption of Lemma \ref{QbM} and \eqref{ycei}, there exists $\e,\zeta>0$ such that for any $k\in\llbracket\kappa N,(1-\kappa)N\rrbracket$, $\b a\in[-1,1]^W$ and $\delta>0$ we have ($[\bx]_i$ denotes  the $i$-th component of a vector $\bx$)
\be\label{41}
\sup_{\|{\b g}\|_\infty \le N^{-1+  \zeta } }  
\P\left( \left| \sum_{i=1}^{W} \b a(i)  \left(\Big[\b u^{\b g}_{\hat k}\left( \lambda_k^{\b g}\right)\Big]_i ^2  - \frac{1}{W}\right) \right| \ge \delta\right)\le C_\kappa 
N^{-\e}/\delta^2, \qquad {\hat k}:=k'(\lambda_k^{\b g})=k-\cal N_D(\lambda_k^{\b g}).
\ee 
 \end{lemma}
\begin{proof}[Proof of Lemma \ref{QbM}] Clearly, to deduce  \eqref{jPP} from \eqref{41},  one only needs to show that there exists $\widetilde\e>0$ such that 
\be\label{ngza} 
\sup_{ \|{\b g}\|_\infty \le N^{-1+  \zeta}}
 \P\left( \left|\sum_{1\le i\le   W} \psi_k^{\b g}(i)^2-\frac{W}{N}\right|\ge N^{-\widetilde\e}\right)\le C_\kappa 
    N^{-\widetilde\e }.
 \ee
To see this, we  first note that by choosing 
$
{\b a}(i)={\mathds{1}}_{i\le W/2}-\mathds{1}_{i> W/2},
$
and $\delta=N^{-\e/10}$ in \eqref{41}, we have 
$$
 \P\left( \left|\sum_{1\le i\le   W/2} \psi_k^{\b g}(i)^2-\sum_{W/2< i\le   W} \psi_k^{\b g}(i)^2\right|\ge N^{-\e/10}\right)\le 
 C_\kappa 
    N^{-\e/10}.
$$
In the above equation, the index set $\llbracket1,W \rrbracket$ which determines the decomposition \eqref{H} can be replaced by $\llbracket1+ n   W/2,W+ n   W/2 \rrbracket $      with $n\in\llbracket 0, 2(N/W-1)\rrbracket$. By a simple union bound, we can assume all these bounds hold 
simultaneously.  In particular, the local $\ell^2$-norms of $\psi_k^{\b g}$ on each consecutive $W/2$ batches
of indices  coincide approximately.
As $\b\psi_k^{\b g}$ is normalized, all these local norms are close to $W/(2N)$, which  implies \eqref{ngza} and completes the proof. 
\end{proof}

In Lemma \ref{QsM},   the energy $\lambda_k^{\b g}$ is random  and the  index includes a random shift $\cal N_{D^{\b g}}(\lambda_k^{\b g})$.
To prove Lemma \ref{QsM}, we need the following lemma (proved at the end of this section) which replaces the random parameter $\lambda_k^{\b g}$ 
in \eqref{41} by a deterministic one.

\begin{lemma}[Quantum unique ergodicity for mean-field matrices with fixed shift $e$]\label{QsM2} 
Let $\kappa>0$ be fixed. Under the assumption of Lemma \ref{QbM} and \eqref{ycei}, there exists $\e,\zeta>0$ such that for any $k\in\llbracket\kappa N,(1-\kappa)N\rrbracket$, 
$|e-\gamma_k| \le N^{-1+ 2 \zeta}$, $\|{\b g}\|_\infty \le N^{-1+  \zeta}$,
$\b a\in[-1,1]^W$, and $\delta>0$,
we have 
$$
\P\left(\exists j: |j-k'|\leq N^\zeta,\ \left| \sum_{i=1}^{W} \b a(i)  \left(\left[{\b u}^{\b g}_{j}\left( e\right)\right]_i^2  - \frac1W\right) \right| \ge \delta\right)\le 
C_\kappa N^{-\e}/\delta^2
$$
where $k'=k'(e)= k-\cal N_{D^{\b g}}(e)$.
\end{lemma}

\begin{proof}[Proof of Lemma \ref{QsM}] 
Since $ | \lambda_k^{\b g} - \lambda_k | \le  \|{\b g}\|_\infty$ and the rigidity estimate holds for $\lambda_k$ (see \eqref{rigidity}), with high probability we have
\be\label{rigig}
\lambda_k^{\b g}-\gamma_k=\OO(N^{-1+\zeta})
\ee
for any  $\|{\b g}\|_\infty \le N^{-1+  \zeta}$.

We discretize the set of the parameter $e$.
Denote $e_m=\gamma_k+mN^{-1- \zeta}$. For small enough  $ \e,\zeta>0$, for any fixed $\b a$ and $\|\b g\|_\infty$ as in the assumptions of Lemma \ref{QsM2}, we thus have, from this lemma (and a union bound),
\be\label{YPz}
\P\left(\exists m\in \Z, \exists j:\ |m|\le N^{3\zeta},\  |j-k'(e_m)|\leq N^\zeta,\ \left| \sum_{i=1}^{W} \b a(i)  \left(\left[ {\b u}^{\b g}_{j}\left( e_m\right)\right]_i^2  -\frac{1}{W}\right) \right| \ge \delta\right)\le 
   C_\kappa N^{-\e}/\delta^2.
\ee
 Using   \eqref{rigig},   we have  with high probability 
that there exists a random integer $|\wt m|\le N^{3\zeta}$  such that  \nc
\be\label{wtm}
\left|e_{\wt m}-\lambda_k^{\b g}\right|\le N^{-1-\zeta}.
\ee
Defining  the $W\times W$ matrix  $J$ by $J_{ij}:={\b a}(i) \delta_{ij}$ and setting  $e^*:=\lambda_k^{\b g}$, we have 
\be\label{zap}
\sum_i {\b a}(i)\left[\b u^{\b g}_{\hat k}\left( \lambda_k^{\b g}\right)\right]_i^2
=\left({\b u}^{\b g}_{\hat k}\left( e^* \right), J{\b u}^{\b g}_{\hat k}\left(e^* \right) \right)
=\left({\b u}^{\b g}_{k'(e_{\wt m})}\left(e_{\wt m} \right), J {\b u}^{\b g}_{k'(e_{\wt m})}\left(e_{\wt m} \right) \right)
+  \int _{e_{\wt m}}^{e^*} \frac{\rd }{\rd e}\Big({\b u}^{\b g}_{k'(e)}(e), J  {\b u}^{\b g}_{k'(e)}(e) \Big)\rd e.
\ee
From (\ref{YPz}) and (\ref{wtm}),
$$
\P\left( \left| \sum_{i=1}^{W} \b a(i)  \left(\left[{\b u}^{\b g}_{k'(e_{\wt m})}\left(e_{\wt m} \right)\right]_i^2  -\frac{1}{W}\right) \right| \ge \delta\right)\le 
   \widetilde C_\kappa N^{-\e}/\delta^2.
$$
We therefore just need to bound the second term on the right hand side of (\ref{zap}).
A simple calculation yields (we now abbreviate $k'=k'(e)$ and similarly  $\ell'=\ell'(e)=\ell - \cal N_D(e)$) 
$$
  \frac{\rd }{\rd e}\Big( {\b u}^{\b g}_{k'}(e), J \,  {\b u}^{\b g}_{k'}(e) \Big)
 =   2  \sum_{\ell \ne k}\frac{\Big({\b u}^{\b g}_{k'}(e), J \,   {\b u}^{\b g}_{\ell'} (e) \Big)}{\cal C^{\b g}_{k}(e)-\cal C^{\b g}_\ell (e)}
\left({\b u}^{\b g}_{\ell'}  (e),  B^* \frac{1}{(D^{\b g}-e)^2}B \;  {\b u}^{\b g}_{k'}(e)\right).
$$
Together with $\|\b a\|_\infty\leq 1$, this gives
 $$
 \left|  \frac{\rd }{\rd e}\Big({\b u}^{\b g}_{k'}(e), J \, {\b u}^{\b g}_{k'}(e) \Big)\right|
  \le  \sum_{\ell \ne k}
  \frac{C}{|\cal C^{\b g}_{k}(e)-\cal C^{\b g}_\ell (e)|}  \Big\| \frac{1}{D^{\b g}-e} B {\b u}^{\b g}_{\ell'}(e) \Big\|   \Big\| \frac{1}{D^{\b g}-e} B {\b u}^{\b g}_{k'}(e) \Big\|.
 $$
  By  \eqref{evbound}, 
  for all $ e \not \in \sigma(D^{\b g}) $, we can bound $ \Big\| \frac{1}{D^{\b g}-e} B {\b u}^{\b g}_{\ell'}(e) \Big\| $ by 
$C (1+|{\cal C}^{\b g}_{\ell'} (e) |)$ 
 with high probability.  Since  for $e\in [e_1, e_{N^{3\zeta}}] $, 
$\cal C^{\b g}_k(e) =\OO(1)$ with high probability, we have 
$$
\left|  \frac{\rd }{\rd e}\Big({\b u}^{\b g}_{k'}(e), J \, {\b u}^{\b g}_{k'}(e) \Big)\right|
\le  C\sum_{\ell\ne k}    \frac{C(1+|\cal C^{\b g}_\ell (e) |)}
{|\cal C^{\b g}_k(e)-\cal C^{\b g}_\ell (e)|} , \qquad \ e\in [e_1, e_{N^{3\zeta}}] \setminus \sigma(D^{\b g}),
$$
with high probability. 
Using  \eqref{evder} and \eqref{Evden3} in Theorem \ref{local law} with $t=q$ (note that,
with the notations of Theorem \ref{local law}, we have
 $\cal C^{\b g}_k(e)=\xi^{\b g}_{k'}(e,q,q)$, $k'=k-\cal N_{D^{\b g}}(e)$  and $ \xi_{k}^{\b g}(e,q,t)$ is  the $k$-th eigenvalue of $Q ^{\b g}_e(t, q)$, which is defined in  \eqref{defqbt}),   we have
$$
\sum_{\ell: |\ell - k|\ge  2 N^{2\om}    }
    \frac{C(1+|\cal C^{\b g}_\ell (e) |)} {|\cal C^{\b g}_k(e)-\cal C^{\b g}_\ell (e)|}   \le  C N^{1 +
    3\om},   \qquad {e\in [e_1, e_{N^{3\zeta}}] \setminus  \sigma(D^{\b g})}
$$
with high probability  for any small $\om$.
We have thus proved that with high probability
$$
\left|  \frac{\rd }{\rd e}\Big( {\b u}^{\b g}_{k'}(e), J \, {\b u}^{\b g}_{k'}(e) \Big)\right| \le \sum_{\ell: |\ell - k|\le  2 N^{2\om}   }
  \frac{  1}{|\cal C^{\b g}_k(e)-\cal C^{\b g}_\ell (e)|}  + C N^{1+3\om}
$$
for any $  e\in[e_1, e_{N^{3\zeta}}]\setminus \sigma(D^{\b g})$.
Inserting the last equation  into \eqref{zap}, using \eqref{wtm},
the ordering of the curves $\mathcal{C}^{\b g}_k$ in $k$, and choosing $\omega\leq \zeta/10$, we obtain that with high probability,  
\be\label{zap2}
\left|\left({\b u}^{\b g}_{k'}\left(\lambda_m^{\b g} \right), J{\b u}^{\b g}_{k'}\left(\lambda_m^{\b g} \right) \right)
-\left({\b u}^{\b g}_{k'}\left(e_{\wt m} \right), J{\b u}^{\b g}_{k'}\left(e_{\wt m} \right) \right)\right|
\le C  N^{ 2 \om} \nc  \int _{e_{\wt m}}^{e^*}   \sum_{\ell=k\pm1}\frac{  1}{|\cal C^{\b g}_k(e)-\cal C^{\b g}_\ell(e)|}\rd e +N^{  -\zeta/2}.
\ee
 By H\"older's inequality, we have 
\begin{align}\label{weakz}
  \E   \left|  \int _{e_{\wt m}}^{e^*} 
 \frac1{|\cal C^{\b g}_k(e)-\cal C^{\b g}_\ell(e)|}\rd e \right|  
  &\le  \left(\E  \left|  \int_{e_{\wt m}}^{e^*}  \rd e \right|\right)^{1/3}  \left( \E     \int_{e_{\wt m}}^{e^*}       \frac{1 }{ |\cal C^{\b g}_k(e)-\cal C^{\b g}_\ell(e)|^{3/2}}  \  \rd e \right)^{2/3}  \\\nonumber
& \le  C N^{-\zeta- 1/3}    \left(  N^{-1+\zeta}\max_{e: |e-\gamma_k|\le N^{-1+{ 2}\zeta}}\E     \left| \cal C^{\b g}_k(e)-\cal C^{\b g}_\ell(e)\right|^{-3/2}\right)^{2/3}.
\end{align} 

As in the proof of Lemma \ref{popcorn}, we  apply Theorem \ref{thm:gap}  to the operator $ Q^{\b g}_e$ in \eqref{QV}. 
We can  similarly verify  that Assumption \ref{def:v} 
holds with high probability and thus  the  level repulsion estimate \eqref{lr} holds. 
Since $|\b g|\ll N^{-1/2}\ll r$ ($r$ is chosen as in (\ref{eqn:parameters})),  we have  $V^{\b g}_k\in [e-r/3, e+r/3]$ for any $k$ such that  $|e-\gamma_k|\le N^{-1+2\zeta}$. \nc
Thus     for any small $\om$ we have
$$
 \max_{e: |e-\gamma_k|\le N^{-1+2\zeta}}\E     \left|  \cal C^{\b g}_k(e)-\cal C^{\b g}_\ell(e)\right|^{-3/2}  \le  C_\om N^{3/2+\om}, \qquad  \ell = k\pm 1.
$$
Together with  the Markov inequality, \eqref{weakz} and \eqref{zap2}, this concludes the proof of  Lemma \ref{QsM}. 
\end{proof}

\begin{proof}[Proof of Lemma \ref{QsM2}.]
 We will need  the local QUE from \cite{BouHuaYau2016}.
 Remember the notations from Subsection \ref{subsec:univ-proof} and the  control parameter $\varphi= W^{\frak a}=cN^{\frak a}$. 
Let $\bu_1(t), \ldots, \bu_W(t)$ 
be the real eigenvectors for the matrix $\wt H_t$ defined in \eqref{deformG}
and let $u_j(i,t)$ be the $i$-th component of $\bu_j(t)$.
The following 
result is the content of Corollary 1.3 in \cite{BouHuaYau2016}.

  \begin{theorem}[Quantum unique ergodicity for mean-field perturbations \cite{BouHuaYau2016}]\label{c:que}
 We assume the initial matrix $\wt H_0= V$  satisfies  Assumption \ref{def:v} in Subsection \ref{subsec:univ-proof}. 
 We further assume that there exists a
small  constant $ \frak b$ such that
 \begin{align} \label{e:delocalize}
 \big |  (\wt H_0-z)^{-1}_{ij} - m_0(z) \delta_{ij}  \big | \leq    \frac{1}{W^{\frak b}},    \ {\rm with}\ m_0(z) = \frac 1 N \tr (\wt H_0-z)^{-1}, 
\end{align}
uniformly in $\{z=E+\ii \eta: E\in[E_0-r, E_0+r], \eta_*\leq \eta\leq r \}$ 
with $E_0, \eta_*$ and $r$ as in Assumption  \ref{def:v}.
Then the following quantum unique ergodicity holds: for any $ \mu>0$ there exists $\e,C_\mu > 0$ (depending also on $\frak a, \frak b$
and $C_2$ from Assumption  \ref{def:v}) such that for any $T$ with $\varphi \eta_* \le  T \le r/ \varphi $, $\b a\in[-1,1]^W$, 
and  $\delta> 0$, we have
\be\label{localque}
\sup_{ j: |\lambda_{T,j}-E_0| < {(1-\mu)}r }  \P\left(\left|\frac{1}{\|{\bf a}\|_1}\sum_{i=1}^W \b a(i)(Wu_j^2(i, T)-1) \right|>\delta \right)\leq C_\mu\ \left(W^{-\e}+\|{\bf a}\|_1^{-1}\right)/\delta^{2}.
\ee
\end{theorem}

\noindent We now return to the proof of Lemma \ref{QsM2}.   
Theorem \ref{c:que} implies in particular that
\be\label{butterfly}
 \sup_{  j  : |\lambda_{T,j}-E_0| < {(1-\mu)}r }  \P\left(\left| \sum_{i=1}^N \b a(i) \left( u_j^2(i, T )-N^{-1}\right) \right|>  \delta\right)\le C\ N^{-\e}/\delta^2.
\ee 
Similarly to \eqref{yujia}, we apply Theorem \ref{c:que}
 to the matrix
$$
\wt H_T=Q^{\b g}_{e} = \sqrt{q}A_{1}+V  \ {\rm where}\   V=   \sqrt {1-q} A_2+\sum_{1\leq i\leq W}g_i\b e_i\b e_i^* -
 B^*(D^{\b g}-e)^{-1} B,  
$$
with the following choices: 
$$
  T=q=N^{-1+\theta},\   \eta_*= N^{ - 1+\theta/2},\ r = N^{-1/2+\theta},\  E_0=e,
$$
 In particular, the supremum in \eqref{localque} will cover all indices $j$ such that 
$|j- (k-\cal N_D^{\b g}(e))|\leq N^\zeta$ and recall that  $u_j(i,T)=u_{j}^{\b g}(i)$ for such $j$.
Using the results from the next section,   both requirements of Assumption~\ref{def:v} hold
for our $V$, in particular 
 \eqref{e:imasup} 
 is satisfied by \eqref{Evden2} for $q=t=0$.
 Moreover  \eqref{e:delocalize} holds \nc by \eqref{Evden}.
  Hence the assumption for Theorem \ref{c:que} are verified. Therefore with  \eqref{butterfly}  we obtain that there exists $\e>0$ such that for any $\delta$, 
\begin{equation}\label{eqn:ending}
 \sup_{ \ell : |\cal C^{\b g}_{\ell }(e)-e| < {(1-\mu)}N^{-1/2+\theta} }  \P\left(\left| \sum_{i=1}^W \b a(i)  \left(  \left[ {\b u}^{\b g}_{\ell'}\left( e\right)\right]_i^2-N^{-1}\right) \right|>  \delta\right)\le CN^{\e}/\delta^2,\ \quad \ell' = \ell - \cal N_{D^{\b g}}(e).
\end{equation}
Moreover for any index  $k$ satisfying $|e-\gamma_k|\leq N^{-1+2\zeta}$ we have $|\cal C^{\b g}_k(e)-e| < (1-\mu)N^{-1/2+\theta}$. Indeed, with the rigidity property \eqref{rigidity} and the trivial
perturbation estimate $|\lambda^{\b g}_k-\lambda_k|\le \|\b g\|_\infty$, we know that 
$$
|\lambda^{\b g}_k-e|\le |\lambda^{\b g}_k-\lambda_k|+ |\lambda_k-\gamma_k|+|\gamma_k-e|\le CN^{-1+\zeta}.
$$
By definition,  $\cal C^{\b g}_k(\lambda^{\b g}_k)=\lambda^{\b g}_k$. Hence together with  \eqref{evder},  we have
$
|\cal C^{\b g}_k(e)-e|\le CN^{-1+\zeta}
$
with high probability.

 Finally, after choosing such a $k$ satisfying
  $|e-\gamma_k|\leq N^{-1+2\zeta}$, for any $j$ such that $|j-k'(e)|\leq N^\zeta$ we have $j=\ell'(e)$ for some $\ell$.
  Moreover, $|\cal C^{\b g}_\ell(e)-e|\le 
 |\cal C^{\b g}_\ell(e)-C^{\b g}_k(e)|+CN^{-1+\zeta}\le 
 CN^{-1+\zeta}$, so that we can apply (\ref{eqn:ending}). This concludes the proof of Lemma \ref{QsM2} by a simple union bound over all $j$'s such that $|j-k'(e)|\leq N^\zeta$.
\end{proof}

\section{Local law}
\label{sec:local}

The main purpose of this section is to prove  the local law of the Green's function of $H^{\b g}$, $Q^{\b g}_e$ and some variations of them  (recall the notations from Section~\ref{subsec:sketch}). As we have seen in the previous sections, 
 these local laws  are the basic inputs \nc for proving universality and QUE of these matrices.

\begin{theorem} [Local law  for $Q$ \nc]\label{local law}  Recall $\b S( e,  N; \omega)$, $\wh{\b S}( e,  N; \omega)$ and $m(z)$ defined in \eqref{eqn:omega}-\eqref{credit}.  We fix a vector ${\b g}{ \in \R^N}$ with $\|\b g\|\le N^{-1/2}$, 
numbers  $0\le t\le q\le N^{-1/2} $, a  positive  $N$-independent \nc threshold $\kappa>0$ and any energy
 $e$ with $ |e|\le 2-\kappa$.   Set  
\be\label{defqbt}
Q^{\b g}_e(t,q):=   \sqrt t \nc A_1+  \sqrt {1-q} \nc A_2  - \nc
\sum_{1\le i\le W}g_i\b e_i\b e_i^*- B^*\frac{1}{D^{\b g}-e} B.
\ee
For any (small) $\omega>0$ and  (small)  $\zeta>0$ and (large) $D$, we have  
\be\label{Evden}
\mathbb P\left(\exists z\in\b    S( e,  N; \omega)   \; s.t.\;  \max_{ij}\left|\left[Q ^{\b g}_e(t, q)-z\right]^{-1}_{ij}-m(z)\delta_{ij}\right|\ge N^\zeta\left( (N\eta)^{-1/2}+ |z-e|\right)\right)\le N^{-D},
\ee
and there exists $c>0$ such that 
\be\label{Evden2}
\mathbb P\left(\exists z\in\wh{\b S}( e,  N; \omega)   \; s.t.\;  \frac1W \im\sum_i\big[Q ^{\b g}_e(t, q)-z\big]^{-1}_{ii} \notin [c, c ^{-1}] \right)\le N^{-D}.
\ee
 Notice that \eqref{Evden2} holds in $\wh{\b S}( e,  N; \omega)$, which is larger than the set $ {\b S}( e,  N; \omega)$ 
  used in \eqref{Evden}. But instead of a precise error estimate as in \eqref{Evden},  here \eqref{Evden2} only provides a rough bound.  
 
Let $ \xi_{k}^{\b g}(e,t, q)$ be the $k$-th eigenvalue
 of $Q ^{\b g}_e(t, q)$.    Then for any (small) $\omega>0$ and (large) $D$
\be\label{Evden3}\P\left(   \exists k, \ell: 
\xi_{k}^{\b g}(e,t, q), \xi_{\ell}^{\b g}(e,t, q)  \in [e-N^{-\omega}, e+N^{-\om}],    
\; \left|\xi_{k}^{\b g}(e,t, q)-\xi_{\ell}^{\b g}(e,t, q)\right|\le \frac {|\ell-k|}{ N^{1+\omega}}-N^{-1+\omega} \right)\le N^{-D}
\ee
  Notice the minus sign in front of $-N^{-1+\omega} $ so that the right hand side of the  last inequality 
is positive only when $|k-\ell| \ge N^{  2\omega}$.

\end{theorem}
 
\subsection{Local law for generalized Green's function. } 

To prove Theorem \ref{local law}, we start with a more general setting.  Let  $\wt H$ be an $N\times N$ 
 real symmetric random matrix with centered and independent  entries, up to symmetry. 
(Here we use a different notation since it is different from the $H$ of main part. 
 Moreover, this $\wt H$ is also different from the matrix defined in \eqref{deformG}.)  
Define
$$ 
\wt s_{ij}:= \E \wt H^2_{ij}, \quad { 1\le i,j\le N}.
$$ Assume that $\wt s_{ij}=\OO(N^{-1})$ and there exist $  s_{ij}$ such that for some $c>0$, 
\be\label{sijL}
\wt s_{ij}=(1+{ \OO(N^{-1/2-c})})  s_{ij}, 
\ee
 and 
 $$
s_{ij}=s_{ji}, \quad \sum_{i}s_{ij}=1.
$$
 Note that the row sums of the matrix of variances of $\wt H$ is not exactly 1 any more, so this
class of matrices $\wt H$ goes slightly beyond the concept of generalized Wigner matrices introduced in \cite{ErdYauYin2012Univ} but still remain in their perturbative regime.
A detailed analysis of the general case was given in \cite{AjaErdKru2015}.  \nc

As in \eqref{H}, we define   
\be\label{Hgdef}
\wt H^{\b g}=\wt H- \sum_i g_i \b e_i \b e_i^*, \quad 
   \wt H^{\b g} =  \begin{pmatrix} \wt A^{\b g}  & \wt B^* \cr \wt B & \wt D^{\b g} \end{pmatrix}, \quad \b g=(g_1, \dots, g_N)\in \R^N,
\ee
where $\wt A^{\b g}$ is a $W\times  W$ matrix. We define 
\be\label{defwQ}
\wt Q^{\b g}_e=\wt A^{\b g}  
 - \wt B^* \left(\wt D^{\b g}-e\right)^{-1}\wt B. 
 \ee
Clearly  $Q^{\b g}_e(t,q)$   defined in \eqref{defqbt} equals to $\wt Q^{\b g}_e(t,q)$ if we choose 
\be\label{WTYH}
    \wt H =  \wt H(t,q) =\nc \begin{pmatrix}    \sqrt t    A_1+ \sqrt {1-q} \nc A_2   &   B^* \cr  B &  D  \end{pmatrix}.
\ee
   We now prove the local law of  $\wt Q^{\b g}_e = \wt Q^{\b g}_e(t,q) \nc $  
    by going to the  large \nc matrix.   In the following everything depends on the parameters $t, q$
   but we will often omit this from the notation. \nc
   
       For any $\wt H$ 
   and complex parameters $z, z'\in \C$ \nc we define
\begin{equation}\label{band}
\wt G^{\b g } (z, z'):=\begin{pmatrix}
\wt A^{\b g}-z& \wt B^* \cr
\wt B & \wt D^{\b g}-z'
\end{pmatrix}^{-1}:= \left({ \wt H^{\b g}(z,z')}\right)^{-1}:=\left(\wt H^{\b g}-zJ-z'J'\right)^{-1}
\end{equation}
 with $J_{ij}=\delta_{ij} {\mathds{1}}(i\le W)$ and   $J'_{ij}=\delta_{ij}{\mathds{1}}(i> W)$.
  Clearly
 $$
 (\wt Q^{\b g}_e-z)^{-1}_{ij}=\wt  G^{\b g}(z, e)_{ij}, \quad 1\le i,j\le W
 $$
 Note that  $\wt G^{\b g } (z, z')$ is not a Green's function unless $z=z'$; we will  
  call it {\it generalized  Green function.}  In Lemma \ref{locQq} below we show
  that an analogue of the local law
  holds for $\wt G^{\b g } (z, z')$ in a sense that its  diagonal entries are well approximated
  by deterministic functions $M_i^{\b g}(z,z')$ and the off diagonal entries are small.
The functions $M_i^{\b g}$ are defined via a self-consistent equation in the following lemma. \nc

 \begin{lemma}\label{mI}
{ Recall $m(z)$ defined in \eqref{credit}.}  For $z\in \C$,  such that $\im z>0$, $| z|\le C$, and  $|z^2-4|\ge \kappa$, 
for some fixed  $C, \kappa>0$, we define  
 $$
A(z, \zeta):=\left\{z'\in \C: \im z'>0, \quad |z-z'|\le N^{-\zeta} \right\}\subset \C.
 $$
 For any $z'\in A(z, \zeta)$, $\|{\b g}\|_\infty \le N^{-\zeta}$, there is a unique 
solution   $M ^{\b g}_i(z, z' )$ to the equation 
\be\label{falv}
\frac1{M ^{\b g}_i(z, z')}=-(z'-z){\mathds{1}}_{i>W}-g_i-z- \sum_{j=1}^N   s_{ij}M ^{\b g}_j(z, z'), \quad { 1\le i\le N}
 \ee
with  the constraint
\be\label{falv2}
 \max_i \left|M ^{\b g}_i(z, z')-m(z)\right|=\OO(\log N)^{-1}.  
\ee  \nc
Furthermore, $ M ^{\b g}_i(z, z') $ is continuous w.r.t. to $z'$ and $\b g$, 
and it satisfies the following bound 
\be\label{falv2.5}
  \max_i \left|M ^{\b g}_i(z, z')-m(z)\right|=\OO(\log N)\left(|z-z'|+\|\b g\|_\infty \right),
\ee
in particular $M ^{\b g = 0}_i(z, z)= m(z)$.
 \end{lemma}

 This theorem in a very  general setup (without the restriction \eqref{sijL}) was proved in Lemma 4.4 of \cite{AjaErdKru2015/2}. In particular, it
showed the existence, uniqueness and stability for any small additive perturbation 
of the equation 
\be\label{unpert}
  \frac{1}{M_i} = -z - \sum_{j} s_{ij}M_j.
\ee
which has a unique solution $M_i=m(z)$ in the upper half plane.
In other words, the solution $\b M(\b d)$  of the perturbed equation
\be\label{pert}
  \frac{1}{M_i(\b d)} = -z - d_i - \sum_{j} s_{ij}M_j(\b d)
\ee
depends analytically on the vector $\b d$ for $\| \b d\|\leq c/\log N$.
(Thanks to \eqref{sijL}, here we need only the special case when the perturbation is around the semicircle, $M_i=m$,
this result was essentially contained in \cite{ErdYauYin2012Univ} although not stated explicitly.)
The necessary input  is a  bound on the norm 
\be\label{Sbound}
   \Big\| \frac{1}{1- m^2(z) S}\Big\|_{\ell^\infty\to\ell^\infty} \le C_\e\log N, \qquad \mbox{for} \; |z^2-4|\ge \e,
\ee
that was first  proven in \cite{ErdYauYin2012Univ}, see also part (ii) of Proposition A.2 in \cite{ErdKnoYauYin2013}.
The bound  \eqref{Sbound} requires a spectral gap above $-1$ in the spectrum of $S$ which is guaranteed by
Lemma A.1 from \cite{ErdYauYin2012Univ} under the condition \eqref{eqn:band2}. 
In fact, for our band matrices the $\log N$ factor in \eqref{Sbound} can be removed, see Lemma 2.11 in
\cite{AjaErdKru2015/2}.

\medskip

 \begin{lemma} \label{locQq}  Recall  $\wt G^{\b g}(z,z')$, the generalized Green's function 
 of $\wt H$ from \eqref{band}.  Let $\Omega$ be the subset of the probability space \nc
  such that  for any  two complex numbers 
$y,y'\in \C$ satisfying   $0\le \im y'\le \im y$ and $|y|$, $|y'|\le 3$,  we have  
\be\label{zuic}
\|\wt G^{\b g }(y, y')\|\le C(\im y)^{-1}.
\ee
 Suppose that $\P(\Omega )\ge 1-N^{-D}$ for any fixed $D>0$. 
Assume that  $\b g$,  $z$ and $z'$  satisfy 
$$
 \|{\b g}\|_\infty \le N^{-1/2}, \quad \quad |z^2-4|\ge\kappa, \quad N^{-1+\zeta}\le \im z\le \zeta^{-1} , \quad \quad \zeta, \kappa>0
$$
and
$$
|z-z'|\le N^{-\zeta},\quad \quad  0\le \im z'\le \im z.  
$$
Then for any small $\e>0$, we have 
 \be\label{sYY}
\max_{ij}\left| \wt G^{\b g }_{ij}( z, z')-M^{\b g}_i(z,z' )\delta_{ij} \right|  \le  (N\eta)^{-1/2} N^\e,
\qquad \eta =\im z, 
 \ee
 holds with probability { greater than }$1-N^{-D}$ for any fixed $D>0$. 
 \end{lemma}
  
 Note that both the condition \eqref{zuic}
and the estimate \eqref{sYY} are uniform in $\im y'$ and $\im z'$, respectively, in particular
\eqref{sYY} holds even if $z'$ is on the real axis. This is formulated more explicitly in 
the following:

\begin{corollary}
\label{locQq2}  In the setting  of  Lemma \ref{locQq}, we  assume
 $ \|{\b g}\|_\infty \le N^{-1/2}$ and pick  an $e\in \R$ with  $|e|\le 2-\kappa$ for some $\kappa>0$.   \nc
 Then we have
 \be\label{sYYA}
\max_{1\le i, j\le W}\; \sup_{ N^{-1+\zeta}\le \im z\le N^{- \zeta}} \; \sup_{E:|E-e|\le N^{-\zeta}}\left| \wt G^{\b g }_{ij}( z, e)-M^{\b g}_i(z,e)\delta_{ij} \right|   \le (N\eta)^{-1/2}N^{\e}, \quad z=E+i \eta  
 \ee 
  holds with probability $1-N^{-D}$ for any fixed $D>0$  and $\e,\zeta>0$. \nc
\end{corollary}
\nc
\begin{proof}[Proof]  From Lemma \ref{locQq}, we know \eqref{sYYA} holds for fixed $z$ and $e$. Hence we only need to prove that that they hold at same time for all $z=E+i\eta : $
$|E-e|\le N^{-\zeta}$ and $ N^{-1+\zeta}\le \eta\le N^{- \zeta}$. We choose an $N^{-10}$-grid in both parameter spaces
so the validity of \eqref{sYYA} can be simultaneously guaranteed for each element of this net. Since  in $\Omega$ we have
$\left|\partial_z \wt G_{ij}^{\,\b g}\right|\le \|\wt G^{\b g }\|^2\le \eta^{-2}\le N^2$,
 and the same bound holds for $\partial_e \wt G_{ij}^{\, \b g}$,  we can approximate $\wt G_{ij}^{\,\b g}(z,e)$
at a nearby grid point with very high accuracy. The same argument holds for $M^{\b g}_i(z,e)$ by the stability
of its defining equation. This proves Corollary~\ref{locQq2}.
\end{proof}

\begin{proof}[Proof of Lemma \ref{locQq}]   For the proof we proceed in three steps.\\

{\it Step 1:  We first consider the case  $z=z'$.}   By \eqref{falv2.5}, we only need to prove that for any small $\e>0$
\be\label{DYA}
 \max_{ij}\left| \wt G^{\b g }_{ij}( z,  z \nc)-m(z  )\delta_{ij} \right|\le     (N\eta)^{-1/2}N^{\e}  
\ee
holds with probability  { greater than } $1-N^{-D}$. To prove this estimate,  we claim that
there exists a set $\Xi$ so that  
$\P(\Xi)\ge 1-N^{-D}$ for any $D>0$, { and  in $\Xi$}
$$
\max_{ij}| \wt G^{\b g }_{ij} -m(z)\delta _{ij}|\le (\log N)^{-1}.
$$
 Furthermore an approximate \nc self-consistent equation for $\wt G^{\b g }_{ii}$ holds in $\Xi$;
 more precisely, we have 
$$
{\mathds{1}}_\Xi \left|\left(\wt G^{\b g }_{ii}\right)^{-1}-\wt H_{ii}-g_i+z+\sum_{j}\wt s_{ij}\wt G^{\b g }_{jj}\right|\le (N\eta)^{-1/2}N^{\e}  .
$$
 These facts were  shown in   \cite{ErdYauYin2012Rig} with $\b g=0$ and the same argument holds
to the letter including  
a small perturbation $\b g$. \nc
Since by assumption $\|\b g\|_\infty\le N^{-1/2}$, and $\wt s_{ij}=(1+{ \OO(N^{-1/2-c})})  s_{ij}$
   we   obtain 
$$
{\mathds{1}}_\Xi \left|\left(\wt G^{\b g }_{ii}\right)^{-1}-\wt H_{ii}+z+\sum_{j}s_{ij}\wt G^{\b   g  }_{jj}\right|\le (N\eta)^{-1/2}N^{\e}. 
$$
 Since $|\wt H_{ii}|\le N^{-1/2+\e}$ with very high probability, 
using the stability of the unperturbed self-consistent equation
as in \cite{ErdYauYin2012Rig}, \nc we obtain \eqref{DYA}.

   \medskip
  
{\it Step 2:  Proof of  \eqref{sYY} for  $z'\ne z$.}   Clearly the $\wt G^{\b g }(z,z')-\wt G^{\b g }(z,z)$ is a continuous function w.r.t.  $z'$ and it equals to zero at $z=z'$.  We define the following  interpolation  between 
$y(0)=z$ to  $y(1)=\re z'+\ii\im z$ and then to  $y(2)=z'$: 
\begin{eqnarray*}
y(s)=
\begin{cases}
(1-s) \re z+s\re z'+\ii \im z   & 0\le s\le 1 \\
\re z'+(2-s)\ii \im z + (s-1) \ii \im z' & 1\le s\le 2. \\
\end{cases}
\end{eqnarray*}
Denote by  $s_k=kN^{-4}$ and $y_k=y(s_k)$,  and our goal is to  prove that 
\eqref{sYY} holds for $z'=y_{k}$ for  $k=2N^4$.  We have proved in Step 1 that  \eqref{sYY} holds for $z'=y_{k=0}$ and we now  apply  induction.    For any fixed $ \alpha <1/2$, we define the event  $\Xi^{(\al)}_k \subset\nc \Omega$
as
$$
\Xi^{(\al)}_k := \Omega \cap \Big \{  \max_{ij} \left| \wt G^{\b g }_{ij}( z, y_k)-M^{\b g }_i(z,y_k)\delta_{ij} \right|  
\le  (N\eta)^{-\al} \Big \} .
$$
Now we claim  that,   for any $1\le k\le 2N^4$,   any small  ${ \e>0}$  and any large  $D$,  we have  
 \be\label{Ygi}
\mathbb P\Big\{ \Big(\bigcap_{\ell\le k}\Xi^{(1/4)}_{ \ell \nc}\Big)   \setminus \Xi^{  (1/2- \e)}_{k} \Big\}\le N^{-D}.
\ee
Assuming this estimate is proved, we continue to prove  \eqref{sYY}. 
Recall the bound  
\be\label{xi01}
\P(\Xi_0^{(1/2 - \e)})\ge 1-N^{-D}
\ee from Step 1. 
Simple calculus  and \eqref{zuic} yield  that  
$$
   | \nc \partial_{z'} \wt G_{ij}^{\,\b g}  |\nc  \le \|\wt G^{\b g }\|^2\le N^2
$$
holds in the set $\Omega$. Hence we can estimate  the difference between $\wt G^{\b g }_{ij}( z, y_{k+1})$ and 
$\wt G^{\b g }_{ij}( z, y_{k})$ by $N^{-2}$. Similar estimate holds between $M^{\b g }_i(z,y_k)$ 
and $M^{\b g }_i(z,y_{k+1})$ by the stability of the self-consistent equation \eqref{falv}
at the parameter $(z, y_k)$, provided by  Lemma 4.4 of \cite{AjaErdKru2015/2}.

  \nc   These bounds easily  imply that 
 \be\label{Ygi2}
 \mathbb P\left( \;\Xi^{(  1/2- \e)}_{k} \setminus  \Xi^{(1/4)}_{k+1}\right)\le N^{-D}.
\ee
It is clear that { the initial bound \eqref{xi01} and } the two estimates \eqref{Ygi} and \eqref{Ygi2} allow us to use induction to conclude  $\P(\Xi^{(1/2-\e)}_{k})\ge  1- \nc N^{-D}$ for any $1\le k\le 2N^4$.  We have thus proved  \eqref{sYY} assuming  \eqref{Ygi}.

 \medskip
 
{\it Step 3. Proof of  \eqref{Ygi}.} 
 Recall $\wt G^{\b g }$ is defined with { $ H^{\b g} $}  in \eqref{band}. We define { $ H^{\b g, {(i)}}(z,z')$} as the matrix obtained by removing the $i$-th row and column of  { $ H^{\b g }(z,z')$} and set
 $$
 \wt G^{\b g, {(i)}}(z,z') : =\left(H^{\b g, {(i)}}(z,z')\right)^{-1}.
 $$
  As in \cite{ErdYauYin2012Rig}, the standard large deviation argument implies    that for any $\e>0$,  in $\Xi^{1/4}_{k }$,  
$$
\frac1{\wt G_{ii}^{\,\b g}(z,y_k)}= -(y_k-z) {\mathds{1}}\nc (i>W)-g_i-z- \sum_{ij} \wt s_{ij}\left(\wt G^{\b g, {(i)} }(z,y_k)\right)_{jj}+O  \left(N^{-1+\e}\|\wt G^{\b g, {(i)} }(z,y_k) \|_{HS}  \right)
$$
holds with probability $1-\OO(N^{-D})$, where $\| \,\cdot\, \|_{HS}$ is the Hilbert-Schmidt norm.  The  { matrix entries of}  $\wt G^{\b g, {(i)} }$ 
can be replaced by $\wt G^{\b g }$ by using the identity { (see \cite[Lemma 4.2]{ErdYauYin2012Univ})}
\be\label{Gijk}
(\wt G^{\b g , {(\ell)} }) _{ij}=\wt G^{\b g }_{ij}-
\frac{\wt G^{\b g }_{i\ell}\wt G^{\b g }_{\ell i}}{\wt G^{\b g }_{\ell\ell}}, 
\qquad  \ell \ne i, j,   \nc
\ee
 and using that both off-diagonal matrix elements are bounded by $(N\eta)^{-1/4}$ on $\Xi_k^{(1/4)}$. \nc
Together with \eqref{sijL}, we have obtained the self-consistent equation 
\be\label{sG}
\frac1{\wt G_{ii}^{\,\b g}(z,y_k)}=  -(y_k-z) {\mathds{1}}\nc(i>W)-g_i-z- \sum_{j}  s_{ij} \wt G_{jj}^{\,\b g}(z, y_k) +O  \left(N^{-1+\e}\|\wt G^{\b g }(z,y_k) \|_{HS} +(N\eta)^{-1/2} \right),
\ee
{ which holds with probability larger than $1-\OO(N^{-D})$. } The standard argument then uses the so-called Ward identity that the Green function $G=(H-z)^{-1}$  of any self-adjoint matrix  $H$
satisfies that 
\be\label{ward}
\| G (z)\|_{HS}^2 = \eta^{-1} \im \tr G(z), \qquad \eta =\im z.
\ee
In our case, $\wt G^{\b g }$ is not a Green function and this presents the major difficulty. The main idea is to write 
$$
\wt G^{\b g }(z,y_k)= \wt G^{\b g }(z,\wt y_k) + \wt G ^{\,\b g} (z,y _k)  \ii (\eta-\im y_k)J  \wt G^{\b g } (z, \wt y_k) , \quad \wt y_k=y_k+\ii (\eta-\im y_k),
$$
where $J$ is the matrix defined by  $J_{ij}=1_{1 \le i\le  W} {\delta_{ij}} $ and  the imaginary part of $\wt y_k$
 equals \nc $\eta = \im z$.  In particular, $\wt G^{\b g }(z,\wt y_k)$ is a Green function
of a self-adjoint matrix, hence the Ward identity is applicable. \nc
 By definition,  $\wt y_k\in  \{y_\ell \, : \; \ell \le k\}$.  \nc
Hence in the set $\bigcap_{\ell\le k}\Xi^{(1/4)}_{ \ell \nc}
\subset \Omega $,   we have 
\begin{align} \label{YLL}
 \|\wt G^{\b g }(z,y_k)\|_{HS} & \le \|\wt G^{\b g }(z,\wt y_k)\|_{HS}
+ \| \wt G ^{\,\b g} (z,y _k)  \ii (\eta-\im y_k)J  \wt G^{\b g } (z,  \wt y_k) \|_{HS}  \\
& \le \|\wt G^{\b g }(z,\wt y_k)\|_{HS}
+ \eta  \| \wt G ^{\,\b g} (z,y _k)\|    \|  \wt G^{\b g } (z,  \wt y_k) \|_{HS}  \le C
\Big[\eta^{-1} \im \tr \wt G^{\b g }(z,\wt y_k) \Big]^{ 1/2\nc},
\end{align}  \nc
where  we have used the Ward identity \eqref{ward} { for $\wt G^{\b g } (z,  \wt y_k)$} and  \eqref{zuic} { for $\wt G^{\b g } (z,   y_k)$} .  
Inserting this bound into \eqref{sG}, we have   that in $\bigcap_{\ell\le k}\Xi^{(1/4)}_{ \ell \nc} $ with probability $1-\OO(N^{-D})$ that { for any $\e>0$, }
$$
\frac1{\wt G_{ii}^{\,\b g}(z,y_k)}=  -(y_k-z) {\mathds{1}}\nc(i>W)-g_i-z- \sum_{ij}  s_{ij} \wt G_{jj}^{\,\b g}(z, y_k) +O   \left((N\eta)^{-1/2} N^{\e}\right).
$$
  Now we compare this equation with  \eqref{falv} and notice that both are 
perturbations of the  equation \eqref{unpert} that is stable in an $O( (\log N)^{-1})$ neighborhood
of the vector $\b m$. We obtain that 
 in $\bigcap_{\ell\le k}\Xi^{(1/4)}_{ \ell \nc} $  with probability $1-\OO(N^{-D})$
$$\max_i\left| \wt G_{ii}^{\,\b g}(z, y_k) -M^{\,\b g}_i(z, y_k)\right|=O   \left((N\eta)^{-1/2} N^{\e}\right).
$$  
 
For the  off-diagonal terms i.e., $\wt G_{ij}^{\,\b g}(z, y_k)$,  { similarly} as in  \cite{ErdYauYin2012Univ}, we know that  in $\Xi_k^{(1/4)}$ with probability $1-\OO(N^{-D})$
$$
\left|\wt G_{ij}^{\,\b g}(z, y_k)\right|\le \left|\wt G_{ij}^{\,\b g}(z, y_k)\right|\left|\wt G_{jj}^{\,\b g, {(i)}}(z, y_k)\right| \left(N^{-1/2+\e}+N^{-1+\e}\|(\wt G^{\b g, {(ij)} }(z,y_k)) \|_{HS} \right)
$$
Then with \eqref{YLL} and \eqref{Gijk}, we obtain that in {   $\bigcap_{\ell\le k}\Xi^{(1/4)}_{ \ell \nc}$} with probability $1-\OO(N^{-D})$
$$
\left|\wt G_{ij}^{\,\b g}(z, y_k)\right|=O   \left((N\eta)^{-1/2} N^{\e}\right) .
$$
This completes  the proof of \eqref{Ygi} and Lemma  \ref{locQq}.
\end{proof}

\subsection{Operator bound of $G(z,z')$. } 
As explained in the beginning of this section, we are going to prove Theorem \ref{local law} with Corollary~\ref{locQq2}.  For this purpose, we need to prove that band matrix satisfies the assumption \eqref{zuic}.  In this subsection, we prove a sufficient condition for \eqref{zuic}. 
 We formulate the result  in a non-random setup and later we will check that 
the conditions hold with very high probability in case of our random band matrix. \nc
 
 \begin{lemma}\label{lem: obG}
 Let $H$ be a   (non random)   symmetric $N\times N$ matrix  and consider its block decomposition as \nc
 $$
H=  \begin{pmatrix} A  & B^* \cr B & D \end{pmatrix}.
 $$
 Suppose that for (small) $\mu >0$ and $C_0$, the following holds:
\begin{enumerate}[(i)]
\item there does not exist $e \in \bR, \;\;  \b u\in  \R^N$ such that
\be\label{hlan}
   \|\bu\|=1, \;  \| B^*\bu\|\le \mu, \; \| ( D-e) \bu\|\le   \mu.
\ee
\item The submatrices are  bounded:
\be\label{apriori3}\|A\|+\|B\|+\|D\| \le C_0.
\ee
\end{enumerate}
  Define
$$
G(z,z'):= \begin{pmatrix} A-z  & B^* \cr B & D-z' \end{pmatrix}^{-1}.
$$ 
Then for any large $C''>0$, there exists $C'>0$,  depending only on $C''$, $\mu$ and $C_0$, \nc such that  if 
$$
z,z'\in \C, \quad 0\le\im z'\le \im z, \quad |z|+|z'|\le C'',
$$
then we have 
\be
\label{xHUI}
 \|G(z,z')\|   \le \nc   \frac{C'}{\im z}.
 \ee

 \end{lemma}

\begin{proof}[Proof of Lemma \ref{lem: obG}]
 Define the symmetric matrix
$$
   P: = \begin{pmatrix} A-\re z  & B^* \cr B & D{ -\re z'}  \end{pmatrix},
$$
then by resolvent identity we have   
$$
G=G(z,z')=\frac{1}{P-\ii\im z}+\frac{1}{P-\ii\im z}(\im z'-\im z)\ii JG,
$$
 where the matrix $J$ was already  defined by  $J_{ij}=1_{1 \le i\le  W} {\delta_{ij}} $. 
Here $W$ is the size of the block $A$. \nc
Then 
$$
\|G\|\le (\im z)^{-1}+ \frac{\im z -\im z'}{\im z}\|G\|,
$$
which implies $\|G\|\le (\im z')^{-1}$.   Furthermore, with 
 [I removed a prime from $J$, I think $J'$ was something obsolete] \nc  
 $$
\partial_{z'}G  =GJG, 
$$
it is easy to see  by integrating $\partial_{z'}G$ from $z'=e'$ to $z'=e'+i\eta'$  \nc
that we only need to prove  \eqref{xHUI} for the case $\im z'=0$. Hence from now on, we assume that 
$$
z=e+\ii\eta, \quad z'=e'. 
$$
 Applying  Schur formula, we obtain
$$
   G= G(z,z')=\begin{pmatrix}  A-z & B^* \cr
   B & D-z'
\end{pmatrix}^{-1} = \begin{pmatrix}  
\frac{1}{A-z- B^*(D-z')^{-1}B} &  -\frac{1}{A-z}B^* \frac{1}{D-z' - 
  B(A-z)^{-1}B^*} \cr
 -\frac{1}{ D-z' - B(A-z)^{-1} B^*}B\frac{1}{A-z} & 
   \frac{1}{D-z' - B(A-z)^{-1}B^*}
\end{pmatrix} . 
$$
First with $\im z'=0$, we have  the trivial bounds (which  follows by  $\|(P+i\eta )^{-1}\|\le \eta^{-1}$ for any 
symmetric matrix $P$):
\be\label{ABC1}
\Big\| \left(A-z -   B^*\frac{1}{D-z'}B\right)^{-1}\Big\| \le \frac{1 }{|\im z|},
  \qquad \Big\| \left(A-z\right)^{-1}\Big\| \le \frac{1}{|\im z|},
\ee
 which controls the upper left corner of $G$. \nc
Second, we claim that  
\be\label{ABC2}
\|\frac{1}{A-z}B^* \frac{1}{D-z' - 
  B\frac{1}{A-z}B^*}\|^2\le  \frac{1}{|\im z|}  
\Big\|  \frac{1}{D-e'   -
  B\frac{1}{A-z}B^*}  \Big\|.
  \ee
 For the proof,  picking any nonzero vector $\bv$ and  setting \nc $\bu=(D-e'  -  B\frac{1}{A-z}B^*)^{-1}\bv$, we have  
\begin{align*}
  \|\bv\| & =\Big\|
\Big(D-e'  -  B\frac{1}{A-z}B^*\Big)\bu  \Big\| \ge  \frac{1}{\| \bu\|}
\Big|\Big\langle \bu, \Big(D-e' -  B\frac{1}{A-z}B^*\Big)\bu\Big\rangle \Big|
\\
  &=  \frac{1}{\| \bu\|}
\Big|\Big\langle \bu, \Big(D-e' - B\frac{A-e}{(A-e)^2+\eta^2}B^*\Big)\bu 
  \Big\rangle - i \Big\langle \bu, \Big( 
 B\frac{\eta}{(A-e)^2+\eta^2}B^*\Big)\bu 
  \Big\rangle 
 \Big| 
\\
& \ge \frac{\eta}{\|\bu\|}  \Big\langle u, 
 B\frac{1}{(A-e)^2+\eta^2}B^*\bu  \Big\rangle  
  =  \frac{\eta}{\|\bu\|}  \Big\langle u, 
 B\frac{1}{|A-z|^2}B^* \bu  \Big\rangle  
 \\
 & =  \frac{\eta}{\|\bu\|}
 \Big\|  \frac{1}{A-z} B^* \bu  \Big\|^2.
 \end{align*}
Changing the vector $\bu$ back to $\bv$, we have
$$
 \Big\| \frac{1}{A-z} B^*  \frac{1}{D-e'   -
  B\frac{1}{A-z}B^*} \bv\Big\|^2  \le  \frac{1}{|\im z|}  
\Big\|  \frac{1}{D-e'   -
  B\frac{1}{A-z}B^*} \bv \Big\| \|\bv\|, 
$$
which implies \eqref{ABC2}. Now it only remains to bound $\Big\| \Big(D-e'   -
  B\frac{1}{A-z}B^*\Big)^{-1} \Big\|$  by $C/\eta$,  which would then control all other three blocks of $G$. \nc
    Suppose for some normalized vector $\bu$ and small $\wt \mu>0$, we have
\be\label{dhp}
 \Big\| (D-e')\bu - B\frac{1}{A-z}B^*\bu\Big\| 
\le \wt  \mu \eta.
\ee
Then
$$
\wt \mu \eta\ge 
\Big| \im \Big\langle \bu, (D-e' )\bu + B\frac{1}{A-z}B^*\bu \Big\rangle\Big|
=  \Big\langle \bu, \Big(   B\frac{\eta}{(A-e)^2+\eta^2}B^*\Big) \bu \Big\rangle.
$$
Then for some $C_1>0$, we have
\be\label{secc}
 \wt  \mu \ge \Big\langle \bu,  B\frac{1}{|A-z|^2}B^* \bu \Big\rangle \ge \frac{1}{C_1}
   \| B^*\bu\|^2
\ee
where we used that the fact $|A-z|^2 $ is bounded.   This shows that
\be\label{need3}
   \| B^*\bu \| \le \sqrt{C_1\wt \mu}, \qquad \| BB^*\bu\|\le \sqrt{  C_0C_1\wt \mu}
\ee
  by  \eqref{apriori3}.
From \eqref{secc}, we also have
$$
    \Big\| B\frac{1}{A-z}B^*\bu\Big\|^2 = 
  \Big\langle \bu,  B\frac{1}{A-    \bar z}B^* B\frac{1}{A-z}B^*\bu \Big\rangle
  \le  C_0\Big\langle \bu,  B\frac{1}{|A-z|^2}B^*\bu \Big\rangle \le C_0\wt \mu.
$$
Then with \eqref{dhp}, for small enough $\wt \mu$, we have 
\be\label{need2}
  \| (D-e')\bu\| \le  
   \Big\| B\frac{1}{A-z}B^*\bu\Big\| +\wt \mu \eta 
  \le  \sqrt{C_0\wt \mu }+\wt \mu \eta \le C\sqrt{ \wt \mu}.
\ee
Combining  \eqref{need3},  \eqref{need2} and \eqref{hlan}, we obtain \eqref{dhp} does not hold for small enough { $\wt \mu$. } Together with \eqref{ABC1} and \eqref{ABC2}, we completed the proof of Lemma \ref{lem: obG}. 
\end{proof}

 \subsection{Proof of Theorem \ref{local law}.  } 
  Now we return to prove Theorem \ref{local law}, i.e., the local law of the Green's function of some particular matrices which are derived from band matrix.  
 
\begin{proof}[Proof of \eqref{Evden}] As explained in \eqref{WTYH}, we know that $Q^{\b g}_e (t, q)$ is a
matrix of the form in \eqref{defwQ}.  We will apply Lemma~\ref{basic lem} with $M=W$ and $L=N-W$. 
Since $H$ is a band matrix with band width $4W-1$, see \eqref{eqn:band2},  the upper $W\times W$
block of $B$ has variance $(4W-1)^{-1}$, i.e.
  $$
s_{ij}=\E B_{ij}^2  = \frac{1}{4W-1} \nc , \quad 1\le i,j\le W.
 $$
Using this information in \eqref{Bu2} to estimate $\sum_{1\le i\le W} |u_i|^2$ from below and inserting 
this into the last condition in \eqref{brownBand}, \nc we learn that  for some small $\mu>0$ we have  
 $$
 \P\left(\exists e\in \R, \; \exists \bu\in \R^{N-W}\; : \;  \|\bu\|=1, \;  \| B^*\bu\|\le \mu, \; \| ( D^{\b g}-e) \bu\|\le \mu   \right)\le N^{-D}.
 $$
  We also know that $\| H\|$, hence $\| A\|$, $\| B\|$ and $\| D\|$ are all bounded
 by a large constant with very high probability.  \nc
 Then using  Lemma \ref{lem: obG}, we obtain that for some large $C>0$, we have 
 $$
  \P\left(\exists z,z'\in \C: |z|, |z'|\le 3, \; \im z\ge \im z'\ge 0,   \|\wt G^{\b g}(z,z')\|\ge C(\im z)^{-1}\right)\le N^{-D} .
  $$
 With this bound, we can use   Corollary~\ref{locQq2}. Together with \eqref{falv2.5}, we complete the proof of \eqref{Evden}.
 \end{proof}
 
\begin{proof}[Proof of \eqref{Evden2}] Because of \eqref{Evden}, it only remains to prove \eqref{Evden2} for 
\be\label{cyL}
 |E-e|\le N^{-\omega}, \quad  { N^{-\omega}\le \eta\le 1}.
 \ee
 Recall  $ \xi_{k}^{\b g}(e,t, q)$, $1\le k\le  W\nc$ is the $k$-th eigenvalue of $Q^{\b g}_e(t,q)$. Then
 $$
 \im\sum_j\big[Q ^{\b g}_e(t, q)-z\big]^{-1}_{jj}=  \im \tr \frac{1}{Q ^{\b g}_e(t, q)-E-i\eta} =\nc
 \sum_{k}\frac{\eta}{|\xi_{k}^{\b g}(e,t, q)-E|^2+\eta^2}, \quad z=E+i\eta.
 $$
 In our case \eqref{cyL}, we know
 $$
 |\xi_{k}^{\b g}(e,t, q)-E|^2+\eta^2\sim  |\xi_{k}^{\b g}(e,t, q)-e|^2+\eta^2.
 $$
 
Therefore, we only need to prove that there exists $c>0$ such that 
$$
\mathbb P\left(\exists \eta,\; N^{-\omega}\le \eta\le 1   \; s.t.\;  \frac1W \im \tr \frac{1}{Q ^{\b g}_e(t, q)-e-i\eta}\notin [c, c ^{-1}] \right)\le N^{-D}.
$$
 After adjusting the constant $c$,  \nc
it will be implied by the following  high probability bound on the eigenvalue density: \nc 
\be\label{MIA2}
\mathbb P\left(\exists\eta, \;  N^{-\omega}\le \eta\le 1   \; s.t.\;   (N\eta)^{-1} \nc
 \#\big\{k: \xi_{k}^{\b g}(e,t, q)\in [e-\eta, e+\eta]\big\}\notin [c, c ^{-1}] \right)\le N^{-D}.
\ee
 From Section~\ref{sec:MF} recall the definition of curves $\cal C^{\b g}_k (e)$ constructed from
the matrix \eqref{meq1}.  Similarly,  starting with the matrix $\wt H^{\b g}$, see \eqref{Hgdef}
and \eqref{WTYH}, \nc
we can define the curves $e\to \cal C^{\b g}_k (e, t, q)$
 for any fixed parameters $t$, $q$. \nc As in Lemma \ref{holo}, we have that for any $K$, there exists $C_K$
such that 
     \be\label{evderTE}
\P\left(  \sup_{   e \not \in \sigma({ D^{\b g}})  }\sup_k {\mathds{1}}( | \cal C^{\b g}_k(e, t, q)| \le K ) 
   \Big| \frac{\rd  \cal C^{\b g}_k}{\rd e} (e, t, q) \Big|    \le   C_K \right)\le N^{-D}.
   \ee
It means the slopes of these curves are bounded  in $[-K,K]^2$.  The crossing points of these curves with $x=y$ line are exactly the points
$$
   (\lambda^{\b g}_k(t, q), \lambda^{\b g}_k(t, q)), \quad 1\le k\le N,
$$
   where $\lambda^{\b g}_k(t, q)$ is the $k$-th eigenvalue of $\wt H^{\b g}$.
   By simple perturbation theory and using $|t-q|\le N^{-1/2}$, $\|\b \g\|\le N^{-1/2}$, \nc 
  it is easy to see that with high probability, we have 
$$
|   \lambda_k  -\lambda^{\b g}_k(t, q)|\ll N^{-\omega}, \quad \lambda_k:=\lambda_k(q,q).
$$
  Note $\lambda_k(q,q)$  is the eigenvalue of a regular generalized Wigner matrix,
   i.e. $\wt H^{\b g=0}$ at $t=q$ has variances summing up exactly to one in each row. \nc
   Then together with the  rigidity of $\lambda_k$, we know 
$$
\mathbb P\left(\exists N^{-\omega}\le \eta\le 1   \; s.t.\;   (N\eta)^{-1} \nc \#\big\{k: \cal C_{k}^{\b g}(e,t, q)\in [e-\eta, e+\eta]\big\}\notin [c, c ^{-1}] \right)\le N^{-D}.
$$
With \eqref{evderTE},  (note $\frac{\rd  \cal C^{\b g}_k}{\rd e} \le 0$ as in \eqref{h1}) we obtain \eqref{MIA2} and complete the proof of 
\eqref{Evden2}. \end{proof}
 
\begin{proof}[Proof of \eqref{Evden3}] With \eqref{Evden2}, we know that 
$$
\P\Big(   \exists x, y   \in [e-N^{-\om}, e+N^{-\om}],  \;|x-y|\ge N^{-1+\omega}, \;
\;  N^{-1}\nc  \#\big\{k: \xi_{k}^{\b g}(e,t, q)\in [x, y]\big\}\ge |x-y| \log N \Big)\le N^{-D}.
$$
It is easy to see that it implies \eqref{Evden3}, which completes the proof of Theorem \ref{local law}. 
\end{proof}


\begin{bibdiv}
\begin{biblist}





\bib{AjaErdKru2015}{article}{
   author={Ajanki, O.},
   author={Erd{\H o}s, L.},
   author={Kruger, T.},
   title={Universality for general Wigner-type matrices},
   journal={prepublication, arXiv:1506.05098},
   date={2015}
}

\bib{AjaErdKru2015/2}{article}{
   author={Ajanki, O.},
   author={Erd{\H o}s, L.},
   author={Kruger, T.},
   title={Quadratic vector equations on complex upper half plane},
   journal={prepublication, arXiv:1506.05095},
   date={2015}
}



\bib{AnaLeM2013}{article}{
   author={Anantharaman, N.},
   author={Le Masson, E.},
   title={Quantum ergodicity on large regular graphs},
   journal={Duke Math. J.},
   date={2015},
   volume={164},
   number={4},
   pages={723--765}
}



   
 \bib{BaoErd2015}{article}{
   author={Bao, Z.},
   author={Erd{\H o}s, L.},,
   title={Delocalization for a class of random block band matrices},
   journal={to appear in Probab. Theory Related Fields},
   date={2016}}   



   

 \bib{BouErdYauYin2015}{article}{
   author={Bourgade, P.},
   author={Erd{\H o}s, L.},
   author={Yau, H.-T.},
   author={Yin, J.},
   title={Fixed energy universality for generalizd Wigner matrices},
   journal={to appear in Communications on Pure and Applied Mathematics},
   date={2016}}
      
      
   
 \bib{BouHuaYau2016}{article}{
   author={Bourgade, P.},
   author={Huang, J.},   
   author={Yau, H.-T.},
   title={Eigenvector statistics of sparse random matrices},
   journal={},
   date={2016}}
   
   

 \bib{BouYau2013}{article}{
   author={Bourgade, P.},
   author={Yau, H.-T.},
   title={The Eigenvector Moment Flow and local Quantum Unique Ergodicity},
   journal={to appear in Commun. Math. Phys.},
   date={2016}}
   
   
   

\bib{Col1985}{article}{
   author={Colin de Verdi{\`e}re, Y.},
   title={Ergodicit\'e et fonctions propres du laplacien},
   language={French, with English summary},
   journal={Comm. Math. Phys.},
   volume={102},
   date={1985},
   number={3},
   pages={497--502}
}

   
   
   \bib{Colin}{book}{
   author={Colin de Verdi{\`e}re, Y.},
   title={Spectres de graphes},
   language={French, with English and French summaries},
   series={Cours Sp\'ecialis\'es [Specialized Courses]},
   volume={4},
   publisher={Soci\'et\'e Math\'ematique de France, Paris},
   date={1998},
   pages={viii+114}
}

   


 \bib{DisPinSpe2002}{article}{
   author={Disertori, M.},
   author={Pinson, L.},
   author={Spencer, T.},
   title={Density of states for random band matrices},
   journal={Commun. Math. Phys.},
   volume={232},
   pages={83--124},
   date={2002}}


\bib{Efe1997}{article}{
   author={Efetov, K.},
   title={Supersymmetry in disorder and chaos},
   journal={Cambridge University Press},
   date={1997}
}



\bib{ErdKnoYauYin2013}{article}{
   author={Erd{\H{o}}s, L.},
   author={Knowles, A.},
   author={Yau, H.-T.},
   author={Yin, J.},
   title={The local semicircle law for a general class of random matrices},
   journal={Elect. J. Prob.},
   volume={18},
   date={2013},
   number={59},
   pages={1--58}
}


\bib{ErdKno2013}{article}{
   author={Erd{\H{o}}s, L.},
   author={Knowles, A.},
   title={Quantum Diffusion and Delocalization for Band Matrices  with General Distribution},
   journal={Ann. Inst. H. Poincar\'e},
   volume={12},
   date={2011},
   number={7},
   pages={1227-1319}
}



\bib{ErdPecRamSchYau2010}{article}{
   author={Erd{\H{o}}s, L.},
   author={P{\'e}ch{\'e}, S.},
   author={Ram{\'{\i}}rez, J. A.},
   author={Schlein, B.},
   author={Yau, H.-T.},
   title={Bulk universality for Wigner matrices},
   journal={Comm. Pure Appl. Math.},
   volume={63},
   date={2010},
   number={7},
   pages={895--925}
}





\bib{ESY1}{article}{
      author={Erd{\H{o}}s, L.},
      author={Schlein, B.},
      author={Yau, H.-T.},
       title={Semicircle law on short scales and delocalization of eigenvectors for Wigner random matrices},
        date={2009},
     journal={Annals of Probability},
      volume={37},
       pages={815\ndash852},
}



\bib{ErdSchYau2011}{article}{
      author={Erd{\H{o}}s, L.},
      author={Schlein, B.},
      author={Yau, H.-T.},
       title={Universality of random matrices and local relaxation flow},
        date={2011},
     journal={Invent. Math.},
      volume={185},
      number={1},
       pages={75\ndash 119},
}



\bib{ES}{article}{
      author={Erd{\H{o}}s, L.},
      author={Schnelli, K.},
      title={Universality for Random Matrix Flows with Time-dependent Density},
      journal={preprint, arXiv:1504.00650},
      date={2015},
}










\bib{ErdYau2012}{article}{
      author={Erd{\H{o}}s, L.},
      author={Yau, H.-T.},
       title={Universality of local spectral statistics of random matrices},
        date={2012},
     journal={Bull. Amer. Math. Soc. (N.S.)},
      volume={49},
      number={3},
       pages={377\ndash 414},
}

\bib{ErdYau2012singlegap}{article}{
      author={Erd{\H{o}}s, L.},
      author={Yau, H.-T.},
       title={Gap universality of generalized Wigner and beta ensembles},
        date={2015},
     journal={ J. Eur. Math. Soc. },
     volume={17},
     pages={1927\ndash 2036}
}



\bib{ErdYauYin2012Univ}{article}{
      author={Erd{\H{o}}s, L.},
      author={Yau, H.-T.},
      author={Yin, J.},
       title={Bulk universality for generalized {W}igner matrices},
        date={2012},
     journal={Probab. Theory Related Fields},
      volume={154},
      number={1-2},
       pages={341\ndash 407},
}

\bib{ErdYauYin2012Rig}{article}{
      author={Erd{\H{o}}s, L.},
      author={Yau, H.-T.},
      author={Yin, J.},
       title={Rigidity of eigenvalues of generalized {W}igner matrices},
        date={2012},
     journal={Adv. Math.},
      volume={229},
      number={3},
       pages={1435\ndash 1515},
}

\bib{fy}{article}{
      author={Fyodorov, Y.V.},
      author={Mirlin, A.D.},
       title={ Scaling properties of localization in random band matrices: A $\sigma$-model approach.},
        date={1991},
     journal={Phys. Rev. Lett.},
      volume={67 },
       pages={2405\ndash 2409},
}


\bib{Hol2010}{article}{
   author={Holowinsky, R.},
   title={Sieving for mass equidistribution},
   journal={Ann. of Math. (2)},
   volume={172},
   date={2010},
   number={2},
   pages={1499--1516}
}

\bib{HolSou2010}{article}{
   author={Holowinsky, R.},
   author={Soundararajan, K.},
   title={Mass equidistribution for Hecke eigenforms},
   journal={Ann. of Math. (2)},
   volume={172},
   date={2010},
   number={2},
   pages={1517--1528}
}

  



\bib{Joh2001}{article}{
   author={Johansson, K.},
   title={Universality of the local spacing distribution in certain
   ensembles of Hermitian Wigner matrices},
   journal={Comm. Math. Phys.},
   volume={215},
   date={2001},
   number={3},
   pages={683--705}
}


\bib{KnoYin2013}{article}{
     author={Knowles, A.},
    author={Yin, J.},
   title={Eigenvector distribution of Wigner matrices},
   journal={Probability Theory and Related Fields},
   volume={155},
   date={2013},
   number={3},
   pages={543--582}
}




\bib{LY}{article}{
   author={Landon, B.},
   author={Yau, H.-T.},
   title={Convergence of local statistics of Dyson Brownian motion},
   journal={preprint, arXiv:1504.03605},
   date={2015}
}



\bib{LeeSchSteYau2015}{article}{
   author={Lee, J.-O.},
   author={Schnelli, K.},
   author={Stetler, B.},
   author={Yau, H.-T},
   title={Bulk universality for deformed Wigner matrices},
   journal={to appear in Annals of Probability},
   date={2015}
}





\bib{Lin2006}{article}{
   author={Lindenstrauss, E.},
   title={Invariant measures and arithmetic quantum unique ergodicity},
   journal={Ann. of Math. (2)},
   volume={163},
   date={2006},
   number={1},
   pages={165--219}
}







\bib{RudSar1994}{article}{
   author={Rudnick, Z.},
   author={Sarnak, P.},
   title={The behaviour of eigenstates of arithmetic hyperbolic manifolds},
   journal={Comm. Math. Phys.},
   volume={161},
   date={1994},
   number={1},
   pages={195--213}
}


\bib{Sch2009}{article}{
   author={Schenker, J.},
      title={Eigenvector localization for random band matrices with power law band width},
   journal={Comm. Math. Phys.},
   volume={290},
   date={2009},
   pages={1065--1097}}


\bib{Sch2014}{article}{
   author={Shcherbina, T.},
      title={Universality of the local regime for the block band matrices with a finite number of blocks},
   journal={J. Stat. Phys.},
   volume={155},
   date={2014},
   pages={466--499}}

   
      \bib{Sch1}{article}{
   author={Shcherbina, T.},
      title={On the Second Mixed Moment of the Characteristic Polynomials of 1D Band Matrices},
   journal={Communications in Mathematical Physics},
   volume={328},
   date={2014},
   pages={45--82}}
   
   \bib{Sch2}{article}{
   author={Shcherbina, T.},
      title={Universality of the second mixed moment of the characteristic polynomials of the 1D band matrices: Real symmetric case},
   journal={J. Math. Phys.},
   volume={56},
   date={2015}}
   

\bib{Shn1974}{article}{
      author={Shnirel'man, A. I.},
        date={1974},
     journal={Uspekhi Mat. Nauk},
      volume={29},
      number={6},
       pages={181\ndash 182},
}


\bib{Sod2010}{article}{
  author={Sodin, S.},
      title={The spectral edge of some random band matrices},
   journal={ Ann. of Math.},
   volume={173},
   number={3},
   pages={2223-2251},
   year={2010}
}

   
\bib{Spe}{article}{
   author={Spencer, T.},
      title={Random banded and sparse matrices (Chapter 23)},
   journal={Oxford Handbook of Random Matrix Theory, edited by G. Akemann, J. Baik, and P. Di Francesco},
   }
    



\bib{TaoVu2011}{article}{
   author={Tao, T.},
   author={Vu, V.},
   title={Random matrices: universality of local eigenvalue statistics},
   journal={Acta Math.},
   volume={206},
   date={2011},
   number={1}
}


   


\bib{Zel1987}{article}{
   author={Zelditch, S.},
   title={Uniform distribution of eigenfunctions on compact hyperbolic
   surfaces},
   journal={Duke Math. J.},
   volume={55},
   date={1987},
   number={4},
   pages={919--941}
}

   
   

\end{biblist}
\end{bibdiv}
\end{document}